\documentclass[a4paper]{amsart}

\RequirePackage{amsmath} 
\RequirePackage{amssymb}
\usepackage{amscd,latexsym,amsthm,amsfonts,amssymb,amsmath,amsxtra}
\usepackage[colorlinks=true,urlcolor=blue,citecolor=blue]{hyperref}
\usepackage{color}
\usepackage[all]{xy}
\usepackage{bm}
\usepackage{mathtools}
\usepackage{mathrsfs}
\usepackage{xcolor}
\usepackage{comment}
\usepackage{marginnote}
\usepackage{enumitem}
\usepackage{thmtools, thm-restate}

\let\Re\undefined

\DeclareMathOperator{\Re}{Re}

\DeclareMathOperator{\Tr}{Tr}

\DeclareMathOperator{\GL}{GL}

\DeclareMathOperator{\Spec}{Spec}
\DeclareMathOperator{\Hom}{Hom}

\newcommand{\Res}{\operatorname{Res}}

\newcommand{\fin}{\operatorname{fin}}

\newcommand{\diag}{\operatorname{diag}}

\newcommand{\Vol}{\operatorname{Vol}}

\newcommand{\Ind}{\operatorname{Ind}}

\newcommand{\RNum}[1]{\uppercase\expandafter{\romannumeral #1\relax}}

\begin{document}
\theoremstyle{plain}
\newtheorem{thm}{Theorem}[section]
	
\newtheorem{cor}[thm]{Corollary}

\newtheorem{thmx}{Theorem}
\renewcommand{\thethmx}{\Alph{thmx}} 

\newtheorem{hy}[thm]{Hypothesis}
\newtheorem*{thma}{Theorem A}
\newtheorem*{corb}{Corollary B}
\newtheorem*{thmc}{Theorem C}
\newtheorem{lemma}[thm]{Lemma}  
\newtheorem{prop}[thm]{Proposition}
\newtheorem{conj}[thm]{Conjecture}
\newtheorem{fact}[thm]{Fact}
\newtheorem{claim}[thm]{Claim}
	
\theoremstyle{definition}
\newtheorem{defn}[thm]{Definition}
\newtheorem{example}[thm]{Example}
\theoremstyle{remark}
	
\newtheorem{remark}[thm]{Remark}	
\numberwithin{equation}{section}

\title[]{Weyl Subconvexity for $\mathrm{GL}_2$ with Simple Supercuspidal Ramification} 
\author{Liyang Yang}
\address{Department of Mathematics, Texas A\&M University, College Station,  TX 77843, USA} 
\email{liyangy@tamu.edu}

\begin{abstract}
We establish Weyl-type subconvexity bounds in the level aspect for cuspidal automorphic representations of $\mathrm{GL}_2/F$ with simple supercuspidal ramification at a prime ideal $\mathfrak{q}$. More precisely, for the family
$\mathcal{F}_t^\zeta(\mathfrak{q}^3;\omega)$
consisting of representations of conductor $\mathfrak{q}^3$, central character $\omega$, and prescribed simple supercuspidal local component, we prove the fourth moment estimate
\begin{align*}
\sum_{\substack{\pi \in \mathcal{F}_{t}^{\zeta}(\mathfrak{q}^3;\omega) \\
C_v(\pi) \leq \mathbf{C}_v,\ v \mid \infty}}
|L(1/2,\pi)|^4
\ll_{F,\varepsilon}
\mathbf C_\infty^{1+\varepsilon}
N_F(\mathfrak{q})^{2+\varepsilon}.
\end{align*}
As a consequence, we deduce the Weyl-type bound
\begin{align*}
L(1/2,\pi)
\ll_{F,\varepsilon}
C_\infty(\pi)^{1/4+\varepsilon}
C_{\fin}(\pi)^{1/6+\varepsilon}.
\end{align*}
In particular, this bound applies to a genuinely \emph{non-self-dual} family of odd conductor exponent, beyond the reach of the cubic moment method.
\end{abstract}
	
\date{\today}%
\maketitle
\tableofcontents

\section{Introduction}

Let $F$ be a number field.
For a unitary cuspidal automorphic representation $\pi$ of
$\mathrm{GL}_n/F$ with analytic conductor $C(\pi)$, the
Phragm\'en--Lindel\"of principle yields the convexity bound
\begin{align*}
L(1/2,\pi) \ll C(\pi)^{1/4+\varepsilon}.
\end{align*}
The subconvexity problem asks for any improvement over the exponent
$1/4$. It is a central problem in the analytic theory of $L$-functions
and a major testing ground for modern analytic techniques.

The first subconvexity bound was obtained by Hardy--Littlewood and
Weyl, who established the exponent $1/6$ for the Riemann zeta function.
Estimates of the form
\begin{align*}
L(1/2,\pi) \ll C(\pi)^{1/6+\varepsilon}
\end{align*}
are therefore called \emph{Weyl-type bounds}. Such bounds represent a
natural barrier for many analytic approaches and are notoriously
difficult to achieve. At present they are known only in a limited
number of cases, notably for certain self-dual $\mathrm{GL}_2$
automorphic representations; see for instance
\cite{CI00, Ivi01, Jut01, You17, Nel19, PY20, PY23, BJN23, WX23, HPY25}.

In the level aspect, obtaining the Weyl bound is particularly delicate. All currently known results ultimately derive from cubic moment formulas for $\mathrm{GL}_2$ $L$-functions and rely on positivity arguments. In particular, these methods require the arithmetic conductor to be a perfect square and the central character to be trivial, ensuring the nonnegativity $L(1/2,\pi)\ge 0$. Consequently, Weyl-type bounds in the level aspect for conductors of odd exponent or for representations with nontrivial central character fall outside the scope of the cubic moment method.

In this paper we establish Weyl-type bounds in the level aspect for cuspidal automorphic representations of $\mathrm{GL}_2/F$ with \emph{simple supercuspidal ramification} (conductor exponent $3$), corresponding to minimal wild ramification, while allowing arbitrary central characters. This places the Weyl exponent in a regime beyond the reach of the cubic moment method.

Our approach substantially refines the symmetric spectral reciprocity framework
of \cite{Yan25} in a form specifically adapted to representations with simple
supercuspidal ramification. This refinement reveals a conductor-lowering
phenomenon: the fourth moment over representations of $\mathrm{GL}_2/F$ of
conductor $\mathfrak q^3$ whose local component at $\mathfrak q$ is simple
supercuspidal is related to fourth moments over automorphic representations of
$\mathrm{PGL}_2/F$ of conductor $\mathfrak q^2$ \textit{without}
supercuspidal ramification. This structure bypasses the positivity barrier
inherent in the cubic moment method.

A central ingredient is the analysis of the ramification weight
$\mathcal Q_{\mathfrak q,\omega}(\sigma)$ arising in the reciprocity
formula. On the type-II family of ramified principal series of conductor
$\mathfrak q^2$, this weight admits an explicit arithmetic description
in terms of Gauss and Jacobi sums, leading to square-root cancellation. Larger values occur only on
thinner spectral strata, where they are compensated for by the sparsity
or additional structure of the corresponding families. This
\textit{cancellation--sparsity dichotomy} ultimately yields the desired
Weyl-type bound; see \textsection\ref{sec1.2} for further discussion
and a comparison with the work of Petrow--Young \cite{PY20, PY23}.

\subsection{Weyl Bound via a Fourth-Moment Estimate}
Let $\mathfrak q\subset \mathcal O_F$ be a prime ideal and
$\omega$ a unitary character of $F^\times\backslash\mathbb A_F^\times$.
Let $\mathcal{F}_t^\zeta(\mathfrak q^3;\omega)$ denote the family of
cuspidal automorphic representations
$\pi=\otimes_v' \pi_v$ of $\mathrm{GL}_2/F$ with central character
$\omega$ and conductor $\mathfrak q^3$ such that
$\pi_v\simeq\sigma_t^\zeta$ for $v\mid\mathfrak q$, where
$\sigma_t^\zeta$ is the simple supercuspidal representation
parametrized by $t\in\mathcal O_v^\times$ and
$\zeta\in\mathbb C^\times$; see \textsection\ref{sec3.1}.
Such representations are said to have
\emph{simple supercuspidal ramification}.

The main analytic input is the following fourth moment estimate.

\begin{thmx}\label{thmA}
For each Archimedean place $v\mid\infty$, let $\mathbf C_v>0$, and set
$\mathbf C_\infty:=\prod_{v\mid\infty}\mathbf C_v$. Then
\begin{equation}\label{1.1}
\sum_{\substack{\pi\in\mathcal{F}_t^\zeta(\mathfrak q^3;\omega)\\
C_v(\pi)\le\mathbf C_v,\ v\mid\infty}}
|L(1/2,\pi)|^4
\ll_{F,\varepsilon}
\mathbf C_\infty^{1+\varepsilon}
N_F(\mathfrak q)^{2+\varepsilon}.
\end{equation}
\end{thmx}

\begin{remark}
The proof of the fourth moment estimate \eqref{1.1} also makes use of the cubic moment for certain $\mathrm{PGL}_ 2$ $L$-functions established by Conrey--Iwaniec \cite{CI00} and extended to number fields by Nelson \cite{Nel19}; see \textsection\ref{sec8.4}.	
\end{remark}

\begin{remark}
Heuristically, the moment estimate \eqref{1.1} may be viewed as a $p$-adic analogue of Jutila's spectral-aspect result \cite{Jut01}. In both cases, one studies the $4$-th moment of $L$-functions of conductor $C$ over a family $\mathcal{F}$ satisfying
\begin{equation}\label{eq1.4}
\log |\mathcal{F}|/\log C=2/3+o(1),
\end{equation}
although the methods are quite different. See Remark \ref{rmk1.7} in \textsection\ref{sec1.3} for further discussion.
\end{remark}

As an immediate consequence we obtain the following Weyl-type bound.

\begin{cor}\label{cor1.1}
Let $\pi$ be a cuspidal automorphic representation of $\mathrm{GL}_2/F$
with simple supercuspidal ramification. Then
\begin{equation}\label{e1.1}
L(1/2,\pi)
\ll_{F,\varepsilon}
C_\infty(\pi)^{\frac14+\varepsilon}
C_{\fin}(\pi)^{\frac16+\varepsilon},
\end{equation}
where $C_\infty(\pi)$ and $C_{\fin}(\pi)$ denote the Archimedean and
arithmetic conductors, respectively. In particular, if $\mathfrak{q}$ is a prime ideal, then every unitary cuspidal automorphic representation $\pi$ of $\mathrm{PGL}_2/F$ with arithmetic conductor $\mathfrak{q}^3$ satisfies the Weyl bound \eqref{e1.1}.
\end{cor}

\begin{remark}
To the best of our knowledge, the bound \eqref{e1.1} is the first
Weyl-type bound in the level aspect for $\mathrm{GL}_2$ $L$-functions in a family of
odd conductor exponent allowing arbitrary central character. In
particular, this lies beyond the reach of the cubic moment method.
\end{remark}

\begin{remark}
Twisting by $|\cdot|^{it}$ affects only the Archimedean conductor.
Applying \eqref{e1.1} to $\pi\otimes |\cdot|^{it}$ therefore gives
\begin{align*}
L(1/2+it,\pi)
=
L(1/2,\pi\otimes |\cdot|^{it})
\ll_{F,\varepsilon}
C_{\infty}(\pi\otimes |\cdot|^{it})^{\frac14+\varepsilon}
C_{\fin}(\pi)^{\frac16+\varepsilon}.
\end{align*}
\end{remark}

As mentioned above, the Weyl-type bound represents the natural limit of
present-day approaches to subconvexity problems. The results obtained
here demonstrate that the symmetric spectral reciprocity method provides
a new mechanism for reaching this barrier in settings where the cubic
moment method is unavailable, such as for conductors of odd exponent or
with nontrivial central character.

\subsection{A Refined Symmetric Spectral Reciprocity}\label{sec1.2}

A principal structural result of this paper is a refined symmetric spectral
reciprocity formula, which provides the key input for the proof of
Theorem \ref{thmA} and is schematically of the shape 
\begin{align*}
\sum_{\substack{\pi\in \mathcal{F}_{t}^{\zeta}(\mathfrak{q}^3;\omega)
\cup \mathcal{F}_{t}^{-\zeta}(\mathfrak{q}^3;\omega)\\
C_v(\pi)\ll \mathbf{C}_v,\ v\mid\infty}}
|L(1/2,\pi)|^4
\rightsquigarrow
\frac{\mathbf{C}_{\infty}^{1/2}}{N_F(\mathfrak{q})}
\int_{\sigma}
\mathcal{Q}_{\mathfrak{q},\omega}(\sigma)
L(1/2,\sigma)^3
|L(1/2,\sigma\times\overline{\omega})|
\,d\mu_{\sigma},
\end{align*}
where $\sigma$ runs over the automorphic spectrum of $\mathrm{PGL}_2/F$
and $\mathcal{Q}_{\mathfrak{q},\omega}(\sigma)$ is the ramification weight
defined in \eqref{e8.6} of \textsection\ref{sec8.1}.

The main analytic task is to determine the support and size of
$\mathcal{Q}_{\mathfrak{q},\omega}(\sigma)$ as $\mathfrak q$ varies.
To describe its spectral support, we introduce the following classification.

\begin{defn}\label{def1.3}
Let $v\mid\mathfrak{q}$, and let $\sigma$ be a unitary generic
automorphic representation of $\mathrm{PGL}_2/F$. Let $\mathrm{St}$ denote the Steinberg representation.  
We classify $\sigma$ according to the local type of $\sigma_v$ as follows:
\begin{itemize}
\item $\sigma$ is of type \RNum{1} if $\sigma_v$ is an unramified principal series;
\item of type \RNum{2} if $\sigma_v$ is a ramified principal series of conductor exponent $2$;
\item of type \RNum{3} if $\sigma_v$ is a complementary series of conductor exponent $2$;
\item of type \RNum{4} if $\sigma_v\simeq \mathrm{St}\otimes\chi_v$ with $\chi_v$ an unramified quadratic character;
\item of type \RNum{5} if $\sigma_v\simeq \mathrm{St}\otimes\chi_v$ with $\chi_v$ a quadratic character of conductor exponent $1$.
\end{itemize}

For $*\in\{\RNum{1},\RNum{2},\RNum{3},\RNum{4},\RNum{5}\}$, we denote by
$\mathcal{F}_*(\mathfrak{q}^2;\mathbf{1})$ the set of unitary generic
automorphic representations of $\mathrm{PGL}_2/F$ of type $*$ that are
unramified at all places $v'\nmid\mathfrak{q}$.
\end{defn}

Once the support of $\mathcal{Q}_{\mathfrak{q},\omega}(\sigma)$ is understood,
we obtain the estimate
\begin{equation}\label{1.2}
\sum_{\substack{\pi\in \mathcal{F}_{t}^{\zeta}(\mathfrak{q}^3;\omega)
\cup \mathcal{F}_{t}^{-\zeta}(\mathfrak{q}^3;\omega)\\
C_v(\pi)\ll \mathbf{C}_v,\ v\mid\infty}}
|L(1/2,\pi)|^4
\ll
\mathbf{C}_{\infty}N_F(\mathfrak{q})^{2}
+
\frac{\mathbf{C}_{\infty}^{1/2}}{N_F(\mathfrak{q})}
\sum_{*\in\{\RNum{1},\RNum{2},\RNum{3},\RNum{4},\RNum{5}\}} S_*,
\end{equation}
where
\begin{align*}
S_* :=
\int_{\substack{\sigma\in \mathcal{F}_*(\mathfrak{q}^2;\mathbf{1})\\
C_v(\sigma)\le
\mathbf{C}_v^{\varepsilon}N_F(\mathfrak{q})^{\varepsilon},\
v\mid\infty}}
\mathcal{Q}_{\mathfrak{q},\omega}(\sigma)
L(1/2,\sigma)^3
|L(1/2,\sigma\times\overline{\omega})|
\,d\mu_{\sigma}.
\end{align*}

This reveals a conductor-lowering phenomenon: the fourth moment over
representations of conductor $\mathfrak q^3$ is controlled by 
moments over automorphic representations of $\mathrm{PGL}_2/F$
of conductor $\mathfrak q^2$. In particular, the reciprocity formula
transfers the sparse family
$\mathcal{F}_{t}^{\zeta}(\mathfrak{q}^3;\omega)\cup
\mathcal{F}_{t}^{-\zeta}(\mathfrak{q}^3;\omega)$
to a denser family of $\mathrm{PGL}_2$ representations of conductor
$\mathfrak q^2$.

The weight $\mathcal{Q}_{\mathfrak{q},\omega}(\sigma)$ plays a role
analogous to the hybrid character sum $g(\chi,\psi)$ in
\cite[equation (11.10)]{CI00}. In the essential case, it admits an
explicit description in terms of $\varepsilon$-factors, Gauss and Jacobi sums,
leading to square-root cancellation. This estimate is not uniform over the
entire dual spectrum: larger values may occur on thinner spectral strata, where
they are compensated for by the sparsity or additional structure of the
corresponding families. This cancellation--sparsity dichotomy is reminiscent
of the work of Petrow--Young \cite{PY20,PY23}: for cube-free conductor, their
character sum $g(\chi,\psi)$ exhibits square-root cancellation uniformly in
$\psi$, whereas at prime-cube conductor it can be larger for certain singular
characters, which are confined to a small number of cosets of a character
subgroup of lower conductor. Combined with suitable moment bounds for
$L$-functions, this structure yields the Weyl-type estimate.

The required moment estimates on the $\mathrm{PGL}_2$ side rely on the
fourth moment of $\mathrm{GL}_2$ $L$-functions established in
\cite{Yan25} and the cubic moment of $\mathrm{PGL}_2$ $L$-functions
proved by Conrey--Iwaniec \cite{CI00} and extended to number fields by
Nelson \cite{Nel19}.

\subsection{Conceptual remarks}\label{sec1.3}

We record several remarks explaining the role of the refined spectral
reciprocity formula \eqref{1.2}, its relation to existing approaches,
and the structural features underlying the conductor-lowering
mechanism. We also briefly discuss the Rankin--Selberg analogue and
the obstacles that prevent the present method from yielding a Weyl
bound in that setting.

\begin{remark}[Why a refined reciprocity formula is needed]
The explicit symmetric spectral reciprocity formula of
\cite[equation (1.3)]{Yan25} applies to fourth moments over the
\emph{full} family $\mathcal{F}(\mathfrak q^3;\omega)$ of unitary
generic representations of conductor $\mathfrak q^3$ and central
character $\omega$. In that setting the reciprocity formula produces
a main term of the same order of magnitude as the size of the family
(roughly $N_F(\mathfrak q)^3$), together with a dual moment of the form
\begin{align*}
\mathbf{C}_{\infty}^{\frac12}N_F(\mathfrak q)^{\frac32}
\int_{\substack{\sigma\in \mathcal{F}(\mathcal{O}_F;\mathbf{1})\\
C_v(\sigma)\le
\mathbf{C}_v^{\varepsilon}N_F(\mathfrak q)^{\varepsilon},\
v\mid\infty}}
|\lambda_{\sigma}(\mathfrak q^3)|
L(1/2,\sigma)^3
|L(1/2,\sigma\times\overline{\omega})|
\,d\mu_{\sigma},
\end{align*}
where $\lambda_{\sigma}(\mathfrak q^3)$ denotes the Hecke eigenvalue.
In principle, subconvexity could be obtained from nontrivial bounds for
$|\lambda_{\sigma}(\mathfrak q^3)|$ together with amplification.
However, since the main term is of size comparable to the family
itself, this framework is insufficient to reach the Weyl exponent.

The refined reciprocity formula \eqref{1.2}, however, restricts the
spectral sum to the sparse family with simple supercuspidal
ramification, whose size is roughly $N_F(\mathfrak q)^2$. Moreover,
on the dual side the level-$\mathcal{O}_F$ spectrum is replaced by
representations of conductor $\mathfrak q^2$, and the Hecke eigenvalue
$\lambda_{\sigma}(\mathfrak q^3)$ is replaced by the more delicate
arithmetic weight $\mathcal Q_{\mathfrak q,\omega}(\sigma)$. In this
setting the conductor-lowering mechanism remains effective and,
combined with a precise analysis of
$\mathcal Q_{\mathfrak q,\omega}(\sigma)$, makes the Weyl bound
accessible. This explains in what sense \eqref{1.2} refines the
explicit formula of \cite{Yan25}.
\end{remark}

\begin{remark}[Comparison with Jutila's method]\label{rmk1.7}
Jutila's argument starts from an approximate functional equation and applies Kuznetsov and Voronoi transformations to reduce the problem to binary additive divisor sums, which are then treated using the spectral theory of automorphic forms. By contrast, our approach proceeds through symmetric spectral reciprocity. After localizing the representations to a suitable sparse family $\mathcal{F}$ satisfying \eqref{eq1.4}, an explicit analysis of the dual side reveals a level-lowering phenomenon. The main arithmetic input is therefore the analysis of the ramification weight $\mathcal{Q}_{\mathfrak{q},\omega}(\sigma)$, rather than that of binary additive divisor sums. We expect that an Archimedean analogue of refined symmetric spectral reciprocity could provide a new proof of Jutila's theorem and extend it to number fields.

A related analytic approach was recently developed by Miao--Zhang \cite{MZ25} for a specific family of triple product $L$-functions associated with $\mathrm{PGL}_2/\mathbb{Q}$ in the cubic level aspect. They prove a subconvexity bound that attains Weyl-type strength under the Ramanujan conjecture.
\end{remark}

\begin{remark}[Comparison with Blomer--Jana--Nelson's method]
It is instructive to compare our approach with that of Blomer--Jana--Nelson \cite{BJN23}. Their method combines the Ichino--Watson formula with the Kuznetsov trace formula to study triple product periods, and hence blends integral representations with analytic number theory. In contrast, the symmetric spectral reciprocity formula may be viewed as an Eisenstein-series analogue of this framework. Our refined reciprocity formula \eqref{1.2}, which is a special case of Theorem \ref{thm2.1}, is obtained from the spectral theory of period integrals together with automorphic projection in the spirit of the trace formula. This makes it possible to treat substantially narrower spectral windows without using the Kuznetsov formula.

A key feature of our method is that the relevant dual families and moments can be localized explicitly. The dual side contains a ramification weight
$\mathcal Q_{\mathfrak{q},\omega}(\sigma)$ whose dependence on
$\mathfrak{q}$ is amenable to direct analysis. In the essential case, this weight is expressed in terms of Gauss and Jacobi sums, which exhibit square-root cancellation. This explicit structure allows us to work uniformly over general number fields
$F$ while keeping the arithmetic mechanism of the reciprocity formula transparent. To the best of our knowledge, this is the first use of a refined spectral reciprocity framework to analyze such period integrals. We expect this method to have further applications.
\end{remark}

\begin{remark}[Choice of test functions]
If the test function in Theorem \ref{thm2.1} is chosen so that the left-hand side of
\eqref{1.2} detects only representations
$\pi\in\mathcal F_t^\zeta(\mathfrak{q}^3;\omega)$---for instance by using
the matrix coefficients in \cite[(5.18)]{Kni25}---then the weight
$\mathcal Q_{\mathfrak{q},\omega}(\sigma)$ need no longer be supported
only on representations of conductor $\mathfrak{q}^2$ and may also
receive contributions from conductor $\mathfrak{q}^3$. In such a
situation the conductor-lowering phenomenon disappears. This reflects a
kind of \textit{uncertainty principle} inherent in the reciprocity formula: a
test function that sharply isolates the family
$\mathcal F_t^\zeta(\mathfrak{q}^3;\omega)$ may destroy the interference
responsible for the conductor drop.

Here $\zeta$ denotes the local root number at $\mathfrak{q}$. When the 
central character $\omega=\mathbf{1}$, however, the situation is
different. In this case the global root number is $\pm1$, and
representations with root number $-1$ do not contribute since
$L(1/2,\pi)=0$. Thus restricting to
$\mathcal F_t^\zeta(\mathfrak{q}^3;\omega)$ effectively amounts to the
pair $\mathcal F_t^{\zeta}(\mathfrak{q}^3;\omega)\cup
\mathcal F_t^{-\zeta}(\mathfrak{q}^3;\omega)$, and the conductor-lowering
phenomenon persists. When the central character is nontrivial, however,
the parity constraint from the functional equation no longer forces one
of the two root-number families to vanish. Nevertheless, since the
spectral side of the reciprocity formula is nonnegative---reflecting the
symmetry of the spectral reciprocity---we may enlarge the family
$\mathcal F_t^\zeta(\mathfrak{q}^3;\omega)$ to
$\mathcal F_t^{\zeta}(\mathfrak{q}^3;\omega)\cup
\mathcal F_t^{-\zeta}(\mathfrak{q}^3;\omega)$
without essential loss and thereby recover the conductor-lowering
phenomenon.
\end{remark}

\begin{remark}[The Rankin--Selberg case]
One may ask whether Theorem~\ref{thmA} extends to the Rankin--Selberg
case. The corresponding reciprocity formula was established in
\cite{HY25}, together with the required local analysis at all
$v\nmid\mathfrak q$. Moreover, the calculations at $v\mid\mathfrak q$
carried out in the present paper remain valid in that setting.
Consequently a similar conductor-lowering phenomenon arises for
Rankin--Selberg $L$-functions, together with an explicit description of
the weight $\mathcal Q_{\mathfrak q,\omega}(\sigma)$.

For general $\omega$, the argument would require a sharp bound for the
first moment of $\mathrm{GL}_3\times\mathrm{GL}_2$ $L$-functions with the
$\mathrm{GL}_3$ form fixed and the $\mathrm{GL}_2$ form varying over a
sparse family, analogous to the cubic moment in \cite{CI00}. Such an
estimate is presently unavailable, and therefore the Weyl bound does
not follow in the Rankin--Selberg case in general.

When $F=\mathbb{Q}$ and $\omega=\mathbf{1}$, the required input can be
supplied by the weaker moment bound of \cite[Theorem 1.2]{GHLN24}
together with the fifth moment estimates of \cite[Theorem 3]{BK19}, after
refining the dependence on the spectral parameters from exponential to
polynomial. Since this requires only a different use of existing moment
estimates, we omit the details.
\end{remark}

\begin{remark}[Sub-Weyl Expectation]
Suppose $F=\mathbb{Q}$ and $\omega=\mathbf{1}$. The essential part of the dual side is roughly a sum of the form
\begin{equation}\label{e1.4}
q^{-\frac{1}{2}}\sum_{\text{$\chi\, (\mathrm{mod}\, q)$: primitive}}\sum_{\substack{\sigma\in \mathcal{F}(q^2;\mathbf{1})\\
\sigma\otimes\chi\in \mathcal{F}(q;\chi^2)}}\varepsilon(1/2,\chi)J(\chi,\chi)|L(1/2,\sigma)|^4,
\end{equation}
where $\varepsilon(1/2,\chi)$ is the root number and $J(\chi,\chi)$ denotes the Jacobi sum. Using the triangle inequality, the sharp bound for $J(\chi,\chi)$, and the fourth moment of $\mathrm{GL}_2$ $L$-functions, in a manner parallel to \cite{CI00}, one obtains the Weyl bound. To go beyond the Weyl exponent, one would need to exploit additional cancellation in \eqref{e1.4}, which seems quite challenging but also of considerable interest. With this goal in mind, we give a detailed explicit calculation in \textsection\ref{sec5}, from which one can recover the precise form of \eqref{e1.4}. Although the resulting expression is more involved, it makes the support of the weight $\mathcal{Q}_{\mathfrak{q},\omega}(\sigma)$ and its size more transparent.
\end{remark}

\subsection{Outline of the Paper}

We briefly outline the structure of the paper.

\begin{itemize}

\item In \textsection\ref{sec2} we review the symmetric spectral reciprocity formula of \cite{Yan25}:
\begin{align*}
\mathcal{J}_{\mathrm{Spec}}^{\heartsuit}(\mathbf{0},R(f_{\mathfrak{q}})\boldsymbol{h})
=
\mathcal{I}_{\Spec}^{\heartsuit}(\mathbf{0},R(f_{\mathfrak{q}})\boldsymbol{h})
+\text{Degenerate},
\end{align*}
where ``\text{Degenerate}'' consists of $24$ additional terms; see Theorem \ref{thm2.1}.

\item In \textsection\ref{sec3} we review the construction of simple supercuspidal representations, construct the test function, and compute the relevant local period integrals on the spectral side.

\item In \textsection\ref{sec4} we construct an explicit orthonormal basis for the space of $I_v(1)$-invariant vectors in the Whittaker model.

\item In \textsection\ref{sec5} we compute certain ramified local period integrals. These calculations play a crucial role in the analysis of the ramification weight $\mathcal{Q}_{\mathfrak{q},\omega}(\sigma)$.

\item In \textsection\ref{sec6} we choose suitable automorphic data for the reciprocity formula tailored to our subconvexity application.

\item In \textsection\ref{sec7} we establish a lower bound for
$\mathcal{J}_{\mathrm{Spec}}^{\heartsuit}(\mathbf{0},R(f_{\mathfrak{q}})\boldsymbol{h})$
in terms of the left-hand side of \eqref{1.2}.

\item In \textsection\ref{sec8} we use the calculations of \textsection\ref{sec5} to derive a sharp upper bound for
$\mathcal{J}_{\mathrm{Spec}}^{\heartsuit}(\mathbf{0},R(f_{\mathfrak{q}})\boldsymbol{h})$
in terms of the second term on the right-hand side of \eqref{1.2}.

\item In \textsection\ref{sec9} and \textsection\ref{sec10} we bound the degenerate terms in the reciprocity formula.

\item In \textsection\ref{sec11} we combine the preceding estimates to complete the proof of Theorem~\ref{thmA}.

\end{itemize}

\subsection{Notation}

\subsubsection{Number Fields and Measures}\label{1.1.1}
Let $F$ be a number field with ring of integers $\mathcal{O}_F.$ Let $N_F$ be the absolute norm. Let $\mathfrak{O}_F$ be the different of $F.$ Let $\mathbb{A}_F$ be the adele group of $F.$ Let $\Sigma_F$ be the set of places of $F.$ Denote by $\Sigma_{F,\fin}$ (resp. $\Sigma_{F,\infty}$) the set of non-Archimedean (resp. Archimedean) places. For $v\in \Sigma_F,$ we denote by $F_v$ the corresponding local field. For a non-Archimedean place $v,$ let $\mathcal{O}_v$ be the ring of integers of $F_v$, and $\mathfrak{p}_v$ be the maximal prime ideal in $\mathcal{O}_v$. Given an integral ideal $\mathcal{I},$ we say $v\mid \mathcal{I}$ if $\mathcal{I}\subseteq \mathfrak{p}_v.$ Fix a uniformizer $\varpi_{v}\in\mathfrak{p}_v.$ Denote by $e_v(\cdot)$ the evaluation relative to $\varpi_v$ normalized as $e_v(\varpi_v)=1.$ Let $d_v=e_v(\mathfrak{D}_{F})$ be the ramification index and $q_v$ be the cardinality of $\mathcal{O}_v/\mathfrak{p}_v.$ We use $v\mid\infty$ to indicate an Archimedean place $v$ and write $v<\infty$ if $v$ is non-Archimedean. Let $|\cdot|_v$ be the norm in $F_v.$ Put $|\cdot|_{\infty}=\prod_{v\mid\infty}|\cdot|_v$ and $|\cdot|_{\fin}=\prod_{v<\infty}|\cdot|_v.$ Let $|\cdot|_{\mathbb{A}_F}=|\cdot|_{\infty}\otimes|\cdot|_{\fin}$. We will simply write $|\cdot|$ for $|\cdot|_{\mathbb{A}_F}$ in calculation over $\mathbb{A}_F^{\times}$ or its quotient by $F^{\times}$.   

Let $\psi_{\mathbb{Q}}$ be the additive character on $\mathbb{Q}\backslash \mathbb{A}_{\mathbb{Q}}$ such that $\psi_{\mathbb{Q}}(t_{\infty})=\exp(2\pi it_{\infty}),$ for $t_{\infty}\in \mathbb{R}\hookrightarrow\mathbb{A}_{\mathbb{Q}}.$ Let $\psi=\psi_F=\psi_{\mathbb{Q}}\circ \Tr_F,$ where $\Tr_F$ is the trace map. Then $\psi(t)=\prod_{v\in\Sigma_F}\psi_v(t_v)$ for $t=(t_v)_v\in\mathbb{A}_F.$ For $v\in \Sigma_F,$ let $dt_v$ be the additive Haar measure on $F_v,$ self-dual relative to $\psi_v.$ Then $dt=\prod_{v\in\Sigma_F}dt_v$ is the standard Tamagawa measure on $\mathbb{A}_F$. Let $d^{\times}t_v=\zeta_{F_v}(1)dt_v/|t_v|_v,$ where $\zeta_{F_v}(\cdot)$ is the local Dedekind zeta factor. In particular, $\Vol(\mathcal{O}_v^{\times},d^{\times}t_v)=\Vol(\mathcal{O}_v,dt_v)=q_v^{-d_v/2}$ for all finite place $v.$ Moreover, $\Vol(F\backslash\mathbb{A}_F; dt_v)=1$ and $\Vol(F\backslash\mathbb{A}_F^{(1)},d^{\times}t)=\underset{s=1}{\Res}\ \zeta_F(s),$ where $\mathbb{A}_F^{(1)}$ is the subgroup of ideles $\mathbb{A}_F^{\times}$ with norm $1,$ and $\zeta_F(s)=\prod_{v<\infty}\zeta_{F_v}(s)$ is the finite Dedekind zeta function. Denote by $\widehat{F^{\times}\backslash\mathbb{A}_F^{(1)}}$  the Pontryagin dual of $F^{\times}\backslash\mathbb{A}_F^{(1)}.$

\subsubsection{Reductive Groups}\label{sec1.3.2}
For an algebraic group $H$ over $F$, we will denote by $[H]:=H(F)\backslash H(\mathbb{A}_F).$ We equip measures on $H(\mathbb{A}_F)$ as follows: for each unipotent group $U$ of $H,$ we equip $U(\mathbb{A}_F)$ with the Haar measure such that, $U(F)$ being equipped with the counting measure. The measure of $[U]$ is $1.$ We equip the maximal compact subgroup $K$ of $H(\mathbb{A}_F)$ with the Haar measure such that $K$ has total mass $1.$ When $H$ is split, we also equip the maximal split torus of $H$ with Tamagawa measure induced from that of $\mathbb{A}_F^{\times}.$

In this paper we set  $A=\diag(\mathrm{GL}(1),1),$ and $G=\mathrm{GL}_2.$ Let $B$ be the group of upper triangular matrices in $G$.  
Let $\overline{G}=Z\backslash G$ and $B_0=Z\backslash B,$ where $Z$ is the center of $G.$ Let $T_B$ be the diagonal subgroup of $B$. Then $A\simeq Z\backslash T_B.$ Let $N$ be the unipotent radical of $B$. Let $w=\begin{pmatrix}
	& 1\\
1	& 
\end{pmatrix}$ be the Weyl element. 

Let $K=\otimes_vK_v$ be a maximal compact subgroup of $G(\mathbb{A}_F),$ where $K_v=\mathrm{U}_2(\mathbb{C})$ if $v$ is complex, $K_v=\mathrm{O}_2(\mathbb{R})$ if $v$ is real, and $K_v=G(\mathcal{O}_v)$ if $v<\infty.$ For $v\in \Sigma_{F,\fin},$ $m\in\mathbb{Z}_{\geq 0},$ define 
\begin{equation}\label{2.1}
K_{0,v}[m]:=\Big\{\begin{pmatrix}
a&b\\
c&d
\end{pmatrix}\in G(\mathcal{O}_v):\ c\in \mathfrak{p}_v^{m}\Big\}.
\end{equation}
For an integral ideal $\mathfrak{q}\subseteq \mathcal{O}_F$, we define $K_0[\mathfrak{q}]:=\otimes_{v<\infty}K_{0,v}[e_v(\mathfrak{q})]$. 
\subsubsection{Whittaker functions}\label{sec1.3.3}
Let $\theta$ be the generic character induced by $\psi$:
\begin{equation}\label{eq1.12}
\theta\left(\begin{pmatrix}
1& u\\
&1
\end{pmatrix}\right)=\psi(u),\  \ u=(u_v)_v\in \mathbb{A}_F.
\end{equation}
Then $\theta=\otimes_v\theta_v$, where $\theta_v\left(\begin{pmatrix}
1& u_v\\
&1
\end{pmatrix}\right):=\psi_v(u_v)$, $v\in \Sigma_F$. 

Let $\pi$ be a generic automorphic representation of $G/F$. Let $\phi\in \pi$ be an automorphic form. Define the Whittaker function of $\phi$ (relative to $\theta$) by 
\begin{equation}\label{eq1.3}
W_{\phi}(g):=\int_{[N]}\phi(ug)\overline{\theta}(u)du.
\end{equation}
Let $W_{\phi}^*(g):=W_{\phi}(\diag(-1,1)
g)$ be the Whittaker function of $\phi$ relative to $\overline{\theta}$. 


\subsubsection{Gauss Sums}
Let $v<\infty$ and $\chi_v$ be a unitary character of $F_v^{\times}$. Let 
\begin{align*}
G(\chi_v,\psi_v):=\sum_{\alpha\in \mathbb{F}_v^{\times}}\chi_v(\alpha)\psi_v(\varpi_v^{-r_{\chi_v}}\alpha)
\end{align*}
be the Gauss sum. We have $|G(\chi_v,\psi_v)|=q_v^{r_{\chi_v}/2}$ if $r_{\xi_v}\geq 1$.

\subsubsection{Jacobi Sums}
Let $v<\infty$ and $\chi_{1,v}$ and $\chi_{2,v}$ be unitary character of $F_v^{\times}$. Let 
\begin{align*}
J(\chi_{1,v},\chi_{2,v}):=\sum_{\alpha\in \mathbb{F}_v-\{0,1\}}\chi_{1,v}(\alpha)\chi_{2,v}(1-\alpha)
\end{align*}
be the Jacobi sum. 

\subsubsection{Conventions on Local Factors}\label{sec1.2.7}
Throughout the paper, for any pure tensor or product over $\mathbb{A}_F$ (e.g., sections, Whittaker functions, or $L$-functions), we use a subscript $v$ to denote its $v$-th local component in the restricted tensor product decomposition. 

For a local representation $\sigma_v$, we write $L(s,\sigma_v)$ for its local $L$-factor. When $\sigma_v$ is the $v$-component of an automorphic representation $\sigma$, we instead write $L_v(s,\sigma_v)$ to indicate that it is the $v$-th Euler factor of the global $L$-function. We denote by $\Lambda(s,\sigma)$ the completed $L$-function (including the Archimedean factors) and by $L(s,\sigma)$ its finite part.

\subsubsection{The $\varepsilon$-conventions}\label{sec1.5.4}
Throughout this paper, we adopt the $\varepsilon$-convention, whereby
$\varepsilon$ denotes a positive quantity that may be taken arbitrarily
small, with its value allowed to vary from one occurrence to another.

\medskip

\textbf{Acknowledgements}
I would also like to express my sincere gratitude to Peter Humphries, Rizwanur Khan, Wenzhi Luo, Philippe Michel and Matthew Young for their precise comments and valuable suggestions.

\section{Symmetric Spectral Reciprocity}\label{sec2}

In this section, we briefly recall the symmetric spectral reciprocity formula established in \cite{Yan25} and develop the refinement needed to isolate a sparse family on the spectral side.

\subsection{Eisenstein Series}\label{sect2.1}
Let $\sigma=\chi|\cdot|^s\boxplus \chi'|\cdot|^{-s}$ be an induced representation of $G/F$. Let $h(\cdot,s)$ be a section in $\sigma$. Define 
\begin{align*}
E(x,s)=\mathrm{Eis}(h)(x,s):=\sum_{\delta\in B(F)\backslash G(F)}h(\delta x,s),\ \ \ x\in G(\mathbb{A}_F). 
\end{align*}

By Poisson summation, the series $E(x,s)$ converges absolutely when $\Re(s)\ggg 1$, and admits a meromorphic function to $s\in \mathbb{C}$ with a functional equation. 

\subsubsection{Fourier Expansions}
Define
\begin{align*}
h^{\Diamond}(x,s)
:=\int_{N(\mathbb{A}_F)} h(wux,s)\,du,
\end{align*}
which is the section obtained by applying the standard intertwining operator to $h$.  
Moreover, the associated Whittaker function (defined by \eqref{eq1.3}) is equal to 
\begin{equation}\label{f2.1}
W(x,s)=W_{E(\cdot,s)}(x)
=\int_{N(\mathbb{A}_F)} h(wux,s)\overline{\theta(u)}\,du. 
\end{equation}

We have the Fourier expansion 
\begin{equation}\label{eq2.6}
E(x,s)=\mathcal{C}(x,s)+\mathcal{G}(x,s),
\end{equation}
where 
\begin{align*}
\mathcal{C}(x,s)=&h(x,s)+h^{\Diamond}(x,s),\ \ \ \ \ 
\mathcal{G}(x,s)=\sum_{\alpha\in F^{\times}}W\left(\begin{pmatrix}
\alpha\\
& 1
\end{pmatrix}x,s\right).
\end{align*}

\subsubsection{Godement Sections}

For $\Re(s)\ggg 1$, the section $h(\cdot,s)$ can be constructed via the Tate integral (and extended to general $s$ by meromorphic continuation) as
\begin{align*}
h(x,\Phi,\chi,\chi',s)
:=\chi(\det x)|\det x|^{s+\frac{1}{2}}
\int_{\mathbb{A}_F^{\times}}
\Phi((0,t)x)\chi\chi'^{-1}(t)|t|^{2s+1}\,d^{\times}t,
\end{align*}
where $\Phi$ is a Schwartz--Bruhat function on $\mathbb{A}_F$.

\subsubsection{Eisenstein Series from  Products}
Let $\pi_{i}=\chi_i|\cdot|^{s_i}\boxplus \chi_i'|\cdot|^{-s_i}$ and let $h_i(\cdot,s_i)$ be a section in $\pi_i$, $1\leq i\leq 2$. Then the product $h_1(\cdot,s_1)h_2(\cdot,s_2)$ is a section in $\chi_1\chi_2|\cdot|^{1/2+s_1+s_2}\boxplus \chi_1'\chi_2'|\cdot|^{-1/2-s_1-s_2}$. The Whittaker function of the Eisenstein series associated with $h_1(\cdot,s_1)h_2(\cdot,s_2)$ is
\begin{equation}\label{e2.3}
W_{h_1,h_2}(x,s_1,s_2)
=\int_{N(\mathbb{A}_F)}
h_1(wux,s_1)h_2(wux,s_2)\overline{\theta}(u)\,du,
\end{equation}
which converges absolutely for all $(s_1,s_2)\in\mathbb{C}^2$.

\subsubsection{Eisenstein Series from the Continuous Spectrum}
Let $\mu_1, \mu_2\in \widehat{F^{\times}\backslash \mathbb{A}_F^{(1)}}$, and  
\begin{multline}\label{f2.3}
H(\mu_1,\mu_2):=\Big\{h\in L^2(K):\ h\left(\begin{pmatrix}
a& b\\
&d
\end{pmatrix}k\right)=\mu_1(a)\mu_2(d)h(k),\\ 
\forall\ k\in K,\ \begin{pmatrix}
a& b\\
&d
\end{pmatrix}\in B(\mathbb{A}_F)\cap K
\Big\}.
\end{multline}

Given $h\in H(\mu_1,\mu_2)$ and $\lambda\in \mathbb{C}$, we define 
\begin{align*}\label{eq2.2}
S(h)\left(\begin{pmatrix}
a& b\\
&d
\end{pmatrix}k,\lambda\right):=\mu_1(a)\mu_2(d)\Big|\frac{a}{d}\Big|^{\lambda+1/2}h(k),
\end{align*}
where $a,d\in \mathbb{A}_F^{\times}$, $b\in \mathbb{A}_F$, and $k\in K$. Define the associated Eisenstein series by
\begin{equation}\label{fc2.4}
E(x,h,\lambda):=\sum_{\gamma\in B(F)\backslash G(F)}S(h)(\gamma x,\lambda),\ \ \Re(\lambda)>1/2.
\end{equation}

These Eisenstein series are used to construct the continuous spectrum of $L^2([G],\omega)$.

\subsection{Rankin--Selberg Convolutions}\label{sec2.1.}

Let $\pi_1$ and $\pi_2$ be two generic representations of $[G]$. Let
$W_1=\otimes_v W_{1,v}$ and $W_2=\otimes_v W_{2,v}$ be pure tensors in the
Whittaker models of $\pi_1$ and $\pi_2$, respectively.
Let $\pi_3=\chi|\cdot|^{s}\boxplus \chi'|\cdot|^{-s}$ be an induced representation of $[G]$, where $\chi$ and $\chi'$ are unitary
characters, and let $h_3=\otimes_v h_{3,v}$ be a section of $\pi_3$.
Assume that the product of the central characters of $\pi_1$, $\pi_2$, and
$\pi_3$ is trivial.

For $\Re(s)\gg 1$, we define
\begin{align*}
\Psi(W_1,W_2,h_3)
:=
\int_{N(\mathbb{A}_F)\backslash \overline{G}(\mathbb{A}_F)}
W_1(x)W_2(x)h_3(x)\,dx
=
\prod_{v\leq \infty}
\Psi_v(W_{1,v},W_{2,v},h_{3,v}),
\end{align*}
where, for each place $v\le \infty$,
\begin{equation}\label{c2.6}
\Psi_v(W_{1,v},W_{2,v},h_{3,v})
:=
\int_{N(F_v)\backslash \overline{G}(F_v)}
W_{1,v}(x_v)W_{2,v}(x_v)h_{3,v}(x_v)\,dx_v.
\end{equation}

By Rankin--Selberg theory, $\Psi(W_1,W_2,h_3)$ admits a meromorphic continuation $\widetilde{\Psi}(W_1,W_2,h_3)$ to all $s\in\mathbb{C}$, satisfying
\begin{align*}
\widetilde{\Psi}(W_1,W_2,h_3)\propto \Lambda(1/2+s,\pi_1\times\pi_2\times\chi).
\end{align*}

\subsection{The Symmetric Spectral Reciprocity Formula}
\subsubsection{Automorphic Data} 
For $1\le j\le 4$, let $s_j\in\mathbb{C}$ and let $\chi_j$, $\omega_j$ be unitary Hecke characters over $F$. Suppose $\omega_1\omega_2=\omega_3\omega_4=\omega$.  Set
\begin{align*}
\pi_j:=\chi_j|\cdot|^{s_j}\boxplus \chi_j^{-1}\omega_j|\cdot|^{-s_j}.
\end{align*}
Let $h_j(\cdot,s_j)$ be a section of $\pi_j$, and let $W_j(\cdot,s_j)$ be the Whittaker function associated  with $h_j(\cdot,s_j)$ as defined in \eqref{f2.1}. Denote by 
\begin{align*}
\boldsymbol{h}:=(h_1(\cdot,s_1),h_2(\cdot,s_2),h_3(\cdot,s_3),h_4(\cdot,s_4)).
\end{align*}

For $g\in G(\mathbb{A}_F)$, define 
\begin{equation}\label{eq2.7}
R_{12}(g)\boldsymbol{h}:=(\pi_{1}(g)h_1(\cdot,s_1),\pi_{2}(g)h_2(\cdot,s_2),h_3(\cdot,s_3),h_4(\cdot,s_4)). 
\end{equation}

\subsubsection{Period Integrals: Cuspidal}
Let $\pi\in \mathcal{A}_0([G],\omega)$. We define 
\begin{align*}
\widetilde{\Psi}(\pi;\mathbf{s},\boldsymbol{h}):=\sum_{\phi\in\mathfrak{B}(\pi)}\widetilde{\Psi}(\overline{W_{\phi}},W_1(\cdot,s_1),h_2(\cdot,s_2))\widetilde{\Psi}(W_{\phi},\overline{W_3(\cdot,\overline{s_3})},\overline{h_4(\cdot,\overline{s_4})}),
\end{align*}
where $\mathfrak{B}(\pi)$ is an orthonormal basis of $\pi$. 

Let $\omega'=\omega_1\overline{\omega_3}$. For $\sigma\in \mathcal{A}_0([G],\omega')$, we define 
\begin{align*} 
\widetilde{\Psi}^*(\sigma;\mathbf{s},\boldsymbol{h}):=\sum_{\phi\in\mathfrak{B}(\sigma)}\widetilde{\Psi}(\overline{W_{\phi}},W_1(\cdot,s_1),\overline{h_3(\cdot,\overline{s_3})})\widetilde{\Psi}(W_{\phi}^*,W_2(\cdot,s_2),\overline{h_4(\cdot,\overline{s_4})}).
\end{align*}

\subsubsection{Period Integrals: Eisenstein Series}
Let $\mu\in \widehat{F^{\times}\backslash\mathbb{A}_F^{(1)}}$ and $\lambda\in i\mathbb{R}$.  
Let $\mathfrak{B}(\mu,\overline{\mu}\omega)$ be an orthonormal basis of $K$-finite vectors in the Hilbert space $H(\mu,\overline{\mu}\omega)$ defined in \eqref{f2.3}.  
Define  
\begin{multline}\label{c3.5}
\widetilde{\Psi}(\lambda,\mu;\mathbf{s},\boldsymbol{h}):=\frac{1}{2}\sum_{h\in \mathfrak{B}(\mu,\overline{\mu}\omega)}\widetilde{\Psi}(\overline{W_{E(\cdot,h,-\overline{\lambda})}},W_1(\cdot,s_1),h_2(\cdot,s_2))\\\widetilde{\Psi}(W_{E(\cdot,h,\lambda)},\overline{W_3(\cdot,\overline{s_3})},\overline{h_4(\cdot,\overline{s_4})}),
\end{multline}
and 
\begin{multline}\label{c3.5}
\widetilde{\Psi}^*(\lambda,\mu;\mathbf{s},\boldsymbol{h}):=\frac{1}{2}\sum_{h\in \mathfrak{B}(\mu,\overline{\mu}\omega')}\widetilde{\Psi}(\overline{W_{E(\cdot,h,-\overline{\lambda})}},W_1(\cdot,s_1),\overline{h_3(\cdot,\overline{s_3})})\\
\widetilde{\Psi}(W_{E(\cdot,h,\lambda)}^*,W_2(\cdot,s_2),\overline{h_4(\cdot,\overline{s_4})}),
\end{multline}
where $E(\cdot,h,\lambda)$ is as in \eqref{fc2.4}. Then $\widetilde{\Psi}(\lambda,\mu;\mathbf{s},\boldsymbol{h})$ and $\widetilde{\Psi}^*(\lambda,\mu;\mathbf{s},\boldsymbol{h})$ are meromorphic functions of $(\lambda,\mathbf{s})\in\mathbb{C}^5$.  

\subsubsection{The Spectral and the Dual Sides}\label{sec2.3.4}
Define 
\begin{align*}
&\mathcal{J}_{\mathrm{Spec}}^{\heartsuit}(\mathbf{s},\boldsymbol{h}):=\sum_{\pi\in \mathcal{A}_0([G],\omega)}\Psi(\pi;\mathbf{s},\boldsymbol{h})+\sum_{\substack{\mu\in \widehat{F^{\times}\backslash\mathbb{A}_F^{(1)}}}}\frac{1}{2\pi i}\int_{i\mathbb{R}}\Psi(\lambda,\mu;\mathbf{s},\boldsymbol{h})d\lambda,\\
&\mathcal{I}_{\Spec}^{\heartsuit}(\mathbf{s},\boldsymbol{h}):=\sum_{\sigma\in \mathcal{A}_0([G],\omega')}\widetilde{\Psi}^*(\sigma;\mathbf{s},\boldsymbol{h})+\sum_{\substack{\mu\in \widehat{F^{\times}\backslash\mathbb{A}_F^{(1)}}}}\frac{1}{2\pi i}\int_{i\mathbb{R}}\widetilde{\Psi}^*(\lambda,\mu;\mathbf{s},\boldsymbol{h})d\lambda.
\end{align*}

By \cite[\textsection 3]{Yan25}, the functions 
$\mathcal{J}_{\Spec}^{\heartsuit}(\mathbf{s},\boldsymbol{h})$ and 
$\mathcal{I}_{\Spec}^{\heartsuit}(\mathbf{s},\boldsymbol{h})$ admit meromorphic continuation to $(\lambda,\mathbf{s})\in\mathbb{C}^5$. 

\subsubsection{Degenerate Terms: the Geometric Side}
Let $\mathbf{s}\in \mathbb{C}^4$.  Define 
\begin{align*}
&\widetilde{\Psi}_{\mathrm{Geo}}^{(1)}(\mathbf{s},\boldsymbol{h}):=\widetilde{\Psi}(W_1(\cdot,s_1),\overline{W_3(\cdot,\overline{s_3})},h_2(\cdot,s_2)\overline{h_4(\cdot,\overline{s_4})}),\\
&\widetilde{\Psi}_{\mathrm{Geo}}^{(2)}(\mathbf{s},\boldsymbol{h}):=\widetilde{\Psi}(W_1(\cdot,s_1),\overline{W_3(\cdot,\overline{s_3})},h_2^{\Diamond}(\cdot,s_2)\overline{h_4(\cdot,\overline{s_4})}),\\
&\widetilde{\Psi}_{\mathrm{Geo}}^{(3)}(\mathbf{s},\boldsymbol{h}):=-\widetilde{\Psi}(W_1^*(\cdot,s_1),W_2(\cdot,s_2),\overline{h_3(\cdot,\overline{s_3})}\overline{h_4(\cdot,\overline{s_4})}),\\
&\widetilde{\Psi}_{\mathrm{Geo}}^{(4)}(\mathbf{s},\boldsymbol{h}):=-\widetilde{\Psi}(W_1^*(\cdot,s_1),W_2(\cdot,s_2),\overline{h_3^{\Diamond}(\cdot,\overline{s_3})}\overline{h_4(\cdot,\overline{s_4})}),\\
&\widetilde{\Psi}_{\mathrm{Geo}}^{(5)}(\mathbf{s},\boldsymbol{h}):=\widetilde{\Psi}(W_{h_1,\overline{h_3}}(\cdot,s_1,\overline{s_3}),W_2(\cdot,s_2),\overline{h_4(\cdot,\overline{s_4})}),\\
&\widetilde{\Psi}_{\mathrm{Geo}}^{(6)}(\mathbf{s},\boldsymbol{h}):=\widetilde{\Psi}(W_{h_1^{\Diamond},\overline{h_3}}(\cdot,s_1,\overline{s_3}),W_2(\cdot,s_2),\overline{h_4(\cdot,\overline{s_4})}),\\
&\widetilde{\Psi}_{\mathrm{Geo}}^{(7)}(\mathbf{s},\boldsymbol{h}):=-\widetilde{\Psi}(W_{h_1,h_2}(\cdot,s_1,s_2),\overline{W_3(\cdot,\overline{s_3})},\overline{h_4(\cdot,\overline{s_4})}),\\
&\widetilde{\Psi}_{\mathrm{Geo}}^{(8)}(\mathbf{s},\boldsymbol{h}):=-\widetilde{\Psi}(W_{h_1^{\Diamond},h_2}(\cdot,s_1,s_2),\overline{W_3(\cdot,\overline{s_3})},\overline{h_4(\cdot,\overline{s_4})}).
\end{align*}

\subsubsection{Degenerate Terms: the Spectral Side}

Let $\mathbf{s}\in \mathbb{C}^4$. We define 
\begin{align*}
&\Psi_{\mathrm{RS}}^{(1)}(\mathbf{s},\boldsymbol{h}):=\mathbf{1}_{\mu=\chi_1\chi_2}\underset{\lambda=s_2+s_1-\frac{1}{2}}{\Res}\ \widetilde{\Psi}(\lambda,\mu;\mathbf{s},\boldsymbol{h}),\\
&\Psi_{\mathrm{RS}}^{(2)}(\mathbf{s},\boldsymbol{h}):=-\mathbf{1}_{\mu=\omega\overline{\chi_1\chi_2}}\underset{\lambda=\frac{1}{2}-s_2-s_1}{\Res}\ \widetilde{\Psi}(\lambda,\mu;\mathbf{s},\boldsymbol{h}),\\
&\Psi_{\mathrm{RS}}^{(3)}(\mathbf{s},\boldsymbol{h}):=\mathbf{1}_{\mu=\overline{\chi_1}\omega_1\chi_2}\underset{\lambda=s_2-s_1-\frac{1}{2}}{\Res}\ \widetilde{\Psi}(\lambda,\mu;\mathbf{s},\boldsymbol{h}),\\
&\Psi_{\mathrm{RS}}^{(4)}(\mathbf{s},\boldsymbol{h}):=-\mathbf{1}_{\mu=\omega\chi_1\overline{\omega_1\chi_2}}\underset{\lambda=\frac{1}{2}+s_1-s_2}{\Res}\ \widetilde{\Psi}(\lambda,\mu;\mathbf{s},\boldsymbol{h}),\\
&\Psi_{\mathrm{RS}}^{(5)}(\mathbf{s},\boldsymbol{h}):=\mathbf{1}_{\mu=\omega\chi_3\omega_3^{-1}\chi_4^{-1}}\underset{\lambda=s_4-s_3-\frac{1}{2}}{\Res}\ \widetilde{\Psi}(\lambda,\mu;\mathbf{s},\boldsymbol{h}),\\
&\Psi_{\mathrm{RS}}^{(6)}(\mathbf{s},\boldsymbol{h}):=-\mathbf{1}_{\mu=\overline{\chi_3}\omega_3\chi_4}\underset{\lambda=\frac{1}{2}+s_3-s_4}{\Res}\ \widetilde{\Psi}(\lambda,\mu;\mathbf{s},\boldsymbol{h}),\\
&\Psi_{\mathrm{RS}}^{(7)}(\mathbf{s},\boldsymbol{h}):=\mathbf{1}_{\mu=\omega\chi_3^{-1}\chi_4^{-1}}\underset{\lambda=s_4+s_3-\frac{1}{2}}{\Res}\ \widetilde{\Psi}(\lambda,\mu;\mathbf{s},\boldsymbol{h}),\\
&\Psi_{\mathrm{RS}}^{(8)}(\mathbf{s},\boldsymbol{h}):=-\mathbf{1}_{\mu=\chi_3\chi_4}\underset{\lambda=\frac{1}{2}-s_3-s_4}{\Res}\ \widetilde{\Psi}(\lambda,\mu;\mathbf{s},\boldsymbol{h}).
\end{align*}

\subsubsection{Degenerate Terms: the Dual Side}

Let $\mathbf{s}\in \mathbb{C}^4$. We define 
\begin{align*}
&\Psi_{\mathrm{Dual}}^{(1)}(\mathbf{s},\boldsymbol{h}):=-\mathbf{1}_{\mu=\chi_1\chi_3^{-1}}\underset{\lambda=s_3+s_1-1/2}{\Res}\ \widetilde{\Psi}^*(\lambda,\mu;\mathbf{s},\boldsymbol{h}),\\
&\Psi_{\mathrm{Dual}}^{(2)}(\mathbf{s},\boldsymbol{h}):=\mathbf{1}_{\mu=\omega'\overline{\chi_1}\chi_3}\underset{\lambda=1/2-s_3-s_1}{\Res}\ \widetilde{\Psi}^*(\lambda,\mu;\mathbf{s},\boldsymbol{h}),\\
&\Psi_{\mathrm{Dual}}^{(3)}(\mathbf{s},\boldsymbol{h}):=-\mathbf{1}_{\mu=\overline{\chi_1}\omega_1\overline{\chi_3}}\underset{\lambda=s_3-s_1-1/2}{\Res}\ \widetilde{\Psi}^*(\lambda,\mu;\mathbf{s},\boldsymbol{h}),\\
&\Psi_{\mathrm{Dual}}^{(4)}(\mathbf{s},\boldsymbol{h}):=\mathbf{1}_{\mu=\omega'\chi_1\overline{\omega_1}\chi_3}\underset{\lambda=1/2+s_1-s_3}{\Res}\ \widetilde{\Psi}^*(\lambda,\mu;\mathbf{s},\boldsymbol{h}),\\
&\Psi_{\mathrm{Dual}}^{(5)}(\mathbf{s},\boldsymbol{h}):=-\mathbf{1}_{\mu=\omega'\chi_2\overline{\omega_2\chi_4}}\underset{\lambda=s_4-s_2-1/2}{\Res}\ \widetilde{\Psi}^*(\lambda,\mu;\mathbf{s},\boldsymbol{h}),\\
&\Psi_{\mathrm{Dual}}^{(6)}(\mathbf{s},\boldsymbol{h}):=\mathbf{1}_{\mu=\overline{\chi_2}\omega_2\chi_4}\underset{\lambda=1/2+s_2-s_4}{\Res}\ \widetilde{\Psi}^*(\lambda,\mu;\mathbf{s},\boldsymbol{h}),\\
&\Psi_{\mathrm{Dual}}^{(7)}(\mathbf{s},\boldsymbol{h}):=-\mathbf{1}_{\mu=\omega'\chi_2\overline{\chi_4}}\underset{\lambda=s_4+s_2-1/2}{\Res}\ \widetilde{\Psi}^*(\lambda,\mu;\mathbf{s},\boldsymbol{h}),\\
&\Psi_{\mathrm{Dual}}^{(8)}(\mathbf{s},\boldsymbol{h}):=\mathbf{1}_{\mu=\overline{\chi_2}\chi_4}\underset{\lambda=1/2-s_2-s_4}{\Res}\ \widetilde{\Psi}^*(\lambda,\mu;\mathbf{s},\boldsymbol{h}),
\end{align*}

\subsubsection{The Symmetric Reciprocity}
Define the region
\begin{multline*}
\mathcal{R}^{\heartsuit}:=\Big\{\mathbf{s}\in \mathbb{C}^4:\ \big|\Re(s_2)-|\Re(s_1)|\big|<1/2,\ \big|\Re(s_4)-|\Re(s_3)|\big|<1/2,\\
\big|\Re(s_3)-|\Re(s_1)|\big|<1/2,\ \big|\Re(s_4)-|\Re(s_2)|\big|<1/2\Big\}.
\end{multline*}

Let $\mathbf{s} \in \mathcal{R}^{\heartsuit}$. By \cite[Theorem A]{Yan25},
\begin{equation}\label{f2.8}
\mathcal{J}_{\mathrm{Spec}}^{\heartsuit}(\mathbf{s},\boldsymbol{h})=\ \mathcal{I}_{\Spec}^{\heartsuit}(\mathbf{s},\boldsymbol{h})+\sum_{*\in\{\mathrm{Geo},\mathrm{RS},\mathrm{Dual}\}}\sum_{i=1}^{8}\widetilde{\Psi}_{*}^{(i)}(\mathbf{s},\boldsymbol{h}).
\end{equation}

Although \eqref{f2.8} is sufficient for deriving uniform subconvexity for general $\GL_2$ $L$-functions (see \cite[\textsection 1.3]{Yan25}), it is not strong enough for Theorem \ref{thmA}. Our aim here is to refine this identity by inserting a projection onto the cuspidal spectrum with prescribed simple supercuspidal local components.

We also note that our notation differs slightly from that of loc. cit.: here we write $\boldsymbol{h}$ in place of $\mathfrak{X}$ in order to emphasize the dependence on the underlying automorphic data. This notation is more convenient for the purposes of the present paper, since we shall replace $\boldsymbol{h}$ by an average of $R_{12}(g)\boldsymbol{h}$, defined in \eqref{eq2.7}, so as to incorporate the desired spectral projection.

We now explain this construction. Let $\mathfrak{q}$ be a prime ideal, and let $v$ be the corresponding finite place of $F$. Let $f_{\mathfrak{q}}\in \mathcal{C}_c^{\infty}(\overline{G}(F_v))$.  
For any quantity $G(\mathbf{s},\boldsymbol{h})$, define
\begin{equation}\label{2.9}
G(\mathbf{s},R(f_{\mathfrak{q}})\boldsymbol{h})
:=\int_{\overline{G}(F_v)}
\overline{f_{\mathfrak{q}}(g)}
\,G(\mathbf{s},R_{12}(g^{-1})\boldsymbol{h})\,dg.
\end{equation}

Let $0<\varepsilon<10^{-3}$. Define 
\begin{align*}
\mathbf{B}_{\varepsilon}^4:=\big\{(s_1,s_2,s_3,s_4)\in \mathbb{C}^4:\ |s_1|\leq \varepsilon^2,\ |s_2|\leq 2\varepsilon^2,\ |s_3|\leq 5\varepsilon^2,\ |s_4|\leq 10\varepsilon^2\big\}.
\end{align*}
Let $\mathbf{S}_{\varepsilon}^4=\partial \mathbf{B}_{\varepsilon}^4$ be the boundary of $\mathbf{B}_{\varepsilon}^4$. Define
\begin{equation}\label{2.12}
\widetilde{\Psi}_{*}^{(i)}(\mathbf{0}\mid R(f_{\mathfrak{q}})\boldsymbol{h}):=\frac{1}{16\pi^4}\iiiint_{\mathbf{S}_{\varepsilon}^4}\frac{\widetilde{\Psi}_{*}^{(i)}(\mathbf{s},R(f_{\mathfrak{q}})\boldsymbol{h})}{s_1s_2s_3s_4}ds_1ds_2ds_3ds_4,
\end{equation}
where the subscripts $*\in \{\mathrm{Geo}, \mathrm{RS}, \mathrm{Dual}\}$. By \eqref{f2.8}, we obtain the following. 
\begin{restatable}[]{thm}{SP}\label{thm2.1}
Let notation be as above.  Then   
\begin{align*}
\mathcal{J}_{\mathrm{Spec}}^{\heartsuit}(\mathbf{0},R(f_{\mathfrak{q}})\boldsymbol{h})=\ \mathcal{I}_{\Spec}^{\heartsuit}(\mathbf{0},R(f_{\mathfrak{q}})\boldsymbol{h})+\sum_{*\in\{\mathrm{Geo},\mathrm{RS},\mathrm{Dual}\}}\sum_{i=1}^{8}\widetilde{\Psi}_{*}^{(i)}(\mathbf{0}\mid R(f_{\mathfrak{q}})\boldsymbol{h}).
\end{align*}
\end{restatable}

The test function $f_{\mathfrak{q}}$ will be constructed in \textsection\ref{sec3}, and the automorphic data $\boldsymbol{h}$ will be specified in \textsection\ref{sec6}.  
After evaluating or estimating each term in Theorem \ref{thm2.1}, we obtain Theorem \ref{thmA}.

\section{Representations and Whittaker Functions}\label{sec3}
In this section, we review the construction of simple supercuspidal representations, construct the test function $f_{\mathfrak{q}}$, and compute the relevant local period integrals in $\mathcal{J}_{\mathrm{Spec}}^{\heartsuit}(\mathbf{s},R(f_{\mathfrak{q}})\boldsymbol{h})$. 

\subsection{Simple Supercuspidal Representations}\label{sec3.1}
Let $v<\infty$ be a non-Archimedean place, and let $\pi_v$ be a simple supercuspidal representation of $\mathrm{GL}_2(F_v)$ with central character $\omega_v$. By \cite[Proposition 3.4]{Tun78}, we have $r_{\omega_v}\leq 1$. Define 
\begin{align*}
I_v(1):=\begin{pmatrix}
1+\mathfrak{p}_v & \mathcal{O}_v\\
\mathfrak{p}_v & 1+\mathfrak{p}_v
\end{pmatrix}.
\end{align*}

Let $t\in \mathbb{F}_v^{\times}$. Define the character
\begin{equation}\label{2.1}
\theta_t:\ \ I_v(1)\rightarrow \mathbb{C}^{\times},\ \ \ \begin{pmatrix}
1+\alpha\varpi_v & \beta\\
\gamma\varpi_v & 1+\delta\varpi_v 
\end{pmatrix}\mapsto \psi_v(\varpi_v^{-1}(\beta+t\gamma)).
\end{equation}

Let $\zeta\in \mathbb{C}^{\times}$ be a solution of $\zeta^2=\omega_v(t\varpi_v)$. Set $g_t:=\begin{pmatrix}
& t\\
\varpi_v
\end{pmatrix}$, and define
\begin{align*}
H_v:=Z(F_v)I_v(1)\bigcup g_tZ(F_v)I_v(1).
\end{align*}

We extend $\theta_t$ to a character on $H_v$ by setting
\begin{equation}\label{2.2}
\widetilde{\theta}_t(g_t^jzk):=\zeta^j\omega_v(z)\theta_t(k).
\end{equation}

By \cite[Proposition 5.3]{Kni25}, the compactly induced representation 
\begin{equation}\label{e3.3}
\sigma_t^{\zeta}:=\text{c-Ind}_{H_v}^{G(F_v)}\widetilde{\theta}_t
\end{equation}
is an
irreducible supercuspidal representation of conductor exponent $3$, with root number $\varepsilon(1/2,\sigma_t^{\zeta},\psi_v)=\zeta$, and  formal degree $\sigma_t^{\zeta}$ is $(q_v^2-1)/2$. Moreover, each $\pi_v$ is of the form $\sigma_t^{\zeta}$ for some $t$ and $\zeta$ as above. 

Let $h_v^{\zeta}:=\widetilde{\theta}_t\mathbf{1}_{H_v}\in \sigma_t^{\zeta}$. 
For $g\in G(F_v)$, let $d_v$ denote the valuation of the different ideal of $F_v$, and define
\begin{equation}\label{e3.4}
W_v(g)
:=q_v^{-\frac{1}{2}+\frac{d_v}{2}}
\int_{F_v}
h_v^{\zeta}\!\left(
\begin{pmatrix}
\varpi_v & \\
& 1
\end{pmatrix}
\begin{pmatrix}
1 & b\\
& 1
\end{pmatrix}
g
\right)
\overline{\psi_v(b)}db.
\end{equation}
Then $W_v$ is the Whittaker function associated with $h_v^{\zeta}$.

\begin{lemma}\label{lem3.1}
Let $y\in F_v^{\times}$ and $g=g_t^jzk'\in H_v$. Then 
\begin{equation}\label{e2.1}
W_v\left(\begin{pmatrix}
y & \\
& 1
\end{pmatrix}g\right)=\zeta^j\omega_v(z)q_v^{\frac{1}{2}}\mathbf{1}_{\varpi_v^{-1}(1+\mathfrak{p}_v)}(y)\theta_t(k').
\end{equation}
In particular, $\langle W_v,W_v\rangle=q_v^{-d_v/2}$. 
\end{lemma} 
\begin{proof}
By definition,
\begin{equation}\label{eq2.5}
W_v\left(\begin{pmatrix}
y & \\
& 1
\end{pmatrix}k'\right)=\omega_v(z)q_v^{-\frac{1}{2}+\frac{d_v}{2}}\int_{F_v}h_v\left(\begin{pmatrix}
\varpi_v y & \varpi_v b\\
& 1
\end{pmatrix}g_t^jk'\right)\overline{\psi_v(b)}db.
\end{equation}

Since $h_v^{\zeta}$ supports in $H_v$, we have
\[
h_v\left(\begin{pmatrix}
\varpi_v y & \varpi_v b\\
& 1
\end{pmatrix}g_t^jk'\right)\equiv 0
\]
unless
\begin{equation}\label{2.6}
\begin{pmatrix}
\varpi_v y & \varpi_v b\\
& 1
\end{pmatrix}\in Z(F_v)I_v(1)g_t^{-j}\bigcup g_tZ(F_v)I_v(1)g_t^{-j}.
\end{equation}

Observe that $g_tI_v(1)g_t^{-1}=I_v(1)$. Consequently, the condition \eqref{2.6} reduces to
\begin{align*}
\begin{pmatrix}
\varpi_v y & \varpi_v b\\
& 1
\end{pmatrix}\in Z(F_v)I_v(1),\ \ \text{that is},\ \ y\in \varpi_v^{-1}(1+\mathfrak{p}_v),\ b\in \mathfrak{p}_v^{-1}.
\end{align*}

As a result, we obtain
\begin{equation}\label{2.7}
h_v^{\zeta}\left(\begin{pmatrix}
\varpi_v y & \varpi_v b\\
& 1
\end{pmatrix}g_t^jk'\right)
=\zeta^j\psi_v(b)\mathbf{1}_{\varpi_v^{-1}(1+\mathfrak{p}_v)}(y)\mathbf{1}_{\mathfrak{p}_v^{-1}}(b).
\end{equation}

Therefore, \eqref{e2.1} follows immediately from \eqref{eq2.5} and \eqref{2.7}.
\end{proof}

For $g\in G(F_v)$, we define
\begin{equation}\label{2.3}
f_v^t(g)
:=\Vol(I_v(1))^{-1}\overline{\theta_t(g)}\cdot \mathbf{1}_{I_v(1)}(g). 
\end{equation}

By \cite[Theorem 4.4]{KL15}, we have the decomposition
\begin{align*}
\Ind_{Z(F_v)I_v(1)}^{\mathrm{GL}_2(F_v)}\widetilde{\theta}_t
\simeq \bigoplus_{\zeta:\ \zeta^2=\omega_v(t\varpi_v)} \sigma_t^{\zeta}.
\end{align*}

Consequently, for an irreducible admissible representation $\pi_v$ of
$\mathrm{GL}_2(F_v)$ with central character $\omega_v$, Frobenius reciprocity implies that
\begin{align*}
\rho_v(f_v^t)
:=\int_{G(F_v)} f_v^t(g)\rho_v(g)dg
\equiv 0
\end{align*}
unless $\rho_v$ occurs as a subrepresentation of
$\Ind_{Z(F_v)I_v(1)}^{\mathrm{GL}_2(F_v)}\widetilde{\theta}_t$, that is,
unless $\rho_v\simeq \sigma_t^{\zeta}$ for some $\zeta$ satisfying
$\zeta^2=\omega_v(t\varpi_v)$. Moreover, in this case the operator
$\rho_v(f_v^t)$ is the orthogonal projection onto the one-dimensional space
$h_v^{\zeta}=\theta_t\mathbf{1}_{H_v}\in \sigma_t^{\zeta}$.

\subsection{A Rankin-Selberg Convolution}
Let $I_v^*(1):=Z(\mathcal{O}_v)I_v(1)$. For $t_1, t_2\in F_v$, we define the function 
\begin{equation}\label{e3.9}
\Phi_{2,v}((t_1,t_2))=\Vol(I_v^*(1))^{-1}\mathbf{1}_{\mathfrak{p}_v}(t_1)\mathbf{1}_{\mathcal{O}_v^{\times}}(t_2)\psi_v(\varpi_v^{-2}tt_1t_2^{-1})\omega_v(t_2).
\end{equation}

Let $\pi_{1,v}$ be an irreducible unramified generic representation of $\mathrm{GL}_2(F_v)$. Let $W_{1,v}^{\circ}$ be a unit spherical vector in the Whittaker model of $\pi_{1,v}$. Let $\Re(s)\ggg 1$ and let $W_v$ be defined by \eqref{e3.4}. Define 
\begin{align*}
\Psi_v(s+1/2,\overline{W_v},W_{1,v}^{\circ},\Phi_{2,v}):=\int_{N(F_v)\backslash G(F_v)}W_{1,v}^{\circ}\left(x\begin{pmatrix}
\varpi_v\\
& 1
\end{pmatrix}\right)\\
\overline{W_v(x)}\Phi_{2,v}(\mathbf{e}_2x)|\det x|_v^{s+1/2}dx.	
\end{align*}

\begin{lemma}\label{lem3.2}
Let $|\Re(s)|<1/10$. Then 
\begin{equation}\label{e3.11}
\Psi_v(s+1/2,\overline{W_v},W_{1,v}^{\circ},\Phi_{2,v})=W_{1,v}^{\circ}(I_2)q_v^{s-d_v}L_v(1+2s,\omega_v^{-1})^{-1}.
\end{equation}
\end{lemma}
\begin{proof}
Write $x=\varpi_v^j\begin{pmatrix}
\varpi_v^i\\
& 1
\end{pmatrix}k$, where $k=\begin{pmatrix}
k_{11} & k_{12}\\
k_{21} & k_{22}
\end{pmatrix}\in K_v$. Then 
\begin{multline}\label{3.10}
\Psi_v(s+1/2,\overline{W_v},W_{1,v}^{\circ},\Phi_{2,v})=\Vol(I_v(1))^{-1}q_v^{-d_v}\sum_{j\in \mathbb{Z}}\omega_v^{-1}(\varpi_v^j)q_v^{-(1+2s)j}\\
\sum_{i\in \mathbb{Z}}\int_{K_v}\overline{W_v\left(\begin{pmatrix}
\varpi_v^i\\
& 1
\end{pmatrix}k\right)}W_{1,v}^{\circ}\left(\begin{pmatrix}
\varpi_v^i\\
& 1
\end{pmatrix}k\begin{pmatrix}
\varpi_v\\
& 1
\end{pmatrix}\right)\\
\mathbf{1}_{\mathfrak{p}_v}(k_{21})\mathbf{1}_{\mathcal{O}_v^{\times}}(k_{22})\psi_v(\varpi_v^{-2}tk_{21}k_{22}^{-1})\omega_v(k_{22})q_v^{(1/2-s)i}dk.
\end{multline}

If $k\in K_v$ satisfies
$\mathbf{1}_{\mathfrak{p}_v}(k_{21})\mathbf{1}_{\mathcal{O}_v^{\times}}(k_{22})=1$,
we may write
\begin{align*}
k=\begin{pmatrix}
\kappa_1\kappa_2 \\
& \kappa_2
\end{pmatrix}\begin{pmatrix}
1+\alpha\varpi_v & \beta\\
\gamma\varpi_v & 1+\delta\varpi_v 
\end{pmatrix},\ \ \kappa_1, \kappa_2\in \mathbb{F}_v^{\times},\ \alpha, \beta, \gamma, \delta\in \mathcal{O}_v. 
\end{align*}
Since $r_{\omega_v}\leq 1$, then 
\begin{equation}\label{e3.13}
\psi_v(\varpi_v^{-2}tk_{21}k_{22}^{-1})\omega_v(k_{22})=\psi_v(\varpi_v^{-1}t\gamma)\omega_v(\kappa_2).	
\end{equation}

By Lemma \ref{lem3.1} and this parametrization, we obtain 
\begin{equation}\label{3.11}
W_v\left(\begin{pmatrix}
\varpi_v^i\\
& 1
\end{pmatrix}k\right)=q_v^{\frac{1}{2}}\mathbf{1}_{1+\mathfrak{p}_v}(\kappa_1)\mathbf{1}_{i=-1}\psi_v(\varpi_v^{-1}(\beta+t\gamma))\omega_v(\kappa_2).
\end{equation}

Since $W_{1,v}^{\circ}$ is spherical, for $i=-1$ and $\kappa_1\in 1+\mathfrak{p}_v$, we have
\begin{equation}\label{3.12}
W_{1,v}^{\circ}\left(\begin{pmatrix}
\varpi_v^i\\
& 1
\end{pmatrix}k\begin{pmatrix}
\varpi_v\\
& 1
\end{pmatrix}\right)=\psi_v(\varpi_v^{-1}\beta)W_{1,v}^{\circ}(I_2). 
\end{equation}

Substituting \eqref{e3.13}, \eqref{3.11} and \eqref{3.12} into \eqref{3.10} yields 
\begin{multline}\label{3.14}
\Psi_v(s+1/2,\overline{W_v},W_{1,v}^{\circ},\Phi_{2,v})=\Vol(I_v^*(1))^{-1}q_v^{-d_v}L_v(1+2s,\omega_v^{-1})^{-1}\\
q_v^{\frac{1}{2}}\int_{I_v^*(1)}\overline{\psi_v(\varpi_v^{-1}(\beta+t\gamma))}\psi_v(\varpi_v^{-1}\beta)W_{1,v}^{\circ}(I_2)
\psi_v(\varpi_v^{-1}t\gamma)q_v^{-(1/2-s)}dk.
\end{multline}

The identity \eqref{e3.11} now follows immediately from \eqref{3.14}.
\end{proof}

\section{The $I_v(1)$-invariant Whittaker Functions}\label{sec4}
Let $v<\infty$ and let $\sigma_v$ be the $v$-th component of a unitary generic automorphic representation of $\mathrm{PGL}_2/F$. 
Suppose that $\sigma_v^{I_v(1)}\neq 0$. Then $\sigma_v$ is either a principal series or a special representation; in both cases, its conductor exponent is $\leq 2$.

\subsection{$\sigma_v$ is a principal series}
Suppose $\sigma_v$ is a principal series with $\sigma_v^{I_v(1)}\neq 0$. Then $\sigma_v=\xi_v\boxplus\xi_v^{-1}$ for some character $\xi_v$ with conductor exponent $r_{\xi_v}\leq 1$.

Let $\nu\in \mathbb{R}$ be such that $|\xi_v(a)|^2=\xi_v(a)\overline{\xi}_v(a)=|a|_v^{2\nu}$ for all $a\in F_v^{\times}$. Since $\sigma_v$ is the local component of a unitary generic automorphic representation of $\mathrm{PGL}_2/F$, it falls into one of the following cases:
\begin{itemize}
\item $\sigma_v=\xi_v\boxplus\xi_v^{-1}$ is tempered, i.e.\ $\xi_v$ is unitary with $r_{\xi_v}\leq 1$;
\item $\sigma_v=\xi_v\boxplus\xi_v^{-1}$ is a twisted complementary series representation, where $\xi_v=\chi_v|\cdot|_v^{\nu}$, with $\chi_v$ a unitary quadratic character and $\nu\in(0,1/2)$. By \cite{KS03}, we have $0<\nu\le 7/64$.
\end{itemize}

By Bruhat decomposition, we obtain $\dim\sigma_v^{I_v(1)}=2$. Let 
\begin{align*}
&f_1\left(g\right)=\begin{cases}
\xi_v(a)\xi_v^{-1}(d)|a|_v^{\frac{1}{2}}|d|_v^{-\frac{1}{2}},\ & \text{if $g\in \begin{pmatrix}
a& *\\
& d
\end{pmatrix}I_v(1)$}\\
0,\ & \text{otherwise},
\end{cases}\\
&f_2\left(g\right)=\begin{cases}
\xi_v(a)\xi_v^{-1}(d)|a|_v^{\frac{1}{2}}|d|_v^{-\frac{1}{2}},\ & \text{if $g\in \begin{pmatrix}
a& *\\
& d
\end{pmatrix}wI_v(1)$}\\
0,\ & \text{otherwise}
\end{cases}
\end{align*}
be normalized sections that are right $I_v(1)$-invariant. Then $f_1$ and $f_2$ form an orthogonal basis of $\sigma_v^{I_v(1)}$. For $j\in \{1,2\}$, define the Whittaker function 
\begin{equation}\label{e4.1.}
W_v^{(j)}(g)=\int_{F_v}f_j\left(w\begin{pmatrix}
1& b\\
& 1
\end{pmatrix}g\right)\overline{\psi_v(b)}db.
\end{equation}

Normalize these by
\begin{equation}\label{f4.2}
W_v^{(f_j)}(g):=\langle W_v^{(j)},W_v^{(j)}\rangle^{-1/2}W_v^{(j)}(g),\ \ g\in \mathrm{GL}_2(F_v).	
\end{equation}

Then $W_v^{(f_1)}$ and $W_v^{(f_2)}$ form an orthonormal basis of the Kirillov model of $\sigma_v^{I_v(1)}$. 

\begin{lemma}\label{lem3.3}
Let $\alpha\in \mathcal{O}_v^{\times}$ and $\gamma\in \mathcal{O}_v$. Then 
\begin{multline}\label{equ4.1}
W_v^{(1)}\left(\begin{pmatrix}
\alpha\varpi_v^i\\
\gamma & 1
\end{pmatrix}\right)=\bigg[\mathbf{1}_{\mathfrak{p}_v}(
\gamma)\xi_v(-\alpha)\int_{F_v}\xi_v^{-2}(b)|b|_v^{-1}\mathbf{1}_{e_v(b)<0}\overline{\psi_v(\varpi_v^ib)}db\\
+\xi_v(\gamma)\mathbf{1}_{\mathcal{O}_v^{\times}}(
\gamma)\int_{\mathcal{O}_v^{\times}}\xi_v^{-1}(b)\mathbf{1}_{\mathfrak{p}_v}(\alpha+b\gamma)\overline{\psi_v(\varpi_v^ib)}db\bigg]\cdot q_v^{-\frac{i}{2}}\xi_v^{-i}(\varpi_v).
\end{multline}
\end{lemma}
\begin{proof}
By the definition \eqref{e4.1.} and the change of variables $b\mapsto \varpi_v^i b$, we obtain
\begin{equation}\label{equ4.2}
W_v^{(1)}\left(\begin{pmatrix}
\alpha\varpi_v^i\\
\gamma & 1
\end{pmatrix}\right)
=\frac{\xi_v^{-1}(\varpi_v^i)}{q_v^{i/2}}
\int_{F_v}
f_1\left(
w\begin{pmatrix}
1 & b\\
& 1
\end{pmatrix}
\begin{pmatrix}
\alpha\\
\gamma & 1
\end{pmatrix}
\right)
\overline{\psi_v(\varpi_v^i b)}db.
\end{equation}

We analyze the integrand by considering the valuation of $b$.

\begin{itemize}
\item
Suppose that $e_v(b)<0$. Then
\begin{equation}\label{eq4.1}
w\begin{pmatrix}
1 & b\\
& 1
\end{pmatrix}
\begin{pmatrix}
\alpha\\
\gamma & 1
\end{pmatrix}
=
\begin{pmatrix}
-b^{-1} & 1\\
& b
\end{pmatrix}
\begin{pmatrix}
\alpha & \\
\alpha b^{-1}+\gamma & 1
\end{pmatrix}.
\end{equation}
By the definition of $f_1$, it follows that
\begin{equation}\label{eq4.2}
f_1\left(
w\begin{pmatrix}
1 & b\\
& 1
\end{pmatrix}
\begin{pmatrix}
\alpha\\
\gamma & 1
\end{pmatrix}
\right)
=\xi_v(-1)\xi_v(\alpha)\xi_v^{-2}(b)
|b|_v^{-1}
\mathbf{1}_{\mathfrak{p}_v}(\gamma)
\mathbf{1}_{e_v(b)<0}.
\end{equation}

\item
Suppose that $e_v(b)\geq 0$. Then
\begin{equation}\label{equ4.5}
w\begin{pmatrix}
1 & b\\
& 1
\end{pmatrix}
\begin{pmatrix}
\alpha\\
\gamma & 1
\end{pmatrix}
=
\begin{pmatrix}
\gamma & 1\\
\alpha+b\gamma & b
\end{pmatrix}
\in K_v.
\end{equation}
Consequently, by the definition of $f_1$, we obtain
\begin{equation}\label{eq4.3}
f_1\left(
w\begin{pmatrix}
1 & b\\
& 1
\end{pmatrix}
\begin{pmatrix}
\alpha\\
\gamma & 1
\end{pmatrix}
\right)
=\xi_v(\gamma)\xi_v^{-1}(b)
\mathbf{1}_{\mathcal{O}_v^{\times}}(\gamma)
\mathbf{1}_{\mathcal{O}_v^{\times}}(b)
\mathbf{1}_{\mathfrak{p}_v}(\alpha+b\gamma).
\end{equation}
\end{itemize}

Substituting \eqref{eq4.2} and \eqref{eq4.3} into \eqref{equ4.2} yields \eqref{equ4.1}, completing the proof.
\end{proof}

\begin{lemma}\label{lem3.4}
Let $\alpha\in \mathcal{O}_v^{\times}$ and $\gamma\in \mathcal{O}_v$. Then
\begin{multline*}
W_v^{(2)}\left(\begin{pmatrix}
\alpha\varpi_v^i\\
\gamma & 1
\end{pmatrix}\right)=\bigg[\xi_v(-1)\mathbf{1}_{\mathcal{O}_v^{\times}}(\gamma)\xi_v(\alpha)\int_{F_v}\xi_v^{-2}(b)|b|_v^{-1}\mathbf{1}_{e_v(b)<0}\overline{\psi_v(\varpi_v^ib)}db\\
+\xi_v(\alpha)\int_{\mathcal{O}_v}\xi_v^{-2}(\alpha+b\gamma)\mathbf{1}_{\mathcal{O}_v^{\times}}(\alpha+b\gamma)\overline{\psi_v(\varpi_v^ib)}db\bigg]\cdot q_v^{-\frac{i}{2}}\xi_v^{-i}(\varpi_v).
\end{multline*} 
\end{lemma}
\begin{proof}
Arguing as in \eqref{equ4.2}, we obtain 
\begin{equation}\label{equ4.6}
W_v^{(2)}\left(\begin{pmatrix}
\alpha\varpi_v^i\\
\gamma & 1
\end{pmatrix}\right)
=
\frac{\xi_v^{-i}(\varpi_v)}{q_v^{i/2}}
\int_{F_v}
f_2\left(
w\begin{pmatrix}
1 & b\\
& 1
\end{pmatrix}
\begin{pmatrix}
\alpha\\
\gamma & 1
\end{pmatrix}
\right)
\overline{\psi_v(\varpi_v^i b)}db.
\end{equation}

We analyze the integral by distinguishing cases according to the valuation of $b$.
\begin{itemize}
\item
Suppose that $e_v(b)<0$. By \eqref{eq4.1} and the definition of $f_2$, we have
\begin{equation}\label{e4.7}
f_2\left(
w\begin{pmatrix}
1 & b\\
& 1
\end{pmatrix}
\begin{pmatrix}
\alpha\\
\gamma & 1
\end{pmatrix}
\right)
=\xi_v(-1)\xi_v(\alpha)\xi_v^{-2}(b)
|b|_v^{-1}
\mathbf{1}_{\mathcal{O}_v^{\times}}(\gamma).
\end{equation}

\item
Suppose that $e_v(b)\geq 0$. By \eqref{equ4.5} and the definition of $f_2$, we have
\begin{align*}
f_2\left(
w\begin{pmatrix}
1 & b\\
& 1
\end{pmatrix}
\begin{pmatrix}
\alpha\\
\gamma & 1
\end{pmatrix}
\right)\equiv 0
\end{align*}
unless $\alpha+b\gamma\in \mathcal{O}_v^{\times}$. Assume henceforth that $\alpha+b\gamma\neq 0$.
A straightforward calculation shows that
\begin{align*}
\begin{pmatrix}
\gamma & 1\\
\alpha+b\gamma & b
\end{pmatrix}
=
\begin{pmatrix}
\alpha(\alpha+b\gamma)^{-1} & \gamma \\
& \alpha+b\gamma
\end{pmatrix}
w
\begin{pmatrix}
1 & b(\alpha+b\gamma)^{-1}\\
& 1
\end{pmatrix}.
\end{align*}

Therefore, under the conditions $\alpha+b\gamma\in \mathcal{O}_v^{\times}$ and $b\in \mathcal{O}_v$, we obtain
\begin{equation}\label{e4.8}
f_2\left(
w\begin{pmatrix}
1 & b\\
& 1
\end{pmatrix}
\begin{pmatrix}
\alpha\\
\gamma & 1
\end{pmatrix}
\right)
=
\xi_v(\alpha)\xi_v^{-2}(\alpha+b\gamma)
\mathbf{1}_{\mathcal{O}_v^{\times}}(\alpha+b\gamma).
\end{equation}
\end{itemize}

Combining \eqref{equ4.6}, \eqref{e4.7}, and \eqref{e4.8} yields Lemma \ref{lem3.4}, as claimed.
\end{proof}

As a consequence of Lemmas \ref{lem3.3} and \ref{lem3.4}, we obtain 
\begin{equation}\label{f4.12} 
\langle W_v^{(1)},W_v^{(1)}\rangle^{-1/2}\asymp_{F_v} q_v^{\frac{1}{2}+2\nu},\ \ \langle W_v^{(2)},W_v^{(2)}\rangle^{-1/2}\asymp_{F_v} 1.
\end{equation}

\subsection{$\sigma_v$ is special}\label{sec4.2} 
Suppose $\sigma_v$ is a special representation with $\sigma_v^{I_v(1)}\neq 0$. Then $\sigma_v=\mathrm{St}\otimes\chi_v$ for some character $\chi_v$ satisfying $\chi_v^2=\mathbf{1}$ and $r_{\chi_v}\leq 1$. In this case, we have $\dim \sigma_v^{I_v(1)}=1$. Let $W_v^{\circ}$ be a unit local new vector in the Whittaker model of $\mathrm{St}$. Then 
\begin{align*}
W_{v}^{\chi_v}(g):=\chi_v(\det g)W_v^{\circ}(g),\ \ g\in \mathrm{PGL}_2(F_v)
\end{align*}
generates the space of right $I_v(1)$-invariant Whittaker functions of $\sigma_v$. In particular, we have, for $x\in F_v^{\times}$, that 
\begin{equation}\label{4.11}
W_v^{\chi_v}\left(\begin{pmatrix}
x\\
& 1
\end{pmatrix}\right)=\chi_v(x)|x|_v\mathbf{1}_{\mathcal{O}_v}(x)W_v^{\circ}(I_2).
\end{equation}
Notice that $W_v^{\chi_v}$ is the local new vector only when $\chi_v$ is unramified.

\begin{lemma}\label{lem4.3}
Let $i\in \mathbb{Z}$. We have 
\begin{equation}\label{4.20}
W_v^{\circ}\left(\begin{pmatrix}
\varpi_v^i\\
& 1
\end{pmatrix}w\right)=-q_v^{-i-1}W_v^{\circ}(I_2)\mathbf{1}_{i\geq -1}.
\end{equation}
\end{lemma}
\begin{proof}
Let $g=\diag(\varpi_v^i,1)$ and $\gamma'\in \mathbb{F}_v$. We have 
\begin{align*}
W_v^{\circ}\left(gw\right)=W_v^{\circ}\left(gw\begin{pmatrix}
1\\
\gamma'\varpi_v & 1
\end{pmatrix}\right)=W_v^{\circ}\left(g\begin{pmatrix}
1 & \gamma'\varpi_v \\
& 1
\end{pmatrix}w\right)=\psi_v(\varpi_v^{i+1}\gamma')W_v^{\circ}\left(gw\right).
\end{align*}
As a result, we obtain 
\begin{equation}\label{e4.15}
W_v^{\circ}\left(\begin{pmatrix}
\varpi_v^i\\
& 1
\end{pmatrix}w\right)\equiv 0\ \ \text{unless}\ \ i\geq -1. 
\end{equation}

Let $j_{\mathrm{St}}$ be the Bessel function associated with $\mathrm{St}$, e.g., see \cite{Cog14}. Then for all $g\in \mathrm{GL}_2(F_v)$, we have  
\begin{align*}
W_v^{\circ}\left(\begin{pmatrix}
\varpi_v^i\\
& 1
\end{pmatrix}w\right)=\int_{F_v^{\times}}j_{\mathrm{St}}(y)W_v^{\circ}\left(\begin{pmatrix}
\varpi_v^{-i}y\\
&1
\end{pmatrix}\right)d^{\times}y.
\end{align*}

Making use of $W_v^{\circ}(\diag(x,1))=|x|_v\mathbf{1}_{\mathcal{O}_v}(x)W_v^{\circ}(I_2)$, we obtain 
\begin{equation}\label{4.15}
W_v^{\circ}\left(\begin{pmatrix}
\varpi_v^i\\
& 1
\end{pmatrix}w\right)=q_v^iW_v^{\circ}(I_2)\sum_{n=i}^{\infty}q_v^{-n}\int_{\varpi_v^n\mathcal{O}_v^{\times}}j_{\mathrm{St}}(y)d^{\times}y.
\end{equation}

By the local functional equation, for $-\Re(s)\ggg 1$ we have 
\begin{equation}\label{4.16}
\gamma_v(1/2+s,\mathrm{St},\psi_v)
=\int_{F_v^{\times}}j_{\mathrm{St}}(y)|y|_v^{-s}d^{\times}y=\sum_{n\in\mathbb{Z}}q_v^{ns}\int_{\varpi_v^n\mathcal{O}_v^{\times}}j_{\mathrm{St}}(y)d^{\times}y,
\end{equation}
where $\gamma_v(1/2+s,\mathrm{St},\psi_v)$ is the $\gamma$-factor associated with $\mathrm{St}$ relative to the unramified character $\psi_v$, defined by  
\begin{align*}
\gamma_v(1/2+s,\mathrm{St},\psi_v)=\frac{\varepsilon_v(1/2+s,\mathrm{St},\psi_v)L_v(1/2-s,\mathrm{St})}{L_v(1/2+s,\mathrm{St})}.
\end{align*}

Since $\varepsilon_v(1/2+s,\mathrm{St},\psi_v)=-q_v^{-s}$ and $L_v(s,\mathrm{St})=\zeta_v(s+1/2)$, we obtain 
\begin{equation}\label{4.17}
\gamma_v(1/2+s,\mathrm{St},\psi_v)=-\frac{1-q_v^{-s-1}}{q_v^{s}(1-q_v^{s-1})}=-q_v^{-s}(1-q_v^{-s-1})\sum_{n=0}^{\infty}q_v^{n(s-1)}.
\end{equation}

Comparing \eqref{4.16} and \eqref{4.17} yields 
\begin{equation}\label{4.18}
\int_{\varpi_v^n\mathcal{O}_v^{\times}}j_{\mathrm{St}}(y)d^{\times}y=-q_v^{-n-1}\mathbf{1}_{n\geq -1}+q_v^{-n-3}\mathbf{1}_{n=-2}.
\end{equation}

Substituting \eqref{e4.15} and \eqref{4.18} into \eqref{4.15} gives \eqref{4.20}.  
\end{proof}

\begin{lemma}
Let $i\in \mathbb{Z}$, $\alpha\in \mathcal{O}_v^{\times}$, and $\gamma\in \mathcal{O}_v$. 
\begin{itemize}
\item Suppose $\gamma\in \mathfrak{p}_v$. Then 
\begin{equation}\label{4.12}
W_v^{\chi_v}\left(\begin{pmatrix}
\alpha\varpi_v^i\\
\gamma & 1
\end{pmatrix}\right)=\chi_v(\alpha\varpi_v^i)q_v^{-i}W_v^{\circ}(I_2)\mathbf{1}_{i\geq 0}.
\end{equation}
\item Suppose $\gamma\in \mathcal{O}_v^{\times}$. Then 
\begin{equation}\label{4.13}
W_v^{\chi_v}\left(\begin{pmatrix}
\alpha\varpi_v^i\\
\gamma & 1
\end{pmatrix}\right)=-\psi_v(\varpi_v^i\alpha\gamma^{-1})\chi_v(\alpha\varpi_v^i)q_v^{-i-1}W_v^{\circ}(I_2)\mathbf{1}_{i\geq -1}.
\end{equation}
\end{itemize} 	
\end{lemma}
\begin{proof}
The formula \eqref{4.12} follows from the definition. We thus proceed to assume $\gamma\in \mathcal{O}_v^{\times}$. Notice that 
\begin{align*}
\begin{pmatrix}
\alpha\varpi_v^i\\
\gamma & 1
\end{pmatrix}=\begin{pmatrix}
\alpha\\
 & 1
\end{pmatrix}\begin{pmatrix}
1& \varpi_v^i\gamma^{-1}\\
 & 1
\end{pmatrix}\begin{pmatrix}
\varpi_v^i\\
& 1
\end{pmatrix}
w\begin{pmatrix}
\gamma & 1\\
 & -\gamma^{-1}
\end{pmatrix}.
\end{align*}

As a consequence, we derive 
\begin{equation}\label{4.14}
W_v^{\chi_v}\left(\begin{pmatrix}
\varpi_v^i\\
\gamma & 1
\end{pmatrix}\right)=\psi_v(\varpi_v^i\alpha\gamma^{-1})\chi_v(\alpha\varpi_v^i)W_v^{\circ}\left(\begin{pmatrix}
\varpi_v^i\\
& 1
\end{pmatrix}w\right).
\end{equation}

Therefore, \eqref{4.13} follows from \eqref{4.14} and Lemma \ref{lem4.3}. 
\end{proof}

\section{Ramified Local Period Integrals}\label{sec5}
In this section, we compute certain local period integrals in the ramified case, which will be used in \textsection\ref{sec8} to establish an upper bound for 
$\mathcal{I}_{\Spec}^{\heartsuit}(\mathbf{0},R(f_{\mathfrak{q}})\boldsymbol{h})$.

Let $v<\infty$, and let $\sigma_v$ be an irreducible unitary generic
representation of $\mathrm{PGL}_2(F_v)$. Let $W_v$ be a vector
in the Whittaker model of $\sigma_v$.
Throughout this section, we fix the following matrix expressions:
\begin{equation}\label{eq5.1}
k'=\begin{pmatrix}
1+\alpha\varpi_v & \beta\\
\gamma\varpi_v & 1+\delta\varpi_v 
\end{pmatrix}\in I_v(1),\ \ \ k=\begin{pmatrix}
k_{11} & k_{12}\\
k_{21} & k_{22}
\end{pmatrix}\in K_v.
\end{equation} 
\subsection{The Whittaker Function $W_{2,v}$}
Let $\Phi_{2,v}$ be the function defined in \eqref{e3.9}.
For $\Re(s_2)>0$, we define
\begin{equation}\label{f5.2}
h_{2,v}(x)
:=|\det x|^{\frac{1}{2}}
\int_{F_v^{\times}}
\Phi_{2,v}(z\mathbf{e}_2 x)
\omega_v^{-1}(z)
|z|_v 
d^{\times}z,
\end{equation}
which is the Godement section in $\pi_{2,v}:=\mathbf{1}\boxplus \omega_v$ associated with $\Phi_{2,v}$.  
Let
\begin{equation}\label{e9.1}
W_{2,v}(x)
=\int_{F_v}
h_{2,v}\left(
w\begin{pmatrix}
1 & b\\
& 1
\end{pmatrix}
x\right)
\overline{\psi_v(b)}db
\end{equation}
denote the corresponding Whittaker function.

\begin{lemma}\label{lem4.1}
Let $y\in F_v^{\times}$ and $k=\begin{pmatrix}
k_{11} & k_{12}\\
k_{21} & k_{22} 
\end{pmatrix}\in I_v^*(1)$. Then 
\begin{multline}\label{e9.2}
W_{2,v}\left(\begin{pmatrix}
y\\
& 1
\end{pmatrix}k\right)
=\psi_v(\varpi_v^{-2}tk_{21}k_{22}^{-1})\omega_v(k_{22})\zeta_v(1)^{-1}
\Vol(I_v^*(1))^{-1}
|y|_v^{\frac{1}{2}}\\
\sum_{l=-1}^{e_v(y)-1}\omega_v^l(\varpi_v)\int_{\mathcal{O}_v^{\times}}
\psi_v(\varpi_v^{-l-2}ty\kappa^{-1}-\varpi_v^{l}\kappa )\omega_v(\kappa)
d^{\times}\kappa.
\end{multline}
\end{lemma}
\begin{proof}
By the definition \eqref{e9.1}, 
\begin{align*}
W_{2,v}\left(\begin{pmatrix}
y\\
& 1
\end{pmatrix}k\right)
=
|y|_v^{\frac{1}{2}}
\int_{F_v^{\times}}\int_{F_v}
\Phi_{2,v}((zy,zb)k)
\overline{\psi_v(b)}db
\omega_v^{-1}(z)|z|_vd^{\times}z.
\end{align*}

Write $k=k_{22}\begin{pmatrix}
1+\alpha\varpi_v & \beta\\
\gamma\varpi_v & 1 
\end{pmatrix}\in I_v^*(1)$. We have
$z(y,b)k\in \mathfrak{p}_v\oplus \mathcal{O}_v^{\times}$ if and only if
$z(y,b)\in \mathfrak{p}_v\oplus \mathcal{O}_v^{\times}$.
Moreover, we have 
\begin{align*}
z(y,b)k
=
z\bigl(y(1+\alpha\varpi_v)+b\gamma\varpi_v,
y\beta+b(1+\delta\varpi_v)\bigr),
\end{align*}
from which it follows that
\begin{multline}\label{e4.4}
\Phi_{2,v}((zy,zb)k')
=\omega_v(k_{22})
\Vol(I_v^*(1))^{-1}\mathbf{1}_{\mathfrak{p}_v}(zy)\mathbf{1}_{\mathcal{O}_v^{\times}}(zb)\\
\psi_v\left(\varpi_v^{-2}t\frac{y(1+\alpha\varpi_v)+b\gamma\varpi_v}{y\beta+b(1+\delta\varpi_v)}\right)\omega_v(z(y\beta+b(1+\delta\varpi_v))).
\end{multline}

After a simplification, \eqref{e4.4} amounts to 
\begin{align*}
\Phi_{2,v}((zy,zb)k')
=\psi_v(\varpi_v^{-1}t\gamma)
\Vol(I_v^*(1))^{-1}\mathbf{1}_{\mathfrak{p}_v}(zy)\mathbf{1}_{\mathcal{O}_v^{\times}}(zb)\psi_v(\varpi_v^{-2}tyb^{-1})\omega_v(zbk_{22}).
\end{align*}

Substituting this into the preceding integral yields
\begin{multline}\label{e9.3}
W_{2,v}\left(\begin{pmatrix}
y\\
& 1
\end{pmatrix}k'\right)
=\psi_v(\varpi_v^{-1}t\gamma)\omega_v(k_{22})
\Vol(I_v^*(1))^{-1}
|y|_v^{\frac{1}{2}}\\
\int_{F_v^{\times}}\mathbf{1}_{\mathfrak{p}_v}(z^{-1}y)\int_{\mathcal{O}_v^{\times}}
\psi_v(\varpi_v^{-2}tz^{-1}yb^{-1}-zb)\omega_v(b)db\omega_v(z)
d^{\times}z.
\end{multline}

A straightforward calculation now shows that \eqref{e9.2} follows from \eqref{e9.3} and the change of variable $z\mapsto b^{-1}z$. 
\end{proof}

\begin{lemma}\label{lem4.2}
Let $k'=\begin{pmatrix}
1+\alpha\varpi_v & \beta\\
\gamma\varpi_v & 1+\delta\varpi_v 
\end{pmatrix}\in I_v(1)$,  $k=\begin{pmatrix}
k_{11} & k_{12}\\
k_{21} & k_{22}
\end{pmatrix}\in K_v$, and $j\in \mathbb{Z}$. Then 
\begin{equation}\label{e9.4}
\Phi_{2,v}(\varpi_v^j\mathbf{e}_2kk')
=\frac{\psi_v(
\varpi_v^{-1}t\gamma)
\mathbf{1}_{j=0}\mathbf{1}_{\mathfrak{p}_v}(k_{21})
\psi_v(
\varpi_v^{-2}tk_{21}k_{22}^{-1})\omega_v(k_{22})}{\Vol(I_v^*(1))}.
\end{equation}
\end{lemma}
\begin{proof}
By definition of $\Phi_{2,v}$, we have
\begin{align*}
\Phi_{2,v}(\varpi_v^j\mathbf{e}_2kk')\equiv 0
\quad\text{unless}\quad
\varpi_v^j\mathbf{e}_2k\in \mathfrak{p}_v\oplus \mathcal{O}_v^{\times}.
\end{align*}
In this case, a direct computation shows that 
\begin{align*}
\Phi_{2,v}(\varpi_v^j\mathbf{e}_2kk')
=\frac{\mathbf{1}_{\mathfrak{p}_v}(k_{21})
\mathbf{1}_{j=0}\omega_v(k_{22})
\psi_v\!\left(
\varpi_v^{-2}t\frac{k_{21}(1+\alpha\varpi_v)+\gamma k_{22}\varpi_v}{k_{21}\beta+k_{22}(1+\delta\varpi_v)}
\right)}{\Vol(I_v^*(1))}.
\end{align*}
A simplification of the above fraction in $\psi_v(\cdot)$ yields \eqref{e9.4}.
\end{proof}

\subsection{The Integral $\mathcal{I}_{v}(W_v;k')$}\label{sec5.2}
We define the integral
\begin{equation}\label{e4.6}
\mathcal{I}_{v}(W_v;k'):=\int_{N(F_v)\backslash G(F_v)}W_v(x)\overline{W_{2,v}(x)}\Phi_{2,v}(\mathbf{e}_2xk')|\det x|_v^{\frac{1}{2}}dx.
\end{equation}

It follows from the Iwasawa decomposition and Lemmas \ref{lem4.1} and \ref{lem4.2} that 
\begin{equation}\label{f5.9}
\mathcal{I}_{v}(W_v;k')\equiv 0\ \ \text{unless} \ \ \sigma_v^{I_v(1)}\neq 0.
\end{equation}

\subsubsection{$\sigma_v$ is a principal series}
Suppose $\sigma_v=\xi_v\boxplus\xi_v^{-1}$ is a principal series with $r_{\xi_v}\leq 1$. From the discussion in \textsection\ref{sec4}, we may take $W_v\in \Big\{W_v^{(1)}, W_v^{(2)}\Big\}$, which is defined by \eqref{e4.1.}.  

\begin{lemma}\label{lem5.3}
 We have the following. 
\begin{itemize}
\item Suppose $r_{\xi_v}=1$ and $r_{\overline{\omega}_v\xi_v}=1$. Then
\begin{multline}\label{5.13}
\mathcal{I}_{v}(W_v^{(1)};k')=\omega_v(\varpi_v)\xi_v^{-1}(t)\psi_v(
\varpi_v^{-1}t\gamma)\xi_v^{2}(\varpi_v)q_v^{-2-d_v}\Vol(I_v^*(1))^{-1}\\
G(\xi_v,\psi_v)G(\overline{\omega}_v\xi_v,\psi_v)\Big[\mathbf{1}_{r_{\xi_v^2}=1}G(\overline{\xi}_v^2,\psi_v)-\mathbf{1}_{r_{\xi_v^2}=0}\Big].
\end{multline}

\item Suppose $r_{\xi_v}=1$ and $r_{\overline{\omega}_v\xi_v}=0$. Then
\begin{multline}\label{5.19}
\mathcal{I}_{v}(W_v^{(1)};k')=\psi_v(
\varpi_v^{-1}t\gamma)\Vol(I_v^*(1))^{-1}\omega_v\xi_v^{2}(\varpi_v)q_v^{-\frac{1}{2}-d_v}\overline{\xi}_v(-t)G(\xi_v,\overline{\psi}_v)\\
\bigg[\overline{\omega}_v\xi_v(\varpi_v)\zeta_v(1)^{-1}q_v^{-1}\Big[\mathbf{1}_{r_{\overline{\xi}_v^2}=1}G(\overline{\xi}_v^2,\overline{\psi}_v)
-\mathbf{1}_{r_{\overline{\xi}_v^2}=0}\Big]L(1/2,\overline{\omega}_v\xi_v)+q_v^{-\frac{3}{2}}\mathbf{1}_{r_{\overline{\xi}_v^2}=0}\\
+\zeta_v(1)^{-2}\overline{\omega}_v\xi_v^{-1}(\varpi_v)L(1/2,\overline{\omega}_v\xi_v^{-1})L(1/2,\overline{\omega}_v\xi_v)\mathbf{1}_{r_{\overline{\xi}_v^2}=0}
-
q_v^{-\frac{3}{2}}G(\overline{\xi}_v^2,\overline{\psi}_v)\mathbf{1}_{r_{\overline{\xi}_v^2}=1}
\bigg].
\end{multline}

\item Suppose $r_{\xi_v}=0$ and $r_{\omega_v}=1$. Then 
\begin{multline}\label{5.21}
\mathcal{I}_{v}(W_v^{(1)};k')=\xi_v(-1)\psi_v(
\varpi_v^{-1}t\gamma)\Vol(I_v^*(1))^{-1}\omega_v\xi_v(\varpi_v)q_v^{-\frac{1}{2}-d_v}\\
G(\overline{\omega}_v\xi_v,\psi_v)\Big[\zeta_v(1)^{-2}L(1/2,\xi_v^{-1})L(1/2,\xi_v)\\
-\zeta_v(1)^{-1}\xi_v^{2}(\varpi_v)q_v^{-1}L(1/2,\xi_v)
+
\overline{\xi}_v(t)\xi_v(\varpi_v)q_v^{-\frac{3}{2}}\Big].
\end{multline}

\item Suppose $r_{\xi_v}=r_{\omega_v}=0$. Then 
\begin{multline}\label{5.25}
\mathcal{I}_{v}(W_v^{(1)};k')=\xi_v(-1)\psi_v(
\varpi_v^{-1}t\gamma)\Vol(I_v^*(1))^{-1}\zeta_v(1)^{-1}\omega_v(\varpi_v)q_v^{-d_v}\\
\sum_{i\geq 0}q_v^{-\frac{i}{2}}\xi_v^{-i}(\varpi_v)\bigg[\zeta_v(1)^{-1}\mathbf{1}_{i\geq 1}\sum_{m=1}^{i}\xi_v^{2m}(\varpi_v)-\xi_v^{2i+2}(\varpi_v)q_v^{-1}\bigg]\\
\Big[q_v\zeta_v(1)^{-1}\mathbf{1}_{i\geq 2}\sum_{l=0}^{i-2}\overline{\omega}_v^{l+1}(\varpi_v)-(\overline{\xi}_v(t)\overline{\omega}_v^{i}(\varpi_v) 
-1)\mathbf{1}_{i\geq 1}+q_v^{-1}\zeta_v(1)\overline{\xi}_v(t)\Big].
\end{multline}
\end{itemize}
\end{lemma}
\begin{proof}
By the Iwasawa decomposition, 
\begin{multline}\label{5.9}
\mathcal{I}_{v}(W_v^{(1)};k')=\sum_{j\in \mathbb{Z}}\sum_{i\in \mathbb{Z}}q_v^{-j}q_v^{i/2}\\
\int_{K_v}W_v^{(1)}\left(\begin{pmatrix}
\varpi_v^i\\
& 1
\end{pmatrix}k\right)\overline{W_{2,v}\left(\begin{pmatrix}
\varpi_v^i\\
& 1
\end{pmatrix}k\right)}\Phi_{2,v}(\varpi_v^j\mathbf{e}_2kk')dk.
\end{multline}

By Lemmas \ref{lem4.1} and \ref{lem4.2}, together with the decomposition 
\begin{equation}\label{eq5.13}
K_{0,v}[1]=\bigsqcup_{\alpha\in \mathbb{F}_v^{\times}}\begin{pmatrix}
\alpha\\
& 1
\end{pmatrix}
I_v^*(1),
\end{equation}
we derive from \eqref{5.9} that 
\begin{multline}\label{5.10}
\mathcal{I}_{v}(W_v^{(1)};k')=\psi_v(
\varpi_v^{-1}t\gamma)\Vol(I_v^*(1))^{-1}\zeta_v(1)^{-1}
\sum_{i\in \mathbb{Z}}\sum_{\alpha\in \mathbb{F}_v^{\times}}\\
W_v^{(1)}\left(\begin{pmatrix}
\alpha\varpi_v^i\\
& 1
\end{pmatrix}\right)
\sum_{l=-1}^{i-1}\overline{\omega}_v^l(\varpi_v)\int_{\mathcal{O}_v^{\times}}
\overline{\psi}_v(\varpi_v^{i-l-2}t\alpha\kappa^{-1}-\varpi_v^{l}\kappa)\overline{\omega}_v(\kappa)
d^{\times}\kappa. 
\end{multline}

By Lemma \ref{lem3.3}, we have the explicit formula 
\begin{equation}\label{5.11}
W_v^{(1)}\left(\begin{pmatrix}
\alpha\varpi_v^i\\
 & 1
\end{pmatrix}\right)=\frac{\xi_v(-\alpha)}{\xi_v^{i}(\varpi_v)q_v^{\frac{i}{2}}}\int_{F_v}\xi_v^{-2}(b)|b|_v^{-1}\mathbf{1}_{e_v(b)<0}\overline{\psi_v(\varpi_v^ib)}db.
\end{equation}

Substituting \eqref{5.11} into \eqref{5.10}, and applying the change of variables $\alpha\mapsto \kappa\alpha$, we obtain
\begin{multline}\label{5.12}
\mathcal{I}_{v}(W_v^{(1)};k')=\xi_v(-1)\psi_v(
\varpi_v^{-1}t\gamma)\Vol(I_v^*(1))^{-1}\zeta_v(1)^{-1}\\
\sum_{i\in \mathbb{Z}}q_v^{-\frac{i}{2}}\xi_v^{-i}(\varpi_v)
\int_{F_v}\xi_v^{-2}(b)|b|_v^{-1}\mathbf{1}_{e_v(b)<0}\overline{\psi_v(\varpi_v^ib)}db\\
\sum_{l=-1}^{i-1}\overline{\omega}_v^l(\varpi_v)\sum_{\alpha\in \mathbb{F}_v^{\times}}\xi_v(\alpha)
\overline{\psi}_v(\varpi_v^{i-l-2}t\alpha)\int_{\mathcal{O}_v^{\times}}\psi_v(\varpi_v^{l}\kappa)\overline{\omega}_v\xi_v(\kappa)
d^{\times}\kappa.
\end{multline}

We now analyze \eqref{5.12} by distinguishing cases according to the values of $r_{\xi_v}$ and $r_{\overline{\omega}_v\xi_v}$.
\begin{itemize}
\item Suppose that $r_{\xi_v}=1$ and $r_{\overline{\omega}_v\xi_v}=1$. Then  
\begin{equation}\label{e5.13}
\sum_{\alpha\in \mathbb{F}_v^{\times}}\xi_v(\alpha)
\overline{\psi}_v(\varpi_v^{i-l-2}t\alpha)=\overline{\xi}_v(t)\mathbf{1}_{l=i-1}G(\xi_v,\overline{\psi}_v),
\end{equation}
and 
\begin{equation}\label{e5.14}
\int_{\mathcal{O}_v^{\times}}\psi_v(\varpi_v^{l}\kappa)\overline{\omega}_v\xi_v(\kappa)
d^{\times}\kappa=q_v^{-1-\frac{d_v}{2}}\zeta_v(1)\mathbf{1}_{l=-1}G(\overline{\omega}_v\xi_v,\psi_v).
\end{equation}

Notice that $0\leq r_{\xi_v^2}\leq r_{\xi_v}=1$. By a straightforward calculation, 
\begin{equation}\label{e5.15}
\int_{F_v}\frac{\xi_v^{-2}(b)\mathbf{1}_{e_v(b)<0}\overline{\psi_v(b)}}{|b|_v}db=\xi_v^{2}(\varpi_v)q_v^{-1-\frac{d_v}{2}}\Big[\mathbf{1}_{r_{\xi_v^2}=1}G(\overline{\xi}_v^2,\psi_v)-\mathbf{1}_{r_{\xi_v^{2}}=0}\Big].
\end{equation}

Therefore, \eqref{5.13} follows from \eqref{5.12}, \eqref{e5.13}, \eqref{e5.14} and \eqref{e5.15}. 

\item Suppose $r_{\xi_v}=1$ and $r_{\overline{\omega}_v\xi_v}=0$. Then \eqref{e5.13} still holds, and \eqref{e5.14} becomes 
\begin{equation}\label{e5.18}
\int_{\mathcal{O}_v^{\times}}\psi_v(\varpi_v^{l}\kappa)\overline{\omega}_v\xi_v(\kappa)
d^{\times}\kappa=q_v^{-\frac{d_v}{2}}\mathbf{1}_{l\geq 0}-q_v^{-1-\frac{d_v}{2}}\zeta_v(1)\mathbf{1}_{l=-1}.
\end{equation}

For each $i\geq 0$, we have
\begin{align*}
\int_{F_v}\frac{\xi_v^{-2}(b)\mathbf{1}_{e_v(b)<0}\overline{\psi_v(\varpi_v^ib)}}{|b|_v}db=\zeta_v(1)^{-1}\sum_{m=1}^{\infty}\xi_v^{2m}(\varpi_v)\int_{\mathcal{O}_v^{\times}}\overline{\xi}_v^2(\beta)\overline{\psi_v(\varpi_v^{i-m}\beta)}d^{\times}\beta.
\end{align*}

Making use of the property of Gauss sums and Ramanujan sums, 
\begin{multline*}
\int_{\mathcal{O}_v^{\times}}\overline{\xi}_v^2(\beta)\overline{\psi_v(\varpi_v^{i-m}\beta)}d^{\times}\beta=\mathbf{1}_{r_{\overline{\xi}_v^2}=1}\mathbf{1}_{m=i+1}q_v^{-1-\frac{d_v}{2}}\zeta_v(1)G(\overline{\xi}_v^2,\overline{\psi}_v)\\
-\mathbf{1}_{r_{\overline{\xi}_v^2}=0}\mathbf{1}_{m=i+1}q_v^{-1-\frac{d_v}{2}}\zeta_v(1)+\mathbf{1}_{r_{\overline{\xi}_v^2}=0}\mathbf{1}_{m\leq i}q_v^{-\frac{d_v}{2}}.
\end{multline*}

As a consequence, 
\begin{multline}\label{e5.19}
\int_{F_v}\xi_v^{-2}(b)|b|_v^{-1}\mathbf{1}_{e_v(b)<0}\overline{\psi_v(\varpi_v^ib)}db=\xi_v^{2i+2}(\varpi_v)q_v^{-1-\frac{d_v}{2}}\mathbf{1}_{r_{\overline{\xi}_v^2}=1}G(\overline{\xi}_v^2,\overline{\psi}_v)\\
-\xi_v^{2i+2}(\varpi_v)q_v^{-1-\frac{d_v}{2}}\mathbf{1}_{r_{\overline{\xi}_v^2}=0}+\zeta_v(1)^{-1}\mathbf{1}_{r_{\overline{\xi}_v^2}=0}\mathbf{1}_{i\geq 1}q_v^{-\frac{d_v}{2}}\sum_{m=1}^{i}\xi_v^{2m}(\varpi_v).
\end{multline}

Substituting \eqref{e5.13}, \eqref{e5.18} and \eqref{e5.19} into \eqref{5.12} yields 
\begin{multline*}
\mathcal{I}_{v}(W_v^{(1)};k')=\xi_v(-1)\psi_v(
\varpi_v^{-1}t\gamma)\Vol(I_v^*(1))^{-1}\zeta_v(1)^{-1}\\
\omega_v(\varpi_v)q_v^{-d_v}\overline{\xi}_v(t)G(\xi_v,\overline{\psi}_v)\sum_{i=0}^{\infty}q_v^{-\frac{i}{2}}\overline{\omega}_v^i\xi_v^{-i}(\varpi_v)
\Big[\mathbf{1}_{i\geq 1}-q_v^{-1}\zeta_v(1)\mathbf{1}_{i=0}\Big]\\
\bigg[\xi_v^{2i+2}(\varpi_v)q_v^{-1}\Big[\mathbf{1}_{r_{\overline{\xi}_v^2}=1}G(\overline{\xi}_v^2,\overline{\psi}_v)
-\mathbf{1}_{r_{\overline{\xi}_v^2}=0}\Big]
+\zeta_v(1)^{-1}\mathbf{1}_{r_{\overline{\xi}_v^2}=0}\mathbf{1}_{i\geq 1}\sum_{m=1}^i\xi_v^{2m}(\varpi_v)\bigg],
\end{multline*}
which amounts to 
\begin{multline}\label{5.12.}
\mathcal{I}_{v}(W_v^{(1)};k')=\psi_v(
\varpi_v^{-1}t\gamma)\Vol(I_v^*(1))^{-1}
\omega_v(\varpi_v)q_v^{-d_v}\overline{\xi}_v(-t)G(\xi_v,\overline{\psi}_v)\\
\zeta_v(1)^{-1}\bigg[\overline{\omega}_v\xi_v^3(\varpi_v)q_v^{-\frac{3}{2}}\Big[\mathbf{1}_{r_{\overline{\xi}_v^2}=1}G(\overline{\xi}_v^2,\overline{\psi}_v)
-\mathbf{1}_{r_{\overline{\xi}_v^2}=0}\Big]L(1/2,\overline{\omega}_v\xi_v)\\
+\zeta_v(1)^{-1}\mathbf{1}_{r_{\overline{\xi}_v^2}=0}\sum_{i=1}^{\infty}q_v^{-\frac{i}{2}}\overline{\omega}_v^i\xi_v^{-i}(\varpi_v)\sum_{m=1}^i\xi_v^{2m}(\varpi_v)\\
-
\zeta_v(1)\xi_v^{2}(\varpi_v)q_v^{-2}\mathbf{1}_{r_{\overline{\xi}_v^2}=1}G(\overline{\xi}_v^2,\overline{\psi}_v)
+\zeta_v(1)\xi_v^{2}(\varpi_v)q_v^{-2}\mathbf{1}_{r_{\overline{\xi}_v^2}=0}\bigg].
\end{multline}

By a directly computation, we have
\begin{equation}\label{f5.23}
\sum_{i=1}^{\infty}q_v^{-\frac{i}{2}}\overline{\omega}_v^i\xi_v^{-i}(\varpi_v)\sum_{m=1}^i\xi_v^{2m}(\varpi_v)=\frac{\overline{\omega}_v\xi_v(\varpi_v)L(1/2,\overline{\omega}_v\xi_v^{-1})L(1/2,\overline{\omega}_v\xi_v)}{q_v^{-1/2}}.
\end{equation}

Substituting \eqref{f5.23} into \eqref{5.12.} yields \eqref{5.19}.

\item Suppose $r_{\xi_v}=0$ and $r_{\omega_v}=1$. Then \eqref{e5.14} holds, while \eqref{e5.13} becomes 
\begin{equation}\label{e5.20}
\sum_{\alpha\in \mathbb{F}_v^{\times}}\xi_v(\alpha)
\overline{\psi}_v(\varpi_v^{i-l-2}t\alpha)=\mathbf{1}_{l\leq i-2}q_v\zeta_v(1)^{-1}-\mathbf{1}_{l=i-1}\overline{\xi}_v(t).	
\end{equation}

Substituting \eqref{e5.14}, \eqref{e5.19} and \eqref{e5.20} into \eqref{5.12}, we obtain  
\begin{multline*}
\mathcal{I}_{v}(W_v^{(1)};k')=\xi_v(-1)\psi_v(
\varpi_v^{-1}t\gamma)q_v^{-d_v}\Vol(I_v^*(1))^{-1}\omega_v(\varpi_v)q_v^{-1}\\
G(\overline{\omega}_v\xi_v,\psi_v)\sum_{i\in \mathbb{Z}}q_v^{-\frac{i}{2}}\xi_v^{-i}(\varpi_v)
\big[\mathbf{1}_{i\geq 1}q_v\zeta_v(1)^{-1}-\overline{\xi}_v(t)\mathbf{1}_{i=0}\big]\\
\bigg[\zeta_v(1)^{-1}\mathbf{1}_{i\geq 1}\sum_{m=1}^i\xi_v^{2m}(\varpi_v)-\xi_v^{2i+2}(\varpi_v)q_v^{-1}\bigg].
\end{multline*}

Opening the brackets, we simplify the above expression as 
\begin{multline}\label{e5.22}
\mathcal{I}_{v}(W_v^{(1)};k')=\xi_v(-1)\psi_v(
\varpi_v^{-1}t\gamma)q_v^{-d_v}\Vol(I_v^*(1))^{-1}\omega_v(\varpi_v)q_v^{-1}
G(\overline{\omega}_v\xi_v,\psi_v)\\
\Big[\zeta_v(1)^{-2}q_v\sum_{i\geq 1}q_v^{-\frac{i}{2}}\xi_v^{-i}(\varpi_v)
\sum_{m=1}^i\xi_v^{2m}(\varpi_v)\\
-\zeta_v(1)^{-1}\xi_v^{2}(\varpi_v)\sum_{i\geq 1}q_v^{-\frac{i}{2}}
\xi_v^{i}(\varpi_v)
+
\overline{\xi}_v(t)\xi_v^{2}(\varpi_v)q_v^{-1}\Big].
\end{multline}

By \eqref{f5.23}, a direct evaluation of the sum over $i\in\mathbb{Z}$ in \eqref{e5.22} yields \eqref{5.21}.

\item Suppose $r_{\xi_v}=r_{\omega_v}=0$. By \eqref{5.12},   \eqref{e5.18}, \eqref{e5.19} and \eqref{e5.20}, we obtain   
\begin{multline}\label{5.23}
\mathcal{I}_{v}(W_v^{(1)};k')=\xi_v(-1)\psi_v(
\varpi_v^{-1}t\gamma)\Vol(I_v^*(1))^{-1}\zeta_v(1)^{-1}q_v^{-d_v}\\
\sum_{i\geq 0}q_v^{-\frac{i}{2}}\xi_v^{-i}(\varpi_v)\bigg[\zeta_v(1)^{-1}\mathbf{1}_{i\geq 1}\sum_{m=1}^{i}\xi_v^{2m}(\varpi_v)-\xi_v^{2i+2}(\varpi_v)q_v^{-1}\bigg]S_i,
\end{multline}
where 
\begin{align*}
S_i:=\sum_{l=-1}^{i-1}\overline{\omega}_v^l(\varpi_v)\big[\mathbf{1}_{l\leq i-2}q_v\zeta_v(1)^{-1}-\overline{\xi}_v(t)\mathbf{1}_{l=i-1}\big]\Big[\mathbf{1}_{l\geq 0}-q_v^{-1}\zeta_v(1)\mathbf{1}_{l=-1}\Big].
\end{align*}

A straightforward calculation leads to 
\begin{multline}\label{5.24}
S_i=q_v\zeta_v(1)^{-1}\mathbf{1}_{i\geq 2}\sum_{l=0}^{i-2}\overline{\omega}_v^l(\varpi_v)-\overline{\xi}_v(t)\mathbf{1}_{i\geq 1}\overline{\omega}_v^{i-1}(\varpi_v)\\
-\mathbf{1}_{l=-1}\mathbf{1}_{i\geq 1}\omega_v(\varpi_v)+q_v^{-1}\zeta_v(1)\overline{\xi}_v(t)\omega_v(\varpi_v).
\end{multline}

Consequently, \eqref{5.25} follows from \eqref{5.23} and \eqref{5.24}. Notice that \eqref{5.25} converges absolutely. 
\end{itemize}

Therefore, Lemma \ref{lem5.3} follows from the above discussions. 
\end{proof}

\begin{lemma}\label{lem5.4}
We have the following. 
\begin{itemize}
\item Suppose 
Suppose $r_{\xi_v}=r_{\omega_v\xi_v}=1$. Then 
\begin{equation}\label{c5.26}
\mathcal{I}_{v}(W_v^{(2)};k')=\frac{\psi_v(
\varpi_v^{-1}t\gamma)\xi_v(t)\omega_v(\varpi_v)\Vol(I_v^*(1))^{-1}\overline{G(\xi_v,\psi_v)}G(\overline{\omega}_v\overline{\xi}_v,\psi_v)}{q_v^{1+d_v}}.
\end{equation}

\item Suppose $r_{\xi_v}=1$ and $r_{\omega_v\xi_v}=0$. Then 
\begin{multline}\label{c5.27}
\mathcal{I}_{v}(W_v^{(2)};k')=\psi_v(
\varpi_v^{-1}t\gamma)\Vol(I_v^*(1))^{-1}\zeta_v(1)^{-1}q_v^{-\frac{1}{2}-d_v}\overline{\xi}_v(\varpi_v)\\
G(\xi_v,\overline{\psi}_v)(1-\omega_v\xi_v(\varpi_v)q_v^{-1/2})L(1/2,\overline{\omega}_v\xi_v^{-1}).
\end{multline}

\item Suppose $r_{\xi_v}=0$ and $r_{\omega_v\xi_v}=1$. Then 
\begin{multline}\label{c5.28}
\mathcal{I}_{v}(W_v^{(2)};k')=\psi_v(
\varpi_v^{-1}t\gamma)\Vol(I_v^*(1))^{-1}\omega_v(\varpi_v)\\
\xi_v(t)q_v^{-1-d_v}G(\overline{\omega}_v\overline{\xi}_v,\psi_v)\Big[\zeta_v(1)^{-1}q_v^{\frac{1}{2}}\xi_v^{-1}(\varpi_v)L(1/2,\xi_v^{-1})-\overline{\xi}_v(t)\Big].
\end{multline}
\item Suppose $r_{\xi_v}=r_{\omega_v\xi_v}=0$. Then 
\begin{multline}\label{c5.29}
\mathcal{I}_{v}(W_v^{(2)};k')=\psi_v(
\varpi_v^{-1}t\gamma)\Vol(I_v^*(1))^{-1}\zeta_v(1)^{-1}\xi_v(t)q_v^{-d_v}\\
\sum_{i\geq 0}q_v^{-\frac{i}{2}}\xi_v^{-i}(\varpi_v)
\Big[q_v\zeta_v(1)^{-1}\mathbf{1}_{i\geq 2}\sum_{l=0}^{i-2}\overline{\omega}_v^l(\varpi_v)\\
-\overline{\omega}_v^{i-1}(\varpi_v)\overline{\xi}_v(t)\mathbf{1}_{i\geq 1}-\omega_v(\varpi_v)\mathbf{1}_{i\geq 1}+\omega_v(\varpi_v)q_v^{-1}\zeta_v(1)\mathbf{1}_{i=0}\overline{\xi}_v(t)\Big].
\end{multline}
\end{itemize}
\end{lemma}
\begin{proof}
In parallel with \eqref{5.10} we obtain 
\begin{multline}\label{f5.13}
\mathcal{I}_{v}(W_v^{(2)};k')=\psi_v(
\varpi_v^{-1}t\gamma)\Vol(I_v^*(1))^{-1}\zeta_v(1)^{-1}
\sum_{i\in \mathbb{Z}}\sum_{\alpha\in \mathbb{F}_v^{\times}}\\
W_v^{(2)}\left(\begin{pmatrix}
\alpha\varpi_v^i\\
& 1
\end{pmatrix}\right)
\sum_{l=-1}^{i-1}\overline{\omega}_v^l(\varpi_v)\int_{\mathcal{O}_v^{\times}}
\overline{\psi}_v(\varpi_v^{i-l-2}t\alpha\kappa^{-1}-\varpi_v^{l}\kappa)\overline{\omega}_v(\kappa)
d^{\times}\kappa. 
\end{multline}

By Lemma \ref{lem3.4} we have 
\begin{align*}
W_v^{(2)}\left(\begin{pmatrix}
\alpha\varpi_v^i\\
& 1
\end{pmatrix}\right)=q_v^{-\frac{i}{2}}\xi_v^{-i}(\varpi_v)\overline{\xi}_v(\alpha)\int_{\mathcal{O}_v}\overline{\psi_v(\varpi_v^ib)}db.
\end{align*}

Substituting this into \eqref{f5.13}, together with the change of variable $\alpha\mapsto t^{-1}\kappa \alpha$, we obtain  
\begin{multline}\label{e5.27}
\mathcal{I}_{v}(W_v^{(2)};k')=\psi_v(
\varpi_v^{-1}t\gamma)\Vol(I_v^*(1))^{-1}\zeta_v(1)^{-1}\xi_v(t)\\
\sum_{i\in \mathbb{Z}}q_v^{-\frac{i}{2}}\xi_v^{-i}(\varpi_v)\int_{\mathcal{O}_v}\overline{\psi_v(\varpi_v^ib)}db\cdot T_i, 
\end{multline}
where 
\begin{align*}
T_i:=\sum_{l=-1}^{i-1}\overline{\omega}_v^l(\varpi_v)\sum_{\alpha\in \mathbb{F}_v^{\times}}\overline{\xi}_v(\alpha)
\overline{\psi}_v(\varpi_v^{i-l-2}\alpha)\int_{\mathcal{O}_v^{\times}}\psi_v(\varpi_v^{l}\kappa)\overline{\omega}_v\overline{\xi}_v(\kappa)
d^{\times}\kappa.
\end{align*}

\begin{itemize}
\item Suppose $r_{\xi_v}=r_{\omega_v\xi_v}=1$. By \eqref{e5.13} and \eqref{e5.14},  
\begin{equation}\label{5.28}
T_i=\omega_v(\varpi_v)q_v^{-1-\frac{d_v}{2}}\zeta_v(1)\mathbf{1}_{i=0}\overline{G(\xi_v,\psi_v)}G(\overline{\omega}_v\overline{\xi}_v,\psi_v).
\end{equation}

\item Suppose $r_{\xi_v}=1$ and $r_{\omega_v\xi_v}=0$. By \eqref{e5.13} and \eqref{e5.18}, 
\begin{equation}\label{5.29}
T_i=\overline{\omega}_v^{i-1}(\varpi_v)q_v^{-\frac{d_v}{2}}\overline{\xi}_v(t)G(\xi_v,\overline{\psi}_v)\big[\mathbf{1}_{i\geq 1}-q_v^{-1}\zeta_v(1)\mathbf{1}_{i=0}\big].
\end{equation}

\item Suppose $r_{\xi_v}=0$ and $r_{\omega_v\xi_v}=1$. By \eqref{e5.14} and \eqref{e5.20}, 
\begin{equation}\label{5.30}
T_i=\omega_v(\varpi_v)q_v^{-1-\frac{d_v}{2}}\big[\mathbf{1}_{i\geq 1}q_v\zeta_v(1)^{-1}-\mathbf{1}_{i=0}\overline{\xi}_v(t)\big]\zeta_v(1)G(\overline{\omega}_v\overline{\xi}_v,\psi_v).
\end{equation}

\item Suppose $r_{\xi_v}=r_{\omega_v\xi_v}=0$. By \eqref{e5.18} and \eqref{e5.20}, we obtain 
\begin{multline}\label{5.31}
T_i=q_v\zeta_v(1)^{-1}\sum_{l=0}^{i-2}\overline{\omega}_v^l(\varpi_v)q_v^{-\frac{d_v}{2}}-\overline{\omega}_v^{i-1}(\varpi_v)q_v^{-\frac{d_v}{2}}\overline{\xi}_v(t)\mathbf{1}_{i\geq 1}\\
-\omega_v(\varpi_v)q_v^{-1-\frac{d_v}{2}}\zeta_v(1)\mathbf{1}_{i\geq 1}q_v\zeta_v(1)^{-1}+\omega_v(\varpi_v)q_v^{-1-\frac{d_v}{2}}\zeta_v(1)\mathbf{1}_{i=0}\overline{\xi}_v(t).
\end{multline}
\end{itemize}

Therefore, the expressions \eqref{c5.26}, \eqref{c5.27}, \eqref{c5.28} and \eqref{c5.29} follow from substituting \eqref{5.28}, \eqref{5.29}, \eqref{5.30} and \eqref{5.31} into \eqref{e5.27}, respectively. 
\end{proof}

\subsubsection{$\sigma_v$ is a special representation}
Suppose $\sigma_v=\mathrm{St}\otimes\chi_v$ is a special representation with $\chi_v^2=\mathbf{1}$ and $r_{\chi_v}\leq 1$. From the discussion in \textsection\ref{sec4.2}, we may take $W_v=W_v^{\chi_v}$, as defined by \eqref{4.11}, to be a right $I_v(1)$-invariant unit Whittaker function of $\sigma_v$.

\begin{lemma}\label{lemma5.5}
We have the following. 
\begin{itemize}
\item Suppose that $r_{\chi_v}=r_{\overline{\omega}_v\chi_v}=1$. Then 
\begin{equation}\label{e5.37}
\mathcal{I}_{v}(W_v^{\chi_v};k')=\frac{\omega_v(\varpi_v)\psi_v(
\varpi_v^{-1}t\gamma)\overline{\chi}_v(t)}{\Vol(I_v^*(1))
q_v^{1+d_v/2}}G(\chi_v,\overline{\psi}_v)G(\overline{\omega}_v\chi_v,\psi_v).
\end{equation}
\item Suppose that $r_{\chi_v}=1$ and $r_{\overline{\omega}_v\chi_v}=0$. Then 
\begin{multline}\label{eq5.40}
\mathcal{I}_{v}(W_v^{\chi_v};k')=\omega_v(\varpi_v)\psi_v(
\varpi_v^{-1}t\gamma)\overline{\chi}_v(t)\Vol(I_v^*(1))^{-1}\zeta_v(1)^{-1}q_v^{-1-\frac{d_v}{2}}
\\
G(\chi_v,\overline{\psi}_v)\big[\overline{\omega}_v\chi_v(\varpi_v)L_v(1,\overline{\omega}_v\chi_v)-\zeta_v(1)\big].
\end{multline}
\item Suppose that $r_{\chi_v}=0$ and $r_{\overline{\omega}_v\chi_v}=1$. Then 
\begin{multline}\label{eq5.41}
\mathcal{I}_{v}(W_v^{\chi_v};k')=\omega_v(\varpi_v)\psi_v(
\varpi_v^{-1}t\gamma)\Vol(I_v^*(1))^{-1}
q_v^{-1-\frac{d_v}{2}}G(\overline{\omega}_v\chi_v,\psi_v)\\
\big[\zeta_v(1)^{-1}\chi_v(\varpi_v)L_v(1,\chi_v)-\overline{\chi}_v(t)\big].
\end{multline}
\item Suppose that $r_{\chi_v}=r_{\overline{\omega}_v\chi_v}=0$. Then 
\begin{multline}\label{eq5.42}
\mathcal{I}_{v}(W_v^{\chi_v};k')=\psi_v(
\varpi_v^{-1}t\gamma)\Vol(I_v^*(1))^{-1}\zeta_v(1)^{-1}q_v^{-1-\frac{d_v}{2}}
\\
\Big[q_v^2\zeta_v(1)^{-1}\sum_{i\geq 2}\chi_v(\varpi_v^i)q_v^{-i}\sum_{l=0}^{i-2}\overline{\omega}_v^l(\varpi_v)
-\overline{\chi}_v(t)\chi_v(\varpi_v)L_v(1,\overline{\omega}_v\chi_v)\\
-\omega_v\chi_v(\varpi_v)L_v(1,\chi_v)+\omega_v(\varpi_v)\overline{\chi}_v(t)\zeta_v(1)\Big].
\end{multline}
\end{itemize}
\end{lemma}
\begin{proof}
In parallel with \eqref{5.10} we obtain 
\begin{multline}\label{e5.38}
\mathcal{I}_{v}(W_v^{\chi_v};k')=\psi_v(
\varpi_v^{-1}t\gamma)\Vol(I_v^*(1))^{-1}\zeta_v(1)^{-1}
\sum_{i\in \mathbb{Z}}\sum_{\alpha\in \mathbb{F}_v^{\times}}\\
W_v^{\chi_v}\left(\begin{pmatrix}
\alpha\varpi_v^i\\
& 1
\end{pmatrix}\right)
\sum_{l=-1}^{i-1}\overline{\omega}_v^l(\varpi_v)\int_{\mathcal{O}_v^{\times}}
\overline{\psi}_v(\varpi_v^{i-l-2}t\alpha\kappa^{-1}-\varpi_v^{l}\kappa)\overline{\omega}_v(\kappa)
d^{\times}\kappa.
\end{multline}

Substituting \eqref{4.11} into \eqref{e5.38}, and making the change of variables $\alpha\mapsto \kappa\alpha$, we obtain
\begin{multline}\label{eq5.39}
\mathcal{I}_{v}(W_v^{\chi_v};k')=\psi_v(
\varpi_v^{-1}t\gamma)\Vol(I_v^*(1))^{-1}\zeta_v(1)^{-1}
\sum_{i\in \mathbb{Z}}\chi_v(\varpi_v^i)q_v^{-i}\\
\sum_{l=-1}^{i-1}\overline{\omega}_v^l(\varpi_v)\sum_{\alpha\in \mathbb{F}_v^{\times}}\chi_v(\alpha) 
\overline{\psi}_v(\varpi_v^{i-l-2}t\alpha)\int_{\mathcal{O}_v^{\times}}\psi_v(\varpi_v^{l}\kappa)\overline{\omega}_v\chi_v(\kappa)d^{\times}\kappa.
\end{multline}

We now analyze \eqref{eq5.39} by distinguishing cases according to the values of $r_{\chi_v}$ and $r_{\overline{\omega}_v\chi_v}$.
\begin{itemize}
\item Suppose that $r_{\chi_v}=1$ and $r_{\overline{\omega}_v\chi_v}=1$. Then \eqref{e5.37} follows from \eqref{e5.13},  \eqref{e5.14} and \eqref{eq5.39}.

\item Suppose that $r_{\chi_v}=1$ and $r_{\overline{\omega}_v\chi_v}=0$. Substituting   \eqref{e5.13} and   \eqref{e5.18} into \eqref{eq5.39} yields 
\begin{multline*}
\mathcal{I}_{v}(W_v^{\chi_v};k')=\psi_v(
\varpi_v^{-1}t\gamma)\overline{\chi}_v(t)\Vol(I_v^*(1))^{-1}\zeta_v(1)^{-1}q_v^{-\frac{d_v}{2}}
\sum_{i\in \mathbb{Z}}\chi_v(\varpi_v^i)q_v^{-i}\\
\overline{\omega}_v^{i-1}(\varpi_v)G(\chi_v,\overline{\psi}_v)\big[\mathbf{1}_{i\geq 1}-q_v^{-1}\zeta_v(1)\mathbf{1}_{i=0}\big].
\end{multline*}
which simplifies to \eqref{eq5.40}. 

\item Suppose that $r_{\chi_v}=0$ and $r_{\overline{\omega}_v\chi_v}=1$. Substituting   \eqref{e5.14} and \eqref{e5.20} into \eqref{eq5.39} yields 
\begin{multline*}
\mathcal{I}_{v}(W_v^{\chi_v};k')=\psi_v(
\varpi_v^{-1}t\gamma)\Vol(I_v^*(1))^{-1}\zeta_v(1)^{-1}
\sum_{i\in \mathbb{Z}}\chi_v(\varpi_v^i)q_v^{-i}\\
\sum_{l=-1}^{i-1}\overline{\omega}_v^l(\varpi_v)\big[\mathbf{1}_{l\leq i-2}q_v\zeta_v(1)^{-1}-\mathbf{1}_{l=i-1}\overline{\chi}_v(t)\big]q_v^{-1-\frac{d_v}{2}}\zeta_v(1)\mathbf{1}_{l=-1}G(\overline{\omega}_v\chi_v,\psi_v),
\end{multline*}
which simplifies to \eqref{eq5.41}. 

\item Suppose that $r_{\chi_v}=r_{\overline{\omega}_v\chi_v}=0$. Suppose that $r_{\chi_v}=1$ and $r_{\overline{\omega}_v\chi_v}=0$. Substituting   \eqref{e5.18} and \eqref{e5.20} into \eqref{eq5.39} yields
\begin{multline*}
\mathcal{I}_{v}(W_v^{\chi_v};k')=\psi_v(
\varpi_v^{-1}t\gamma)\Vol(I_v^*(1))^{-1}\zeta_v(1)^{-1}
\sum_{i\in \mathbb{Z}}\chi_v(\varpi_v^i)q_v^{-i}\sum_{l=-1}^{i-1}\overline{\omega}_v^l(\varpi_v)\\
\big[\mathbf{1}_{l\leq i-2}q_v\zeta_v(1)^{-1}-\mathbf{1}_{l=i-1}\overline{\chi}_v(t)\big]\big[q_v^{-\frac{d_v}{2}}\mathbf{1}_{l\geq 0}-q_v^{-1-\frac{d_v}{2}}\zeta_v(1)\mathbf{1}_{l=-1}\big],
\end{multline*}
which, after a direct calculation, reduces to \eqref{eq5.42}. 
\end{itemize}

Therefore, Lemma \ref{lemma5.5} follows. 
\end{proof}

\subsection{The Integral $\mathcal{J}_{v}(W_v)$}\label{sec5.3}
Let $\pi_{1,v}=\pi_{3,v}$ be an unramified irreducible admissible representation of $G(F_v)$, and let $W_{1,v}^{\circ}=W_{3,v}^{\circ}$ be a fixed spherical Whittaker function in $\pi_{1,v}$. 
We retain the indices to preserve compatibility with the local period integrals in 
$\mathcal{I}_{\Spec}^{\heartsuit}(\mathbf{s},R(f_{\mathfrak{q}})\boldsymbol{h})$ appearing in the symmetric spectral reciprocity of Theorem \ref{thm2.1}.

Analogously to \eqref{e4.6}, we define the integral
\begin{equation}\label{eq5.46}
\mathcal{I}_{v}^{\dag}(W_v;k'):=\int_{N(F_v)\backslash \overline{G}(F_v)}\overline{W_v(x)}W_{1,v}^{\circ}\left(xk'A_{\mathfrak{q}}\right)
\overline{h_{3,v}^{\circ}\left(xA_{\mathfrak{q}}\right)}|\det x|_v^{\frac{1}{2}}dx.
\end{equation}
where $A_{\mathfrak{q}}:=\diag(\varpi_v, 1)$ and  $h_{3,v}^{\circ}$ is the spherical section corresponding to $W_{3,v}^{\circ}$.  

Let $I_v^{\circ}(1):=\{g\in I_v(1):\ g\equiv I_2\pmod{\mathfrak{p}_v}\}$. Then 
\begin{equation}\label{f5.47}
I_v(1)=\bigsqcup_{\beta'\in \mathbb{F}_v}\begin{pmatrix}
1 & \beta'\\
& 1
\end{pmatrix}I_v^{\circ}(1).
\end{equation}
Since the function $x\mapsto W_{1,v}^{\circ}(xA_{\mathfrak{q}})$ is right-$I_v^{\circ}(1)$-invariant, it follows that
\begin{equation}\label{eq5.47}
\mathcal{I}_{v}^{\dag}(W_v;k')=\mathcal{I}_{v}^{\dag}\left(W_v;\begin{pmatrix}
1 &\beta\\
& 1
\end{pmatrix}\right).
\end{equation}

Define the auxiliary sum 
\begin{equation}\label{e5.43}
\mathcal{J}_{v}(W_v):=\sum_{\beta\in \mathbb{F}_v}\overline{\psi}_v(\varpi_v^{-1}\beta)\,\mathcal{I}_{v}^{\dag}(W_v;k').
\end{equation}

As a consequence of \eqref{eq5.47}, the expression in \eqref{e5.43} is well defined; in particular, the right-hand side does not depend on the choice of $k'$.


\subsubsection{$\sigma_v$ is a principal series}
Suppose $\sigma_v=\xi_v\boxplus\xi_v^{-1}$ is a principal series with $r_{\xi_v}\leq 1$. Suppose $W_v$ is the Jacquet integral of a section $f\in \sigma_v$. By $\dim \Hom(\sigma_v\otimes\pi_{1,v} \otimes\widetilde{\pi}_{3,v},\mathbb{C})\leq 1$, there exists a constant $c_{\sigma_v,\pi_{1,v}}$, depending only on $\sigma_v\otimes\pi_{1,v}\otimes\widetilde{\pi}_{3,v}$, such that
\begin{equation}\label{eq5.44}
\mathcal{I}_{v}^{\dag}(W_v;k')
=c_{\sigma_v,\pi_{1,v}}\cdot\int_{N(F_v)\backslash \overline{G}(F_v)}W_{1,v}^{\circ}\left(xk'A_{\mathfrak{q}}\right)\overline{W_{3,v}^{\circ}\left(xA_{\mathfrak{q}}\right)}\overline{f(x)}dx.
\end{equation}
In particular,  $c_{\sigma_v,\pi_{1,v}}$ is independent of $W_v$, $W_{1,v}^{\circ}=W_{3,v}^{\circ}$. By \cite[Corollary 5.2]{Hsi21} we have 
\begin{equation}\label{f5.46}
|c_{\sigma_v,\pi_{1,v}}|=\sqrt{\zeta_v(1)\cdot |\gamma(1/2,\pi_{1,v}\times\widetilde{\pi}_{1,v}\times\xi_v)|}\ll 1.
\end{equation}
Here we have used the bound towards the Ramanujan conjecture; see \cite{KS03}.

By \eqref{eq5.44}, we obtain, for $j\in \{1, 2\}$, that 
\begin{equation}\label{f5.45}
\mathcal{J}_{v}(W_v^{(j)})=c_{\sigma_v,\pi_{1,v}}\cdot \mathcal{J}_{v}^*(W_v^{(j)}),	
\end{equation}
where 
\begin{align*}
\mathcal{J}_{v}^*(W_v^{(j)}):=\sum_{\beta\in \mathbb{F}_v}\overline{\psi}_v(\varpi_v^{-1}\beta)\int_{N(F_v)\backslash \overline{G}(F_v)}W_{1,v}^{\circ}\left(xk'A_{\mathfrak{q}}\right)\overline{W_{3,v}^{\circ}\left(xA_{\mathfrak{q}}\right)}\overline{f_j(x)}dx.
\end{align*}

\begin{lemma}\label{lem5.5}
We have the following. 
\begin{itemize}
\item Suppose $r_{\xi_v}=1$. Then  
\begin{equation}\label{f5.37}
\mathcal{J}_{v}^*(W_v^{(1)})=\Vol(I_v^*(1))q_v^{-\frac{1}{2}}\overline{\xi}_v^{-1}(\varpi_v)
|W_{1,v}^{\circ}(I_2)|^2|G(\xi_v,\overline{\psi}_v)|^2.
\end{equation}
\item Suppose $r_{\xi_v}=0$. Then  
\begin{multline}\label{5.44}
\mathcal{J}_{v}^*(W_v^{(1)})=\Vol(I_v^*(1))(q_v-1)q_v^{-\frac{1}{2}}\overline{\xi}_v^{-1}(\varpi_v)|W_{1,v}^{\circ}(I_2)|^2\\
\big[L(1,\overline{\xi}_v^2)^{-1}L(1/2,\pi_{1,v}\times\widetilde{\pi}_{1,v}\times\overline{\xi}_v)-(q_v-2)(q_v-1)^{-1}
\big]. 
\end{multline}
\end{itemize}
\end{lemma}
\begin{proof}
By the Iwasawa decomposition, 
\begin{multline*}
\mathcal{J}_{v}^*(W_v^{(1)})=\sum_{\beta\in \mathbb{F}_v}\overline{\psi}_v(\varpi_v^{-1}\beta)\sum_{i\in \mathbb{Z}}q_v^{i}\int_{K_v}W_{1,v}^{\circ}\left(\begin{pmatrix}
\varpi_v^i\\
& 1
\end{pmatrix}
kk'\begin{pmatrix}
\varpi_v\\
& 1
\end{pmatrix}\right)\\
\overline{W_{3,v}^{\circ}\left(\begin{pmatrix}
\varpi_v^i\\
& 1
\end{pmatrix}
k\begin{pmatrix}
\varpi_v\\
& 1
\end{pmatrix}\right)}\overline{f_1\left(\begin{pmatrix}
\varpi_v^i\\
& 1
\end{pmatrix}
k\right)}dk.
\end{multline*}

By the definition of $f_1$ in \textsection\ref{sec4} and the decomposition \eqref{eq5.13}, we obtain 
\begin{equation}\label{5.37}
\mathcal{J}_{v}^*(W_v^{(1)})=\sum_{\beta\in \mathbb{F}_v}\overline{\psi}_v(\varpi_v^{-1}\beta)\sum_{i\in \mathbb{Z}}q_v^{\frac{i}{2}}\overline{\xi}_v^i(\varpi_v)\sum_{\alpha\in \mathbb{F}_v^{\times}}\overline{\xi}_v(\alpha)\cdot J_1(\alpha,i,k'),
\end{equation}
where $J_1(\alpha,i,k')$ is defined by 
\begin{align*}
\int_{I_v^*(1)}W_{1,v}^{\circ}\left(\begin{pmatrix}
\alpha\varpi_v^i\\
& 1
\end{pmatrix}
kk'\begin{pmatrix}
\varpi_v\\
& 1
\end{pmatrix}\right)
\overline{W_{3,v}^{\circ}\left(\begin{pmatrix}
\alpha\varpi_v^i\\
& 1
\end{pmatrix}
k\begin{pmatrix}
\varpi_v\\
& 1
\end{pmatrix}\right)}dk.
\end{align*}

Recall that $k'$ is parametrized by \eqref{eq5.1}. Write  
\begin{align*}
k=\kappa\begin{pmatrix}
1+\alpha_1\varpi_v & \beta_1\\
\gamma_1\varpi_v & 1+\delta_1\varpi_v 
\end{pmatrix}\in I_v^*(1), \ \ \kappa\in \mathbb{F}_v^{\times}. 
\end{align*}

Substituting these matrices into the definition of $J_1(k')$, together with the assumption that $W_{1,v}^{\circ}$ and $W_{3,v}^{\circ}$ are spherical, we obtain  
\begin{equation}\label{5.38}
J_1(\alpha,i,k')=\psi_v(\alpha\varpi_v^i\beta)
\bigg|W_{1,v}^{\circ}\left(\begin{pmatrix}
\varpi_v^{i+1}\\
& 1
\end{pmatrix}
\right)\bigg|^2\Vol(I_v^*(1)).
\end{equation}

It thus follows from \eqref{5.37} and \eqref{5.38} that 
\begin{equation}\label{e5.39}
\mathcal{J}_{v}^*(W_v^{(1)})=\Vol(I_v^*(1))\sum_{i\in \mathbb{Z}}q_v^{\frac{i}{2}}\overline{\xi}_v^i(\varpi_v)
\bigg|W_{1,v}^{\circ}\left(\begin{pmatrix}
\varpi_v^{i+1}\\
& 1
\end{pmatrix}
\right)\bigg|^2\cdot S(i)
\end{equation}
where
\begin{align*}
S(i):=\sum_{\beta\in \mathbb{F}_v}\overline{\psi}_v(\varpi_v^{-1}\beta)\sum_{\alpha\in \mathbb{F}_v^{\times}}\overline{\xi}_v(\alpha)\psi_v(\alpha\varpi_v^i\beta). 
\end{align*}

Making the change of variable $\alpha\mapsto \beta^{-1}\alpha$, we derive 
\begin{align*}
S(i)=\sum_{\beta\in \mathbb{F}_v^{\times}}\xi_v(\beta)\overline{\psi}_v(\varpi_v^{-1}\beta)\sum_{\alpha\in \mathbb{F}_v^{\times}}\overline{\xi}_v(\alpha)\psi_v(\varpi_v^i\alpha) 
+\sum_{\alpha\in \mathbb{F}_v^{\times}}\overline{\xi}_v(\alpha)\psi_v(\alpha\varpi_v^{i+1}).
\end{align*}

Consider the following scenarios. 
\begin{itemize}
\item Suppose $r_{\xi_v}=1$. Then 
\begin{equation}\label{5.39}
S(i)=\mathbf{1}_{i=-1}|G(\xi_v,\overline{\psi}_v)|^2+\mathbf{1}_{i=-2}G(\overline{\xi}_v,\psi_v).	
\end{equation}

Consequently, \eqref{f5.37} follows from \eqref{e5.39} and \eqref{5.39}. 

\item Suppose $r_{\xi_v}=0$. Then 
\begin{equation}\label{5.40}
S(i)=\mathbf{1}_{i=-1}+(q_v-1)^2\mathbf{1}_{i\geq 0}-\mathbf{1}_{i=-2}+(q_v-1)\mathbf{1}_{i\geq -1}.	
\end{equation}

Notice that 
\begin{equation}\label{e5.44}
\sum_{i\geq 0}q_v^{\frac{i}{2}}\overline{\xi}_v^i(\varpi_v)
\bigg|W_{1,v}^{\circ}\left(\begin{pmatrix}
\varpi_v^{i}\\
& 1
\end{pmatrix}
\right)\bigg|^2=\frac{L(1/2,\pi_{1,v}\times\widetilde{\pi}_{1,v}\times\overline{\xi}_v)|W_{1,v}^{\circ}(I_2)|^2}{L(1,\overline{\xi}_v^2)}. 
\end{equation}

Hence, \eqref{5.44} follows from \eqref{e5.39}, \eqref{5.40} and \eqref{e5.44}. 
\end{itemize}

Therefore, Lemma \ref{lem5.5} holds. 
\end{proof}

\begin{lemma}\label{lem5.6}
Let $\xi_v$ be a character of $F_v^{\times}$ with $r_{\xi_v}\leq 1$. Let $i\in \mathbb{Z}_{\geq -1}$. Define
\begin{equation}\label{eq5.45}
S_i(\xi_v):=\sum_{\alpha\in \mathbb{F}_v^{\times}}\overline{\xi}_v(\alpha)\Big|\sum_{\beta\in \mathbb{F}_v^{\times}}\overline{\psi}_v(\varpi_v^{-1}\beta)\psi_v(\varpi_v^i\alpha\beta^{-1})\Big|^2.
\end{equation}
Then the following assertions hold.
\begin{itemize}
\item If $r_{\xi_v}=0$, then 
\begin{equation}\label{e5.45}
S_i(\xi_v)=(q_v^2-q_v-1)\mathbf{1}_{i=-1}+(q_v-1)\mathbf{1}_{i\geq 0}.
\end{equation}
\item If $r_{\xi_v}=1$, then
\begin{equation}\label{e5.46}
S_i(\xi_v)=G(\overline{\xi}_v,\psi_v)G(\overline{\xi}_v,\overline{\psi}_v)\mathbf{1}_{i=-1}J(\overline{\xi}_v,\overline{\xi}_v).
\end{equation}
\end{itemize}
\end{lemma}

\begin{proof}
Opening the square in the definition of $S_i(\xi_v)$ and applying the change of variables $\alpha\mapsto \beta\beta'\alpha$, we obtain
\begin{equation}\label{e5.47}
S_i(\xi_v)=\sum_{\alpha\in \mathbb{F}_v^{\times}}\sum_{\beta\in \mathbb{F}_v^{\times}}\sum_{\beta'\in \mathbb{F}_v^{\times}}\overline{\xi}_v(\alpha)\overline{\xi}_v(\beta)\overline{\xi}_v(\beta')\psi_v(\varpi_v^{-1}(\beta'-\beta)(\varpi_v^{i+1}\alpha+1)).
\end{equation}
We now consider the following cases according to the value of $r_{\xi_v}$.
\begin{itemize}
\item Suppose that $r_{\xi_v}=0$. Then $\overline{\xi}_v(\alpha)\equiv 1$ for all $\alpha\in \mathbb{F}_v^{\times}$. Hence, when $i\geq 0$, we have
\begin{equation}\label{e5.48}
S_i(\xi_v)=(q_v-1)\Big|\sum_{\beta\in \mathbb{F}_v^{\times}}\overline{\psi}_v(\varpi_v^{-1}\beta)\Big|^2=q_v-1.
\end{equation}

Next, we consider the case $i=-1$, in which case
\begin{equation}\label{e5.49}
S_i(\xi_v)=\sum_{\beta\in \mathbb{F}_v^{\times}}\sum_{\beta'\in \mathbb{F}_v^{\times}}\sum_{\alpha\in \mathbb{F}_v}\psi_v(\varpi_v^{-1}(\beta'-\beta)(\alpha+1))-\Big|\sum_{\beta\in \mathbb{F}_v^{\times}}\overline{\psi}_v(\varpi_v^{-1}\beta)\Big|^2.
\end{equation}
By the orthogonality of additive characters, we have
\begin{equation}\label{e5.50}
\sum_{\alpha\in \mathbb{F}_v}\psi_v(\varpi_v^{-1}(\beta'-\beta)(\alpha+1))=q_v\mathbf{1}_{\beta'-\beta=0\in \mathbb{F}_v}.
\end{equation}
Combining \eqref{e5.48}, \eqref{e5.49}, and \eqref{e5.50}, we obtain \eqref{e5.45}.

\item Suppose that $r_{\xi_v}=1$. It follows from \eqref{e5.47} and the orthogonality of multiplicative characters that
\begin{equation}\label{5.51}
S_i(\xi_v)=\mathbf{1}_{i=-1}\sum_{\alpha\in \mathbb{F}_v-\{0,-1\}}\sum_{\beta\in \mathbb{F}_v^{\times}}\sum_{\beta'\in \mathbb{F}_v^{\times}}\overline{\xi}_v(\alpha\beta\beta')\psi_v(\varpi_v^{-1}(\beta'-\beta)(\alpha+1)).
\end{equation}
Making the changes of variables $\beta\mapsto (\alpha+1)^{-1}\beta$ and $\beta'\mapsto (\alpha+1)^{-1}\beta'$, it follows from \eqref{5.51} that
\begin{equation}\label{5.52}
S_i(\xi_v)=\mathbf{1}_{i=-1}G(\overline{\xi}_v,\psi_v)G(\overline{\xi}_v,\overline{\psi}_v)\sum_{\alpha\in \mathbb{F}_v-\{0,-1\}}\overline{\xi}_v(\alpha)\xi_v^2(\alpha+1).
\end{equation}
Finally, making the change of variables $\alpha=\alpha'(1-\alpha')^{-1}$, which defines a bijection from $\mathbb{F}_v-\{0,-1\}$ to $\mathbb{F}_v-\{0,1\}$, we see that \eqref{e5.46} follows from \eqref{5.52}.
\end{itemize}
This completes the proof of Lemma \ref{lem5.6}.
\end{proof}

\begin{lemma}
Suppose $L(s,\pi_{1,v})=(1-\alpha_vq_v^{-s})^{-1}(1-\beta_vq_v^{-s})^{-1}$. We have the following. 
\begin{itemize}
\item Suppose $r_{\xi_v}=1$. Then 
\begin{equation}\label{5.54}
\mathcal{J}_{v}^*(W_v^{(2)})=q_v^{-\frac{1}{2}}\Vol(I_v^*(1)) G(\overline{\xi}_v,\psi_v)G(\overline{\xi}_v,\overline{\psi}_v)J(\overline{\xi}_v,\overline{\xi}_v)|W_{1,v}^{\circ}(I_2)|^2.	
\end{equation}
\item Suppose $r_{\xi_v}=0$. Then
\begin{multline}\label{5.55}
\mathcal{J}_{v}^*(W_v^{(2)})=(q_v-1)\Big[q_v^{\frac{1}{2}}\overline{\xi}_v(\varpi_v)L(1,\overline{\xi}_v^2)^{-1}L(1/2,\pi_{1,v}\times\widetilde{\pi}_{1,v}\times \overline{\xi}_v)\\
-2q_v^{-\frac{1}{2}}\Re \left(\overline{\xi}_v(\varpi_v)A\cdot L_v(1/2,\pi_{1,v}\times\widetilde{\pi}_{1,v}\times \overline{\xi}_v)\right)+q_v^{-\frac{1}{2}}(q_v-1)^{-1}(q_v^2-q_v-1)\\
+q_v^{-\frac{1}{2}}\overline{\xi}_v^{-1}(\varpi_v)\big[L(1,\overline{\xi}_v^2)^{-1}L(1/2,\pi_{1,v}\times\widetilde{\pi}_{1,v}\times \overline{\xi}_v)-1\big]\Big]\Vol(I_v^*(1)) |W_{1,v}^{\circ}(I_2)|^2,
\end{multline}
where $A:=\alpha_v\beta_v^{-1}+\alpha_v^{-1}\beta_v+1-|\alpha_v+\beta_v|^2\overline{\xi}_v(\varpi_v)q_v^{-1/2}+\overline{\xi}_v^2(\varpi_v)q_v^{-1}$. 
\end{itemize} 
\end{lemma}
\begin{proof}
By the Iwasawa decomposition, 
\begin{multline*}
\mathcal{J}_{v}^*(W_v^{(2)})=\sum_{\beta\in \mathbb{F}_v}\overline{\psi}_v(\varpi_v^{-1}\beta)\sum_{i\in \mathbb{Z}}q_v^{i}\int_{K_v}W_{1,v}^{\circ}\left(\begin{pmatrix}
\varpi_v^i\\
& 1
\end{pmatrix}
kk'\begin{pmatrix}
\varpi_v\\
& 1
\end{pmatrix}\right)\\
\overline{W_{3,v}^{\circ}\left(\begin{pmatrix}
\varpi_v^i\\
& 1
\end{pmatrix}
k\begin{pmatrix}
\varpi_v\\
& 1
\end{pmatrix}\right)}\overline{f_2\left(\begin{pmatrix}
\varpi_v^i\\
& 1
\end{pmatrix}
k\right)}dk.
\end{multline*}

By the definition of $f_2$ in \textsection\ref{sec4} and the decomposition \eqref{eq5.13}, we obtain 
\begin{multline*}
\mathcal{J}_{v}^*(W_v^{(2)})=\sum_{\beta\in \mathbb{F}_v}\overline{\psi}_v(\varpi_v^{-1}\beta)\sum_{i\in \mathbb{Z}}q_v^{\frac{i}{2}}\overline{\xi}_v^i(\varpi_v)\sum_{\alpha\in \mathbb{F}_v^{\times}}\overline{\xi}_v(\alpha)\\
\sum_{\beta_2\in \mathbb{F}_v}\int_{I_v^*(1)}W_{1,v}^{\circ}\left(\begin{pmatrix}
\alpha\varpi_v^i\\
& 1
\end{pmatrix}
\begin{pmatrix}
1 & \beta_2\\
& 1
\end{pmatrix}
wkk'\begin{pmatrix}
\varpi_v\\
& 1
\end{pmatrix}\right)\\
\overline{W_{3,v}^{\circ}\left(\begin{pmatrix}
\alpha\varpi_v^i\\
& 1
\end{pmatrix}
\begin{pmatrix}
1 & \beta_2\\
& 1
\end{pmatrix}
wk\begin{pmatrix}
\varpi_v\\
& 1
\end{pmatrix}\right)}dk.
\end{multline*}

Hence, we obtain 
\begin{equation}\label{5.45}
\mathcal{J}_{v}^*(W_v^{(2)})=q_v\sum_{\beta\in \mathbb{F}_v}\overline{\psi}_v(\varpi_v^{-1}\beta)\sum_{i\in \mathbb{Z}}q_v^{\frac{i}{2}}\overline{\xi}_v^i(\varpi_v)\sum_{\alpha\in \mathbb{F}_v^{\times}}\overline{\xi}_v(\alpha)J_2(\alpha,i,k'),
\end{equation}
where $J_2(\alpha,i,k')$ is defined by 
\begin{align*}
\int_{I_v^*(1)}W_{1,v}^{\circ}\left(\begin{pmatrix}
\alpha\varpi_v^i\\
& 1
\end{pmatrix}
wkk'\begin{pmatrix}
\varpi_v\\
& 1
\end{pmatrix}\right)\overline{W_{3,v}^{\circ}\left(\begin{pmatrix}
\alpha\varpi_v^i\\
& 1
\end{pmatrix}
wk\begin{pmatrix}
\varpi_v\\
& 1
\end{pmatrix}\right)}dk.
\end{align*}

Let $K_v(1):=\big\{g_v\in I_v(1):\ g_v\equiv I_2\mod\mathfrak{p}_v \big\}$. We have the decomposition  
\begin{equation}\label{5.67}
I_v(1)=\bigsqcup_{\beta_1\in \mathbb{F}_v}\begin{pmatrix}
1 & \beta_1\\
& 1
\end{pmatrix}K_v(1). 
\end{equation}

Let $K_v^*(1):=Z(\mathcal{O}_v)K_v(1)$. It follows from \eqref{5.67} that 
\begin{multline}\label{5.46}
J_2(\alpha,i,k')=\sum_{\beta_1\in \mathbb{F}_v}W_{1,v}^{\circ}\left(\begin{pmatrix}
\alpha\varpi_v^i\\
& 1
\end{pmatrix}
w\begin{pmatrix}
1& \beta_1+\beta\\
& 1
\end{pmatrix}\begin{pmatrix}
\varpi_v\\
& 1
\end{pmatrix}\right)\\\Vol(K_v^*(1))\overline{W_{3,v}^{\circ}\left(\begin{pmatrix}
\alpha\varpi_v^i\\
& 1
\end{pmatrix}
w\begin{pmatrix}
1& \beta_1\\
& 1
\end{pmatrix}\begin{pmatrix}
\varpi_v\\
& 1
\end{pmatrix}\right)}.
\end{multline}

Since $W_{1,v}^{\circ}=W_{3,v}^{\circ}$, by substituting \eqref{5.46} into \eqref{5.45}, together with  
making the change of variable $\beta\mapsto -\beta_1+\beta$, we obtain 
\begin{equation}\label{5.47}
\mathcal{J}_{v}^*(W_v^{(2)})=q_v\Vol(K_v^*(1))\sum_{i\in \mathbb{Z}}q_v^{\frac{i}{2}}\overline{\xi}_v^i(\varpi_v)\cdot J(i),
\end{equation}
where 
\begin{align*}
J(i):=\sum_{\alpha\in \mathbb{F}_v^{\times}}\overline{\xi}_v(\alpha)\bigg|\sum_{\beta\in \mathbb{F}_v}\overline{\psi}_v(\varpi_v^{-1}\beta)W_{1,v}^{\circ}\left(\begin{pmatrix}
\alpha\varpi_v^i\\
& 1
\end{pmatrix}
w\begin{pmatrix}
\varpi_v & \beta\\
& 1
\end{pmatrix}\right)\bigg|^2.
\end{align*}

The inner Whittaker function can be further simplified as follows. 
\begin{itemize}
\item Suppose $\beta\in \mathfrak{p}_v$. Then 
\begin{equation}\label{5.48}
W_{1,v}^{\circ}\left(\begin{pmatrix}
\alpha\varpi_v^i\\
& 1
\end{pmatrix}
w\begin{pmatrix}
\varpi_v & \beta\\
& 1
\end{pmatrix}\right)=W_{1,v}^{\circ}\left(\begin{pmatrix}
\varpi_v^i\\
& \varpi_v
\end{pmatrix}
\right).
\end{equation}

\item Suppose $\beta\in \mathbb{F}_v^{\times}$. It follows from 
\begin{align*}
w\begin{pmatrix}
\varpi_v & \beta\\
& 1
\end{pmatrix}=\begin{pmatrix}
\varpi_v & 1\\
& \beta
\end{pmatrix}\begin{pmatrix}
-\beta^{-1}& \\
\varpi_v\beta^{-1}&1\end{pmatrix}w
\end{align*}
that 
\begin{equation}\label{5.49}
W_{1,v}^{\circ}\left(\begin{pmatrix}
\alpha\varpi_v^i\\
& 1
\end{pmatrix}
w\begin{pmatrix}
\varpi_v & \beta\\
& 1
\end{pmatrix}\right)=\psi_v(\varpi_v^i\alpha\beta^{-1})W_{1,v}^{\circ}\left(\begin{pmatrix}
\varpi_v^{i+1} & \\
& 1
\end{pmatrix}
\right).
\end{equation}
\end{itemize}

Substituting \eqref{5.48} and \eqref{5.49} into the definition of $J(i)$, and opening the square and taking advantage of orthogonality, we obtain 
\begin{multline*}
J(i)=(q_v-1)\mathbf{1}_{r_{\xi_v}=0}\bigg|W_{1,v}^{\circ}\left(\begin{pmatrix}
\varpi_v^{i-1}\\
& 1
\end{pmatrix}
\right)\bigg|^2+\bigg|W_{1,v}^{\circ}\left(\begin{pmatrix}
\varpi_v^{i+1} & \\
& 1
\end{pmatrix}
\right)\bigg|^2S_i(\xi_v)\\
+2\Re \left(\overline{W_{1,v}^{\circ}\left(\begin{pmatrix}
\varpi_v^i\\
& \varpi_v
\end{pmatrix}
\right)}W_{1,v}^{\circ}\left(\begin{pmatrix}
\varpi_v^{i+1} & \\
& 1
\end{pmatrix}
\right)T_i(\xi_v)\right),
\end{multline*}
where $S_i(\xi_v)$ is defined by \eqref{eq5.45}, and 
\begin{align*}
T_i(\xi_v):=\sum_{\alpha\in \mathbb{F}_v^{\times}}\sum_{\beta\in \mathbb{F}_v^{\times}}\overline{\xi}_v(\alpha)\overline{\psi}_v(\varpi_v^{-1}\beta)\psi_v(\varpi_v^i\alpha\beta^{-1}).
\end{align*}

Note that $W_{1,v}^{\circ}(\diag(\varpi_v^i,\varpi_v))= 0$ unless $i\geq 1$. For $i\geq 1$, we have  
\begin{equation}\label{5.59}
T_i(\xi_v)=\sum_{\alpha\in \mathbb{F}_v^{\times}}\overline{\xi}_v(\alpha)\sum_{\beta\in \mathbb{F}_v^{\times}}\overline{\psi}_v(\varpi_v^{-1}\beta)=-(q_v-1)\mathbf{1}_{r_{\xi_v}=0}. 
\end{equation}

By \eqref{5.59} and Lemma \ref{lem5.6}, we obtain from \eqref{5.47} that 
\begin{multline}\label{5.61}
\mathcal{J}_{v}^*(W_v^{(2)})=q_v(q_v-1)\Vol(K_v^*(1)) \bigg[\sum_{i\in \mathbb{Z}}q_v^{\frac{i}{2}}\overline{\xi}_v^i(\varpi_v)\Big|W_{1,v}^{\circ}\left(\begin{pmatrix}
\varpi_v^{i-1}\\
& 1
\end{pmatrix}
\right)\Big|^2\\
-2\Re \left(\sum_{i\in \mathbb{Z}}q_v^{\frac{i}{2}}\overline{\xi}_v^i(\varpi_v)\overline{W_{1,v}^{\circ}\left(\begin{pmatrix}
\varpi_v^i\\
& \varpi_v
\end{pmatrix}
\right)}W_{1,v}^{\circ}\left(\begin{pmatrix}
\varpi_v^{i+1} & \\
& 1
\end{pmatrix}
\right)\right)\\
+\frac{|W_{1,v}^{\circ}(I_2)|^2(q_v^2-q_v-1)}{q_v^{1/2}(q_v-1)}
+\sum_{i\geq 0}q_v^{\frac{i}{2}}\overline{\xi}_v^i(\varpi_v)\Big|W_{1,v}^{\circ}\left(\begin{pmatrix}
\varpi_v^{i+1} & \\
& 1
\end{pmatrix}
\right)\Big|^2\bigg]\mathbf{1}_{r_{\xi_v}=0}\\
+q_v\Vol(K_v^*(1)) q_v^{-\frac{1}{2}}G(\overline{\xi}_v,\psi_v)G(\overline{\xi}_v,\overline{\psi}_v)J(\overline{\xi}_v,\overline{\xi}_v)|W_{1,v}^{\circ}(I_2)|^2\mathbf{1}_{r_{\xi_v}=1}.
\end{multline}

By Casselman--Shalika formula and a straightforward calculation we have
\begin{multline}\label{5.62}
\sum_{i\geq 0}q_v^{\frac{i}{2}}\overline{\xi}_v^i(\varpi_v)\overline{W_{1,v}^{\circ}\left(\begin{pmatrix}
\varpi_v^{i}\\
& \varpi_v
\end{pmatrix}
\right)}W_{1,v}^{\circ}\left(\begin{pmatrix}
\varpi_v^{i+1} & \\
& 1
\end{pmatrix}
\right)\\
=\overline{\xi}_v(\varpi_v)q_v^{-\frac{1}{2}}A\cdot L_v(1/2,\pi_{1,v}\times\widetilde{\pi}_{1,v}\times \overline{\xi}_v)|W_{1,v}^{\circ}(I_2)|^{2},
\end{multline}
and 
\begin{equation}\label{5.63}
\sum_{i\geq 0}q_v^{\frac{i}{2}}\overline{\xi}_v^i(\varpi_v)\Big|W_{1,v}^{\circ}\left(\begin{pmatrix}
\varpi_v^{i} & \\
& 1
\end{pmatrix}
\right)\Big|^2=\frac{L(1/2,\pi_{1,v}\times\widetilde{\pi}_{1,v}\times \overline{\xi}_v)|W_{1,v}^{\circ}(I_2)|^2}{L(1,\overline{\xi}_v^2)}.
\end{equation}

Therefore, \eqref{5.54} and \eqref{5.55} follow from \eqref{5.61}, \eqref{5.62}, \eqref{5.63}, and the fact that $\Vol(K_v^*(1))=q_v^{-1}\Vol(I_v^*(1))$. 
\end{proof}

\subsubsection{$\sigma_v$ is a special representation}
Suppose $\sigma_v=\mathrm{St}\otimes\chi_v$ is a special representation with $\chi_v^2=\mathbf{1}$ and $r_{\chi_v}\leq 1$. 

Suppose $\pi_{1,v}= \chi_{1,v}\boxplus \chi_{2,v}$ is unramified. Then $L(s,\pi_{1,v})=(1-\alpha_vq_v^{-s})^{-1}(1-\beta_vq_v^{-s})^{-1}$, where $\alpha_v=\chi_{1,v}(\varpi_v)$ and $\beta_v=\chi_{2,v}(\varpi_v)$. Let $W_{1,v}'=W_{1,v}^{\circ}\otimes\chi_{1,v}^{-1}$, which is a spherical vector in the Whittaker model of $\pi_{1,v}':=\pi_{1,v}\otimes\chi_{1,v}^{-1}=1\boxplus \chi_{2,v}\chi_{1,v}^{-1}$. Note that $L(s,\pi_{1,v}')=(1-q_v^{-s})^{-1}(1-\alpha_v^{-1}\beta_vq_v^{-s})^{-1}$.

Making the change of variable $x\mapsto \diag(\varpi_v^{-1},1)x$ into \eqref{e5.43}, together with the Iwasawa decomposition, we obtain 
\begin{multline}\label{e5.74}
\mathcal{J}_{v}(W_v^{\chi_v})=\zeta_v(1)q_v^{\frac{1}{2}}\sum_{i\in \mathbb{Z}}\int_{K_v}\overline{W_v^{\chi_v}}\left(\begin{pmatrix}
\varpi_v^i\\
& 1
\end{pmatrix}k\begin{pmatrix}
\varpi_v^{-1}\\
& 1
\end{pmatrix}\right)\\
\sum_{\beta\in \mathbb{F}_v}\overline{\psi}_v(\varpi_v^{-1}\beta)W_{1,v}'\left(\begin{pmatrix}
\varpi_v^i\\
& 1
\end{pmatrix}k\begin{pmatrix}
\varpi_v^{-1}\\
& 1
\end{pmatrix}k'\begin{pmatrix}
\varpi_v\\
& 1
\end{pmatrix}\right)
q_v^{\frac{i}{2}}dk.
\end{multline}

Consequently, we have the decomposition 
\begin{equation}\label{5.74}
\mathcal{J}_{v}(W_v^{\chi_v})=\mathcal{J}_{v}^{(1)}(W_v^{\chi_v})+\mathcal{J}_{v}^{(2)}(W_v^{\chi_v}),	
\end{equation}
where 
\begin{multline}\label{f5.76}
\mathcal{J}_{v}^{(1)}(W_v^{\chi_v})=\zeta_v(1)q_v^{\frac{1}{2}}\sum_{i\in \mathbb{Z}}\int_{K_{0,v}[1]}\overline{W_v^{\chi_v}}\left(\begin{pmatrix}
\varpi_v^i\\
& 1
\end{pmatrix}k\begin{pmatrix}
\varpi_v^{-1}\\
& 1
\end{pmatrix}\right)\\
\sum_{\beta\in \mathbb{F}_v}\overline{\psi}_v(\varpi_v^{-1}\beta)W_{1,v}'\left(\begin{pmatrix}
\varpi_v^i\\
& 1
\end{pmatrix}k\begin{pmatrix}
\varpi_v^{-1}\\
& 1
\end{pmatrix}k'\begin{pmatrix}
\varpi_v\\
& 1
\end{pmatrix}\right)
q_v^{\frac{i}{2}}dk,
\end{multline}
and
\begin{multline}\label{f5.77}
\mathcal{J}_{v}^{(2)}(W_v^{\chi_v})=\zeta_v(1)q_v^{\frac{1}{2}}\sum_{i\in \mathbb{Z}}\int_{K_v-K_{0,v}[1]}\overline{W_v^{\chi_v}}\left(\begin{pmatrix}
\varpi_v^i\\
& 1
\end{pmatrix}k\begin{pmatrix}
\varpi_v^{-1}\\
& 1
\end{pmatrix}\right)\\
\sum_{\beta\in \mathbb{F}_v}\overline{\psi}_v(\varpi_v^{-1}\beta)W_{1,v}'\left(\begin{pmatrix}
\varpi_v^i\\
& 1
\end{pmatrix}k\begin{pmatrix}
\varpi_v^{-1}\\
& 1
\end{pmatrix}k'\begin{pmatrix}
\varpi_v\\
& 1
\end{pmatrix}\right)
q_v^{\frac{i}{2}}dk.
\end{multline}

\begin{lemma}\label{lem5.9}
We have the following. 
\begin{itemize}
\item Suppose $r_{\chi_v}=1$. Then 
\begin{equation}\label{e5.76}
\mathcal{J}_{v}^{(1)}(W_v^{\chi_v})=-\zeta_v(1)\chi_v(\varpi_v)q_v^{\frac{1}{2}}J(\overline{\chi}_v,\overline{\chi}_v)\Vol(I_v^*(1))\overline{W_v^{\circ}(I_2)}W_{1,v}^{\circ}(I_2).
\end{equation}
\item Suppose $r_{\chi_v}=0$. Then 
\begin{multline}\label{e5.77}
\mathcal{J}_{v}^{(1)}(W_v^{\chi_v})=q_v^{\frac{1}{2}}\Vol(I_v^*(1))\overline{W_v^{\circ}}(I_2)W_{1,v}^{\circ}(I_2)\Big[\alpha_v\beta_v\overline{\chi}_v(\varpi_v)q_v^{-1}\\
+(1+\alpha_v^{-1}\beta_v-\alpha_v^{-1}\beta_vq_v^{-1})L_v(1,\pi_{1,v}')L_v(1,\pi_{1,v}'\times\overline{\chi}_v)^{-1}
-(1+\alpha_v^{-1}\beta_v)\\
+\chi_v(\varpi_v)(q_v-2)-\zeta_v(1)\chi_v(\varpi_v)(q_v-2)L_v(1,\pi_{1,v}'\times\overline{\chi}_v)^{-1}
\\
q_v^{-1}\chi_v(\varpi_v)(2L_v(1,\pi_{1,v}'\times\overline{\chi}_v)^{-1}-1)
+\alpha_v^{-1}\beta_v\overline{\chi}_v(\varpi_v)q_v^{-1}
\Big]L_v(1,\pi_{1,v}'\times\overline{\chi}_v).
\end{multline} 
\end{itemize}
\end{lemma}
\begin{proof}
By the coordinate \eqref{eq5.1}, 
\begin{align*}
k'=\begin{pmatrix}
1+\alpha\varpi_v & \beta\\
\gamma\varpi_v & 1+\delta\varpi_v 
\end{pmatrix}=\begin{pmatrix}
1& \beta\\
 & 1
\end{pmatrix}\begin{pmatrix}
1+(\alpha-\beta\gamma)\varpi_v & -\beta\delta\varpi_v\\
\gamma\varpi_v & 1+\delta\varpi_v 
\end{pmatrix}.
\end{align*}

Since $W_{1,v}'$ is spherical, we obtain 
\begin{align*}
W_{1,v}'\left(\begin{pmatrix}
\varpi_v^i\\
& 1
\end{pmatrix}k\begin{pmatrix}
\varpi_v^{-1}\\
& 1
\end{pmatrix}k'\begin{pmatrix}
\varpi_v\\
& 1
\end{pmatrix}\right)=W_{1,v}'\left(\begin{pmatrix}
\varpi_v^i\\
& 1
\end{pmatrix}k\begin{pmatrix}
1& \varpi_v^{-1}\beta\\
& 1
\end{pmatrix}\right).
\end{align*}

In conjunction with the decomposition  \eqref{eq5.13}, we obtain from \eqref{f5.76} that  
\begin{multline}\label{5.76}
\mathcal{J}_{v}^{(1)}(W_v^{\chi_v})=\zeta_v(1)q_v^{\frac{1}{2}}\sum_{\kappa\in \mathbb{F}_v^{\times}}\sum_{i\in \mathbb{Z}}\int_{I_v^*[1]}\overline{W_v^{\chi_v}}\left(\begin{pmatrix}
\kappa\varpi_v^i\\
& 1
\end{pmatrix}k\begin{pmatrix}
\varpi_v^{-1}\\
& 1
\end{pmatrix}\right)\\
\sum_{\beta\in \mathbb{F}_v}\overline{\psi}_v(\varpi_v^{-1}\beta)W_{1,v}'\left(\begin{pmatrix}
\kappa\varpi_v^i\\
& 1
\end{pmatrix}k\begin{pmatrix}
1& \varpi_v^{-1}\beta\\
& 1
\end{pmatrix}\right)
q_v^{\frac{i}{2}}dk,
\end{multline}

Write $k\in I_v^*(1)$ as 
\begin{equation}\label{5.77}
k=\kappa_1\begin{pmatrix}
1+\alpha_1\varpi_v & \beta_1\\
\gamma_1\varpi_v & 1+\delta_1\varpi_v 
\end{pmatrix}.
\end{equation}
Let $\delta_1'=\delta_1+\frac{\beta_1\gamma_1\delta_1\varpi_v}{1+(\alpha_1-\beta_1\gamma_1)\varpi_v}$. Then \eqref{5.77} boils down to 
\begin{align*}
k=\kappa_1\begin{pmatrix}
1 & \beta_1\\
& 1
\end{pmatrix}\begin{pmatrix}
1+(\alpha_1-\beta_1\gamma_1)\varpi_v &  \\
\gamma_1\varpi_v & 1+\delta_1'\varpi_v  
\end{pmatrix}\begin{pmatrix}
1 & -\frac{\beta_1\delta_1\varpi_v}{1+(\alpha_1-\beta_1\gamma_1)\varpi_v}\\
  & 1 
\end{pmatrix}.
\end{align*}

Since $r_{\chi_v}\leq 1$, it follows from the definition of $W_v^{\chi_v}$ and $W_{1,v}'$ that 
\begin{multline}\label{5.78}
\overline{W_v^{\chi_v}}\left(\begin{pmatrix}
\kappa\varpi_v^i\\
& 1
\end{pmatrix}k\begin{pmatrix}
\varpi_v^{-1}\\
& 1
\end{pmatrix}\right)W_{1,v}'\left(\begin{pmatrix}
\kappa\varpi_v^i\\
& 1
\end{pmatrix}k\begin{pmatrix}
1& \varpi_v^{-1}\beta\\
& 1
\end{pmatrix}\right)\\
=\overline{W_v^{\chi_v}}\left(\begin{pmatrix}
\kappa\varpi_v^{i-1} &  \\
\gamma_1 & 1  
\end{pmatrix}\right)W_{1,v}'\left(\begin{pmatrix}
\kappa\varpi_v^i\\
& 1
\end{pmatrix}\begin{pmatrix}
1 &  \\
\gamma_1\varpi_v & 1  
\end{pmatrix}\begin{pmatrix}
1& \varpi_v^{-1}\beta\\
& 1
\end{pmatrix}\right).
\end{multline}

Substituting \eqref{5.78} into \eqref{5.76} leads to  
\begin{multline}\label{5.79}
\mathcal{J}_{v}^{(1)}(W_v^{\chi_v})=\zeta_v(1)q_v^{\frac{1}{2}}\sum_{\kappa\in \mathbb{F}_v^{\times}}\sum_{\beta\in \mathbb{F}_v}\overline{\psi}_v(\varpi_v^{-1}\beta)
\sum_{i\in \mathbb{Z}}\int_{I_v^*(1)}\overline{W_v^{\chi_v}}\left(\begin{pmatrix}
\kappa\varpi_v^{i-1}&  \\
\gamma_1 & 1  
\end{pmatrix}\right)\\
W_{1,v}'\left(\begin{pmatrix}
\kappa\varpi_v^i\\
& 1
\end{pmatrix}\begin{pmatrix}
1 &  \\
\gamma_1\varpi_v & 1  
\end{pmatrix}\begin{pmatrix}
1& \varpi_v^{-1}\beta\\
 & 1
\end{pmatrix}\right)
q_v^{\frac{i}{2}}dk.
\end{multline}

Let $I_v^*(2)$ be the set of elements in $I_v^*(1)$ whose $(2,1)$-th entry lies in $\mathfrak{p}_v^2$. By \eqref{4.13} we can simplify \eqref{5.79} to 
\begin{multline}\label{5.80}
\mathcal{J}_{v}^{(1)}(W_v^{\chi_v})=\zeta_v(1)q_v^{\frac{1}{2}}\Vol(I_v^*(2))\sum_{\kappa\in \mathbb{F}_v^{\times}}\sum_{\beta\in \mathbb{F}_v}\overline{\psi}_v(\varpi_v^{-1}\beta)\sum_{i\geq 0}\sum_{\gamma_1\in \mathbb{F}_v}
\\
\overline{W_v^{\chi_v}}\left(\begin{pmatrix}
\kappa\varpi_v^{i-1}&  \\
\gamma_1 & 1  
\end{pmatrix}\right)W_{1,v}'\left(\begin{pmatrix}
\kappa\varpi_v^i\\
& 1
\end{pmatrix}\begin{pmatrix}
1 &  \\
\gamma_1\varpi_v & 1  
\end{pmatrix}\begin{pmatrix}
1& \varpi_v^{-1}\beta\\
 & 1
\end{pmatrix}\right)
q_v^{\frac{i}{2}}.
\end{multline}

Notice that $\Vol(I_v^*(2))=q_v^{-1}\Vol(I_v^*(1))$. 
According to $\beta\in \mathbb{F}_v^{\times}$ we have the following decomposition 
\begin{equation}\label{5.81}
\mathcal{J}_{v}^{(1)}(W_v^{\chi_v})=\mathcal{J}_{v}^{(1,1)}(W_v^{\chi_v})+\mathcal{J}_{v}^{(1,2)}(W_v^{\chi_v}),
\end{equation}
where $\mathcal{J}_{v}^{(1,1)}(W_v^{\chi_v})$ is defined by 
\begin{align*}
\zeta_v(1)q_v^{-\frac{1}{2}}\Vol(I_v^*(1))\sum_{i\geq 0}\sum_{\gamma_1\in \mathbb{F}_v}\sum_{\kappa\in \mathbb{F}_v^{\times}}
\overline{W_v^{\chi_v}}\left(\begin{pmatrix}
\kappa\varpi_v^{i-1}&  \\
\gamma_1 & 1  
\end{pmatrix}\right)W_{1,v}'\left(\begin{pmatrix}
\varpi_v^i\\
& 1
\end{pmatrix}\right)
q_v^{\frac{i}{2}},
\end{align*}
and 
\begin{multline*}
\mathcal{J}_{v}^{(1,2)}(W_v^{\chi_v}):=\zeta_v(1)q_v^{-\frac{1}{2}}\Vol(I_v^*(1))\sum_{\kappa\in \mathbb{F}_v^{\times}}\sum_{\beta\in \mathbb{F}_v^{\times}}\overline{\psi}_v(\varpi_v^{-1}\beta)\sum_{i\geq 0}\sum_{\gamma_1\in \mathbb{F}_v}
\\
\overline{W_v^{\chi_v}}\left(\begin{pmatrix}
\kappa\varpi_v^{i-1}&  \\
\gamma_1 & 1  
\end{pmatrix}\right)W_{1,v}'\left(\begin{pmatrix}
\kappa\varpi_v^i\\
& 1
\end{pmatrix}\begin{pmatrix}
1 &  \\
\gamma_1\varpi_v & 1  
\end{pmatrix}\begin{pmatrix}
1& \varpi_v^{-1}\beta\\
 & 1
\end{pmatrix}\right)
q_v^{\frac{i}{2}}.
\end{multline*}

The sums $\mathcal{J}_{v}^{(1,1)}(W_v^{\chi_v})$ and $\mathcal{J}_{v}^{(1,2)}(W_v^{\chi_v})$ are computed explicitly in Lemmas \ref{lem5.10} and \ref{lem5.11}, respectively. Therefore, \eqref{e5.76} and \eqref{e5.77} follow from \eqref{5.81}. 
\end{proof}

\begin{lemma}\label{lem5.10}
Let $\mathcal{J}_{v}^{(1,1)}(W_v^{\chi_v})$ be defined as in the proof of Lemma \ref{lem5.9}. Then
\begin{multline}\label{f5.85}
\mathcal{J}_{v}^{(1,1)}(W_v^{\chi_v})=q_v^{\frac{1}{2}}\Vol(I_v^*(1))\overline{W_v^{\circ}}(I_2)W_{1,v}^{\circ}(I_2)\mathbf{1}_{r_{\chi_v}=0}\Big[2q_v^{-1}\chi_v(\varpi_v)\\
+(1+\alpha_v^{-1}\beta_v-\alpha_v^{-1}\beta_vq_v^{-1})L_v(1,\pi_{1,v}')
-q_v^{-1}\chi_v(\varpi_v)L_v(1,\pi_{1,v}'\times\overline{\chi}_v)\Big].
\end{multline}
\end{lemma}
\begin{proof}
By definition, we may write 
\begin{equation}\label{f5.86}
\mathcal{J}_{v}^{(1,1)}(W_v^{\chi_v})=\mathcal{J}_{v,1}+\mathcal{J}_{v,2},
\end{equation}
where 
\begin{align*}
\mathcal{J}_{v,1}:=\zeta_v(1)q_v^{-\frac{1}{2}}\Vol(I_v^*(1))\sum_{i\geq 0}\sum_{\kappa\in \mathbb{F}_v^{\times}}
\overline{W_v^{\chi_v}}\left(\begin{pmatrix}
\kappa\varpi_v^{i-1}&  \\
 & 1  
\end{pmatrix}\right)W_{1,v}'\left(\begin{pmatrix}
\varpi_v^i\\
& 1
\end{pmatrix}\right)
q_v^{\frac{i}{2}},
\end{align*}
and $\mathcal{J}_{v,2}$ is given by 
\begin{align*}
\zeta_v(1)q_v^{-\frac{1}{2}}\Vol(I_v^*(1))\sum_{i\geq 0}\sum_{\gamma_1\in \mathbb{F}_v^{\times}}\sum_{\kappa\in \mathbb{F}_v^{\times}}
\overline{W_v^{\chi_v}}\left(\begin{pmatrix}
\kappa\varpi_v^{i-1}&  \\
\gamma_1 & 1  
\end{pmatrix}\right)W_{1,v}'\left(\begin{pmatrix}
\varpi_v^i\\
& 1
\end{pmatrix}\right)
q_v^{\frac{i}{2}}.
\end{align*}

By \eqref{4.12} and the Casselman--Shalika formula, we obtain 
\begin{multline}\label{5.82}
\sum_{i\geq 0}\sum_{\kappa\in \mathbb{F}_v^{\times}}
\overline{W_v^{\chi_v}}\left(\begin{pmatrix}
\kappa\varpi_v^{i-1}&  \\
 & 1  
\end{pmatrix}\right)W_{1,v}'\left(\begin{pmatrix}
\varpi_v^i\\
& 1
\end{pmatrix}\right)
q_v^{\frac{i}{2}}\\
=(q_v-1)\mathbf{1}_{r_{\chi_v}=0}(1+\alpha_v^{-1}\beta_v-\alpha_v^{-1}\beta_vq_v^{-1})L_v(1,\pi_{1,v}')\overline{W_v^{\circ}(I_2)}W_{1,v}^{\circ}(I_2).
\end{multline}
Here, we also use the fact that $W_{1,v}'(I_2)=W_{1,v}^{\circ}(I_2)$. 

On the other hand, when $\gamma_1\in \mathbb{F}_v^{\times}$, making the change of variables $\kappa\mapsto \gamma_1\kappa$, it follows from \eqref{4.11} and the orthogonality of characters that 
\begin{equation}\label{e5.86}
\sum_{\gamma_1, \kappa\in \mathbb{F}_v^{\times}}
\overline{W_v^{\chi_v}}\left(\begin{pmatrix}
\kappa\varpi_v^{i-1}&  \\
\gamma_1 & 1  
\end{pmatrix}\right)=(q_v-1)\mathbf{1}_{r_{\chi_v}=0}\sum_{\kappa\in \mathbb{F}_v^{\times}}
\overline{W_v^{\chi_v}}\left(\begin{pmatrix}
\kappa\varpi_v^{i-1}&  \\
1 & 1  
\end{pmatrix}\right).
\end{equation}

Combining \eqref{4.13} with \eqref{e5.86}, we obtain  
\begin{equation}\label{e5.87}
\mathcal{J}_{v,2}=-q_v^{-\frac{1}{2}}\mathbf{1}_{r_{\chi_v}=0}\Vol(I_v^*(1))\overline{W_v^{\circ}}(I_2)\mathcal{I}_{v,2},
\end{equation}
where 
\begin{align*}
\mathcal{I}_{v,2}:=\sum_{i\geq 0}\sum_{\kappa\in \mathbb{F}_v^{\times}}\overline{\psi}_v(\varpi_v^{i-1}\kappa)\overline{\chi}_v(\varpi_v^{i-1})q_v^{-\frac{i}{2}}W_{1,v}'\left(\begin{pmatrix}
\varpi_v^i\\
& 1
\end{pmatrix}\right).
\end{align*}

By computing the Ramanujan sum and applying the Casselman--Shalika formula, we obtain
\begin{equation}\label{e5.88}
\mathcal{I}_{v,2}=\chi_v(\varpi_v)(L_v(1,\pi_{1,v}'\times\overline{\chi}_v)-2)W_{1,v}^{\circ}(I_2).
\end{equation}

Therefore, \eqref{f5.85} follows from \eqref{f5.86}, \eqref{5.82}, \eqref{e5.87} and \eqref{e5.88}. 
\end{proof}

\begin{lemma}\label{lem5.11}
Let $\mathcal{J}_{v}^{(1,2)}(W_v^{\chi_v})$ be defined as in the proof of Lemma \ref{lem5.9}. Then
\begin{multline}\label{f5.90}
\mathcal{J}_{v}^{(1,2)}(W_v^{\chi_v})=
-\zeta_v(1)q_v^{-\frac{1}{2}}\Vol(I_v^*(1))\overline{W_v^{\circ}(I_2)}W_{1,v}'(I_2)\Big[
q_v(q_v-2)\mathbf{1}_{r_{\chi_v}=0}\\
+|G(\chi_v,\psi_v)|^2J(\overline{\chi}_v,\overline{\chi}_v)\mathbf{1}_{r_{\chi_v}=1}-(q_v-2)(q_v-1)L_v(1,\pi_{1,v}'\times\overline{\chi}_v)\mathbf{1}_{r_{\chi_v}=0}\Big]\chi_v(\varpi_v)\\
+\Vol(I_v^*(1))\overline{W_v^{\circ}(I_2)}W_{1,v}'(I_2)\mathbf{1}_{r_{\chi_v}=0}\Big[\alpha_v\beta_v\overline{\chi}_v(\varpi_v)q_v^{-\frac{1}{2}}L_v(1,\pi_{1,v}'\times\overline{\chi}_v)\\
-q_v^{\frac{1}{2}}
(1+\alpha_v^{-1}\beta_v-\alpha_v^{-1}\beta_v\overline{\chi}_v(\varpi_v)q_v^{-1})
L_v(1,\pi_{1,v}'\times\overline{\chi}_v)\Big].
\end{multline}
\end{lemma}
\begin{proof}
By definition, we may write 
\begin{equation}\label{f5.91}
\mathcal{J}_{v}^{(1,2)}(W_v^{\chi_v})=\mathcal{J}_{v,3}+\mathcal{J}_{v,4},
\end{equation}
where 
\begin{multline*}
\mathcal{J}_{v,3}:=\zeta_v(1)q_v^{-\frac{1}{2}}\Vol(I_v^*(1))\sum_{\kappa\in \mathbb{F}_v^{\times}}\sum_{\beta\in \mathbb{F}_v^{\times}}\overline{\psi}_v(\varpi_v^{-1}\beta)\\
\sum_{i\geq 0}\overline{W_v^{\chi_v}}\left(\begin{pmatrix}
\kappa\varpi_v^{i-1}&  \\
& 1  
\end{pmatrix}\right)W_{1,v}'\left(\begin{pmatrix}
\kappa\varpi_v^i\\
& 1
\end{pmatrix}\begin{pmatrix}
1& \varpi_v^{-1}\beta\\
 & 1
\end{pmatrix}\right)
q_v^{\frac{i}{2}},
\end{multline*}
and 
\begin{multline}\label{5.90}
\mathcal{J}_{v,4}:=\zeta_v(1)q_v^{-\frac{1}{2}}\Vol(I_v^*(1))\sum_{\kappa\in \mathbb{F}_v^{\times}}\sum_{\beta\in \mathbb{F}_v^{\times}}\overline{\psi}_v(\varpi_v^{-1}\beta)\sum_{i\geq 0}\sum_{\gamma_1\in \mathbb{F}_v^{\times}}
\\
\overline{W_v^{\chi_v}}\left(\begin{pmatrix}
\kappa\varpi_v^{i-1}&  \\
\gamma_1 & 1  
\end{pmatrix}\right)W_{1,v}'\left(\begin{pmatrix}
\kappa\beta\varpi_v^i\\
& 1
\end{pmatrix}\begin{pmatrix}
1 &  \\
\beta\gamma_1\varpi_v & 1  
\end{pmatrix}\begin{pmatrix}
1& \varpi_v^{-1}\\
 & 1
\end{pmatrix}\right)
q_v^{\frac{i}{2}}.
\end{multline}

By \eqref{4.11}, we obtain 
\begin{multline*}
\mathcal{J}_{v,3}=\zeta_v(1)q_v^{-\frac{1}{2}}\Vol(I_v^*(1))\sum_{\kappa\in \mathbb{F}_v^{\times}}\overline{\chi}_v(\kappa)\sum_{\beta\in \mathbb{F}_v^{\times}}\overline{\psi}_v(\varpi_v^{-1}\beta)
\\
\sum_{i\geq 1}\overline{\chi}_v(\varpi_v^{i-1})\overline{W_v^{\circ}}\left(\begin{pmatrix}
\varpi_v^{i-1}&  \\
& 1  
\end{pmatrix}\right)W_{1,v}'\left(\begin{pmatrix}
\varpi_v^i\\
& 1
\end{pmatrix}\right)
q_v^{\frac{i}{2}}.
\end{multline*}

Analogous to the treatments in the proof of Lemma \ref{lem5.10} we obtain by orthogonality of characters and the Cassel--Shalika formula that 
\begin{multline}\label{f5.92}
\mathcal{J}_{v,3}=-q_v^{\frac{1}{2}}\Vol(I_v^*(1))
(1+\alpha_v^{-1}\beta_v-\alpha_v^{-1}\beta_v\overline{\chi}_v(\varpi_v)q_v^{-1})\\
L_v(1,\pi_{1,v}'\times\overline{\chi}_v)\overline{W_v^{\circ}(I_2)}W_{1,v}'(I_2)\mathbf{1}_{r_{\chi_v}=0}.
\end{multline}

Now we proceed to compute $\mathcal{J}_{v,4}$. Consider the following scenario according to the valuation $1+\beta\gamma_1$ as follows. 
\begin{itemize}
\item Suppose $\gamma_2:=1+\beta\gamma_1\in \mathbb{F}_v^{\times}$. Then 
\begin{equation}\label{e5.91}
\begin{pmatrix}
1& \\
\beta\gamma_1\varpi_v &1
\end{pmatrix}\begin{pmatrix}
1& \varpi_v^{-1}\\
&1
\end{pmatrix}=\begin{pmatrix}
1& \gamma_2^{-1}\varpi_v^{-1}\\
&1
\end{pmatrix}\begin{pmatrix}
\gamma_2^{-1} & \\
\beta\gamma_1\varpi_v & \gamma_2
\end{pmatrix}.
\end{equation}
By \eqref{4.13} we have
\begin{equation}\label{5.91}
\overline{W_v^{\chi_v}}\left(\begin{pmatrix}
\kappa\varpi_v^{i-1}\\
\gamma_1 & 1
\end{pmatrix}\right)=-\overline{\psi}_v(\varpi_v^{i-1}\kappa\gamma_1^{-1})\overline{\chi}_v(\kappa\varpi_v^{i-1})q_v^{-i}\overline{W_v^{\circ}(I_2)}\mathbf{1}_{i\geq 0}. 
\end{equation}

Substituting \eqref{e5.91} and \eqref{5.91} into \eqref{5.90} we obtain that the contribution from $\gamma_2=1+\beta\gamma_1\in \mathbb{F}_v^{\times}$ to $\mathcal{J}_{v,4}$ is 
\begin{equation}\label{e5.80}
-\zeta_v(1)\chi_v(\varpi_v)q_v^{-\frac{1}{2}}\Vol(I_v^*(1))\overline{W_v^{\circ}(I_2)}\sum_{i\geq 0}\overline{\chi}_v(\varpi_v^{i})q_v^{-\frac{i}{2}}W_{1,v}'\left(\begin{pmatrix}
\varpi_v^i\\
& 1
\end{pmatrix}\right)A_i,
\end{equation}
where 
\begin{align*}
A_i:=\sum_{\kappa\in \mathbb{F}_v^{\times}}\overline{\chi}_v(\kappa)\mathop{\sum\sum}\limits_{\substack{\beta, \gamma_1\in \mathbb{F}_v^{\times}\\ 1+\beta\gamma_1\in \mathbb{F}_v^{\times}}}\overline{\psi}_v(\varpi_v^{-1}\beta)\overline{\psi}_v(\varpi_v^{i-1}\kappa\gamma_1^{-1})\psi_v(\varpi_v^{i-1}\kappa\beta(1+\beta\gamma_1)^{-1}).
\end{align*}

Making the change of variables $\kappa\mapsto \gamma_1(1+\beta\gamma_1)\kappa$ and $\gamma_1\mapsto \beta^{-1}\gamma_1$,  
\begin{align*}
A_i=\sum_{\beta\in \mathbb{F}_v^{\times}}\chi_v(\beta)\overline{\psi}_v(\varpi_v^{-1}\beta)\sum_{\kappa\in \mathbb{F}_v^{\times}}\overline{\chi}_v(\kappa)\overline{\psi}_v(\varpi_v^{i-1}\kappa)\sum_{\gamma_1\in \mathbb{F}_v-\{0,-1\}}\overline{\chi}_v(\gamma_1(1+\gamma_1)).
\end{align*}

By the change of variable $\beta\mapsto -\beta$ and $\gamma_1\mapsto -\gamma_1$, we obtain 
\begin{multline}\label{5.92}
A_i=|G(\chi_v,\psi_v)|^2J(\overline{\chi}_v,\overline{\chi}_v)\mathbf{1}_{i=0}\mathbf{1}_{r_{\chi_v}=1}+(q_v-2)\mathbf{1}_{i=0}\mathbf{1}_{r_{\chi_v}=0}\\
-(q_v-2)(q_v-1)\mathbf{1}_{i\geq 1}\mathbf{1}_{r_{\chi_v}=0}.
\end{multline}

By \eqref{5.92} and a direct calculation, the expression  \eqref{e5.80} becomes  
\begin{multline}\label{5.93}
-\zeta_v(1)\chi_v(\varpi_v)q_v^{-\frac{1}{2}}\Vol(I_v^*(1))\overline{W_v^{\circ}(I_2)}W_{1,v}'(I_2)\Big[
q_v(q_v-2)\mathbf{1}_{r_{\chi_v}=0}\\
+|G(\chi_v,\psi_v)|^2J(\overline{\chi}_v,\overline{\chi}_v)\mathbf{1}_{r_{\chi_v}=1}-(q_v-2)(q_v-1)L_v(1,\pi_{1,v}'\times\overline{\chi}_v)\mathbf{1}_{r_{\chi_v}=0}\Big].
\end{multline}

\item Suppose $1+\beta\gamma_1\in \mathfrak{p}_v$. Then 
\begin{equation}\label{5.97}
\begin{pmatrix}
1& \\
\beta\gamma_1\varpi_v &1
\end{pmatrix}\begin{pmatrix}
1& \varpi_v^{-1}\\
&1
\end{pmatrix}=\begin{pmatrix}
\varpi_v^{-1}\\
& \varpi_v
\end{pmatrix}\begin{pmatrix}
\varpi_v & 1\\
\beta\gamma_1  & \varpi_v^{-1}(1+\beta\gamma_1)
\end{pmatrix}.
\end{equation}

Substituting  \eqref{5.91} and \eqref{5.97} into \eqref{5.90} we obtain that the contribution from $1+\beta\gamma_1\in \mathfrak{p}_v$ to $\mathcal{J}_{v,4}$ is 
\begin{equation}\label{5.98}
-\frac{\zeta_v(1)\Vol(I_v^*(1))\overline{W_v^{\circ}(I_2)}}{\overline{\chi}_v(\varpi_v)q_v^{1/2}}
\sum_{i\geq 2}\overline{\chi}_v(\varpi_v^{i})q_v^{-\frac{i}{2}}
W_{1,v}'\left(\begin{pmatrix}
\varpi_v^{i-1}\\
& \varpi_v
\end{pmatrix}\right)B_i,
\end{equation}
where
\begin{align*}
B_i:=\sum_{\kappa\in \mathbb{F}_v^{\times}}\sum_{\beta\in \mathbb{F}_v^{\times}}\overline{\psi}_v(\varpi_v^{-1}\beta)\psi_v(\varpi_v^{i-1}\kappa\beta)\overline{\chi}_v(\kappa).
\end{align*}

As a consequence of the explicit formula 
\begin{align*}
B_i=|G(\chi_v,\psi_v)|^2\mathbf{1}_{i=0}\mathbf{1}_{r_{\chi_v}=1}+\mathbf{1}_{i=0}\mathbf{1}_{r_{\chi_v}=0}-(q_v-1)\mathbf{1}_{i\geq 1}\mathbf{1}_{r_{\chi_v}=0},
\end{align*}
the expression \eqref{5.98} boils down to 
\begin{equation}\label{5.99}
\alpha_v\beta_v\overline{\chi}_v(\varpi_v)q_v^{-\frac{1}{2}}\Vol(I_v^*(1))L_v(1,\pi_{1,v}'\times\overline{\chi}_v)\overline{W_v^{\circ}(I_2)}
\mathbf{1}_{r_{\chi_v}=0}.
\end{equation}
\end{itemize}

Combining \eqref{5.93} with \eqref{5.99} we obtain  
\begin{multline}\label{5.100}
\mathcal{J}_{v,4}=-\zeta_v(1)\chi_v(\varpi_v)q_v^{-\frac{1}{2}}\Vol(I_v^*(1))\overline{W_v^{\circ}(I_2)}W_{1,v}'(I_2)\Big[
q_v(q_v-2)\mathbf{1}_{r_{\chi_v}=0}\\
+|G(\chi_v,\psi_v)|^2J(\overline{\chi}_v,\overline{\chi}_v)\mathbf{1}_{r_{\chi_v}=1}-(q_v-2)(q_v-1)L_v(1,\pi_{1,v}'\times\overline{\chi}_v)\mathbf{1}_{r_{\chi_v}=0}\Big]\\
+\alpha_v\beta_v\overline{\chi}_v(\varpi_v)L_v(1,\pi_{1,v}'\times\overline{\chi}_v)q_v^{-\frac{1}{2}}\Vol(I_v^*(1))\overline{W_v^{\circ}(I_2)}W_{1,v}'(I_2)
\mathbf{1}_{r_{\chi_v}=0}.
\end{multline}

Therefore, \eqref{f5.90} follows from \eqref{f5.91}, \eqref{f5.92} and \eqref{5.100}. 
\end{proof}

\begin{lemma}\label{lem5.12}
We have the following. 
\begin{itemize}
\item Suppose $r_{\chi_v}=1$. Then 
\begin{equation}\label{5.103}
\mathcal{J}_{v}^{(2)}(W_v^{\chi_v})=-\alpha_v\beta_v^{-1}\zeta_v(1)\chi_v(-\varpi_v)q_v^{\frac{3}{2}}\Vol(I_v^*(1))\overline{W_v^{\circ}}(I_2)W_{1,v}^{\circ}(I_2).
\end{equation}
\item Suppose $r_{\chi_v}=0$. Then 
\begin{multline}\label{5.104}
\mathcal{J}_{v}^{(2)}(W_v^{\chi_v})=-\chi_v(-\varpi_v)q_v^{\frac{1}{2}}\Vol(I_v^*(1))\overline{W_v^{\circ}}(I_2)W_{1,v}'(I_2)\Big[L(1,\pi_{1,v}'\times \overline{\chi}_v)\\
+\alpha_v\beta_v^{-1}\zeta_v(1)-\alpha_v\beta_v^{-1}q_v(L(1,\pi_{1,v}'\times\overline{\chi}_v)-1)\Big].
\end{multline} 
\end{itemize}
\end{lemma}
\begin{proof}
By Cramer's rule we have 
\begin{align*}
K_v-K_{0,v}[1]=\bigsqcup_{\kappa\in \mathbb{F}_v^{\times}}\begin{pmatrix}
\kappa\\
  & 1
\end{pmatrix}K_v^*,\ \text{where} \ \ K_v^*:=\bigsqcup_{\beta_3\in \mathbb{F}_v}\begin{pmatrix}
1+\mathfrak{p}_v & \beta_3\\
& 1
\end{pmatrix}
wI_v^*(1).
\end{align*}

Let $k\in K_v-K_{0,v}[1]$. We may write 
\begin{equation}\label{e5.105}
k=\begin{pmatrix}
\kappa & \\
& 1
\end{pmatrix}k_1,\ \ \ k_1=\kappa_1\begin{pmatrix}
1 & \beta_1\\
& 1
\end{pmatrix}w\begin{pmatrix}
1 & \beta_2\\
\gamma_2 & 1
\end{pmatrix}\in K_v^*,
\end{equation}
where $\kappa_1, \kappa_2\in \mathbb{F}_v^{\times}$, $\beta_1\in \mathcal{O}_v/\mathfrak{p}_v^2$, $\gamma_2\in \mathfrak{p}_v^2$, and $\beta_2\in \mathcal{O}_v$. Notice that 
\begin{align*}
\begin{pmatrix}
1& -\varpi_v^{-1}\beta\\
& 1
\end{pmatrix}\begin{pmatrix}
1 & \beta_2\\
\gamma_2 & 1
\end{pmatrix}\begin{pmatrix}
1& \varpi_v^{-1}\beta\\
& 1
\end{pmatrix}=\begin{pmatrix}
1-\varpi_v^{-1}\beta \gamma_2 & \beta_2-\varpi_v^{-2}\beta^2\gamma_2\\
\gamma_2 & 1+\varpi_v^{-1}\beta \gamma_2
\end{pmatrix}\in K_v. 
\end{align*}

Hence, with $k$ given in the coordinate \eqref{e5.105}, we have
\begin{multline}\label{e5.106}
\overline{W_v^{\chi_v}}\left(\begin{pmatrix}
\varpi_v^i\\
& 1
\end{pmatrix}k\begin{pmatrix}
\varpi_v^{-1}\\
& 1
\end{pmatrix}\right)W_{1,v}'\left(\begin{pmatrix}
\varpi_v^i\\
& 1
\end{pmatrix}k\begin{pmatrix}
1& \varpi_v^{-1}\beta\\
& 1
\end{pmatrix}\right)\\
=\overline{W_v^{\chi_v}}\left(\begin{pmatrix}
\kappa\varpi_v^i\\
& 1
\end{pmatrix}w\begin{pmatrix}
\varpi_v^{-1}\\
& 1
\end{pmatrix}\right)W_{1,v}'\left(\begin{pmatrix}
\kappa\varpi_v^i\\
& 1
\end{pmatrix}w\begin{pmatrix}
1& \varpi_v^{-1}\beta\\
& 1
\end{pmatrix}\right),
\end{multline}
which is independent of the $K_v^*$-part. 

Substituting \eqref{e5.106} into \eqref{f5.77}, together with \eqref{4.11}, we obtain 
\begin{multline}\label{5.110}
\mathcal{J}_{v}^{(2)}(W_v^{\chi_v})=\chi_v(-\varpi_v)\zeta_v(1)q_v^{\frac{1}{2}}\Vol(K_v^*)\sum_{\kappa\in \mathbb{F}_v^{\times}}\overline{\chi}_v(\kappa)\sum_{i\in \mathbb{Z}}\overline{\chi}_v(\varpi_v^i)q_v^{\frac{i}{2}}\\
\overline{W_v^{\circ}}\left(\begin{pmatrix}
\varpi_v^{i+1}\\
& 1
\end{pmatrix}w\right)
\sum_{\beta\in \mathbb{F}_v}\overline{\psi}_v(\varpi_v^{-1}\beta)W_{1,v}'\left(\begin{pmatrix}
\kappa\varpi_v^i\\
& 1
\end{pmatrix}\begin{pmatrix}
1& \\
\varpi_v^{-1}\beta & 1
\end{pmatrix}\right).
\end{multline}

Note that $\Vol(K_v^*)=q_v\Vol(I_v^*(1))$. Making the change of variables $i\mapsto i-2$ in \eqref{5.110}, and utilizing Lemma \ref{lem4.3}, we obtain the decomposition 
\begin{equation}\label{5.111}
\mathcal{J}_{v}^{(2)}(W_v^{\chi_v})=\mathcal{J}_{v}^{(2,1)}(W_v^{\chi_v})+\mathcal{J}_{v}^{(2,2)}(W_v^{\chi_v}),
\end{equation}
where
\begin{multline*}
\mathcal{J}_{v}^{(2,1)}(W_v^{\chi_v})=-\chi_v(-\varpi_v)\zeta_v(1)q_v^{\frac{1}{2}}\Vol(I_v^*(1))\overline{W_v^{\circ}}(I_2)\\
\sum_{\kappa\in \mathbb{F}_v^{\times}}\overline{\chi}_v(\kappa)\sum_{i\geq 0}\overline{\chi}_v(\varpi_v^i)q_v^{-\frac{i}{2}}
W_{1,v}'\left(\begin{pmatrix}
\varpi_v^{i-2}\\
& 1
\end{pmatrix}\right),
\end{multline*}
and 
\begin{multline}\label{e5.113}
\mathcal{J}_{v}^{(2,2)}(W_v^{\chi_v})=-\chi_v(-\varpi_v)\zeta_v(1)q_v^{\frac{1}{2}}\Vol(I_v^*(1))\overline{W_v^{\circ}}(I_2)\\
\sum_{i\geq 0}\overline{\chi}_v(\varpi_v^i)q_v^{-\frac{i}{2}}\sum_{\beta\in \mathbb{F}_v^{\times}}\overline{\psi}_v(\varpi_v^{-1}\beta)\sum_{\kappa\in \mathbb{F}_v^{\times}}\overline{\chi}_v(\kappa)W_{1,v}'\left(\begin{pmatrix}
\kappa\varpi_v^{i-2} & \\
\varpi_v^{-1}\beta & 1
\end{pmatrix}\right).
\end{multline}

Since $\overline{\chi}_v$ is quadratic, it follows from the change of variable $i\mapsto i+2$ that 
\begin{align*}
\sum_{i\geq 0}\overline{\chi}_v(\varpi_v^i)q_v^{-\frac{i}{2}}
W_{1,v}'\left(\begin{pmatrix}
\varpi_v^{i-2}\\
& 1
\end{pmatrix}\right)=q_v^{-1}L(1,\pi_{1,v}'\times \overline{\chi}_v)W_{1,v}'(I_2).
\end{align*}

Therefore, we obtain 
\begin{multline}\label{5.113}
\mathcal{J}_{v}^{(2,1)}(W_v^{\chi_v})=-\chi_v(-\varpi_v)q_v^{\frac{1}{2}}L(1,\pi_{1,v}'\times \overline{\chi}_v)\mathbf{1}_{r_{\chi_v}=0}\\
\Vol(I_v^*(1))\overline{W_v^{\circ}}(I_2)W_{1,v}'(I_2).
\end{multline}

By a straightforward calculation we obtain  
\begin{equation}\label{5.114}
\begin{pmatrix}
\kappa\varpi_v^{i-2}& \\
\varpi_v^{-1}\beta & 1
\end{pmatrix}=\begin{pmatrix}
\kappa\varpi_v^{i-2}\\
& 1
\end{pmatrix}\begin{pmatrix}
\varpi_v & \beta^{-1}\\
 & \varpi_v^{-1}
\end{pmatrix}\begin{pmatrix}
& -\beta^{-1}\\
\beta & \varpi_v
\end{pmatrix}
\end{equation}
Recall that $W_{1,v}'$ is spherical with central character $\chi_{2,v}\chi_{1,v}^{-1}$. By \eqref{5.114}, 
\begin{equation}\label{5.115}
W_{1,v}'\left(\begin{pmatrix}
\kappa\varpi_v^{i-2} & \\
\varpi_v^{-1}\beta & 1
\end{pmatrix}\right)=\chi_{1,v}\chi_{2,v}^{-1}(\varpi_v)\psi_v(\varpi_v^{i-1}\kappa\beta^{-1})W_{1,v}'\left(\begin{pmatrix}
\varpi_v^{i}\\
& 1
\end{pmatrix}\right).
\end{equation}

Note that $\chi_{1,v}\chi_{2,v}^{-1}(\varpi_v)=\alpha_v\beta_v^{-1}$. Substituting \eqref{5.115} into \eqref{e5.113}, along with the change of variables $\kappa\mapsto \beta\kappa$, we have 
\begin{multline}\label{5.117}
\mathcal{J}_{v}^{(2,2)}(W_v^{\chi_v})=-\alpha_v\beta_v^{-1}\chi_v(-\varpi_v)\zeta_v(1)q_v^{\frac{1}{2}}\Vol(I_v^*(1))\overline{W_v^{\circ}}(I_2)\\
\Big[\overline{G(\chi_v,\psi_v)}\mathbf{1}_{r_{\chi_v}=1}-\mathbf{1}_{r_{\chi_v}=0}\Big]\sum_{i\geq 0}\overline{\chi}_v(\varpi_v^i)q_v^{-\frac{i}{2}}\sum_{\kappa\in \mathbb{F}_v^{\times}}\overline{\chi}_v(\kappa)\psi_v(\varpi_v^{i-1}\kappa)W_{1,v}'\left(\begin{pmatrix}
\varpi_v^{i}\\
& 1
\end{pmatrix}\right).
\end{multline}

Combining \eqref{e5.13} and \eqref{e5.20} yields 
\begin{multline}\label{5.118}
\sum_{\kappa\in \mathbb{F}_v^{\times}}\overline{\chi}_v(\kappa)\psi_v(\varpi_v^{i-1}\kappa)=G(\overline{\chi}_v,\psi_v)\mathbf{1}_{i=0}\mathbf{1}_{r_{\chi_v}=1}-\mathbf{1}_{i=0}\mathbf{1}_{r_{\chi_v}=0}\\
+(q_v-1)\mathbf{1}_{i\geq 1}\mathbf{1}_{r_{\chi_v}=0}.
\end{multline}

Substituting \eqref{5.118} into the \eqref{5.117} gives 
\begin{multline}\label{5.119}
\mathcal{J}_{v}^{(2,2)}(W_v^{\chi_v})=-\alpha_v\beta_v^{-1}\chi_v(-\varpi_v)\zeta_v(1)q_v^{\frac{1}{2}}\Vol(I_v^*(1))\overline{W_v^{\circ}}(I_2)W_{1,v}^{\circ}(I_2)\\
\Big[\overline{G(\chi_v,\psi_v)}G(\overline{\chi}_v,\psi_v)\mathbf{1}_{r_{\chi_v}=1}
-\big[(q_v-1)(L(1,\pi_{1,v}'\times\overline{\chi}_v)-1)-1\big]\mathbf{1}_{r_{\chi_v}=0}\Big].
\end{multline}

Therefore, it follows from \eqref{5.111}, \eqref{5.113} and \eqref{5.119} that 
\begin{multline}\label{5.120}
\mathcal{J}_{v}^{(2)}(W_v^{\chi_v})=-q_v^{\frac{1}{2}}\Vol(I_v^*(1))\overline{W_v^{\circ}}(I_2)W_{1,v}'(I_2)\Big[L(1,\pi_{1,v}'\times \overline{\chi}_v)\mathbf{1}_{r_{\chi_v}=0}\\
+\alpha_v\beta_v^{-1}\zeta_v(1)q_v\mathbf{1}_{r_{\chi_v}=1}
-\alpha_v\beta_v^{-1}\zeta_v(1)\big[(q_v-1)(L(1,\pi_{1,v}'\times\overline{\chi}_v)-1)-1\big]\mathbf{1}_{r_{\chi_v}=0}\Big]\chi_v(-\varpi_v).
\end{multline}
Here, we have also used the fact that $\chi_v$ is quadratic, which implies
\begin{align*}
\overline{G(\chi_v,\psi_v)}G(\overline{\chi}_v,\psi_v)=\overline{G(\chi_v,\psi_v)}G(\chi_v,\psi_v)=q_v. 
\end{align*}

Consequently, \eqref{5.103} and \eqref{5.104} follows from \eqref{5.120}. 
\end{proof}

\section{Choice of Automorphic Data}\label{sec6}
\subsection{Construction of \texorpdfstring{$\boldsymbol{h}$}{}}
Let $\mathfrak{q}$ be a prime ideal, and let
$\omega=\otimes_v'\omega_v$ be a unitary Hecke character of
$F^{\times}\backslash\mathbb{A}_F^{\times}$. 
\begin{itemize}
\item For $v<\infty$, we suppose that the conductor exponent of $\omega_v$ is $\leq 1_{v\mid\mathfrak{q}}$. 
\item For $v\mid\infty$, let $\mathbf{C}_v\ge C_v(\omega)$, where
$C_v(\omega)$ denotes the analytic conductor of $\omega_v$. 
Set $\mathbf{C}_{\infty}:=\prod_{v\mid\infty}\mathbf{C}_v$.
\end{itemize}

\begin{itemize}
\item Let $v\mid\infty$. Fix a nonnegative smooth function $\alpha_v$ on $\mathbb{R}$,
supported in the interval $|t|\le 10^{-1}$ and satisfying $\alpha_v(t)\equiv 1$
for $|t|\le 20^{-1}$. For $t_{1,v},t_{2,v}\in F_v$, we define
\begin{align*}
\Phi_v(t_{1,v},t_{2,v})
:=\mathbf{C}_v^{1/2}\omega_v(t_{2,v})
\alpha_v(\mathbf{C}_v|t_{1,v}|_v)\alpha_v(|t_{2,v}|_v-1).
\end{align*}

\item Let $v\mid\mathfrak{q}$. We set, for $t_{1,v},t_{2,v}\in F_v$, that 
\begin{equation}\label{f2.4}
\Phi_v(t_{1,v},t_{2,v})
:=\Vol(I_v^*(1))^{-1}
\mathbf{1}_{\mathfrak{p}_v}(t_1)\mathbf{1}_{\mathcal{O}_v^{\times}}(t_2)
\psi_v(\varpi_v^{-2} t t_1 t_2^{-1})\omega_v(t_2).
\end{equation}

\item For the remaining places $v$, namely $v<\infty$ with $v\nmid\mathfrak{q}$, we set
\begin{align*}
\Phi_v(t_{1,v},t_{2,v})
:=\mathbf{1}_{\mathcal{O}_v}(t_{1,v})\mathbf{1}_{\mathcal{O}_v}(t_{2,v}).
\end{align*}
\end{itemize}

For $t_i=\otimes_v t_{i,v}\in \mathbb{A}_F$, $1\le i\le 2$, set $\Phi(t_1,t_2):=\otimes_v \Phi_v(t_{1,v},t_{2,v})\in \mathcal{S}(\mathbb{A}_F^2)$. Let 
\begin{align*}
A_{\mathfrak{p}}:=\diag(\varpi_v,1),\ \ \text{where $v\mid\mathfrak{p}$.}
\end{align*}

For $s\in \mathbb{C}$, let $h_1^{\circ}(\cdot,s)=h_3^{\circ}(\cdot,s)$ be a fixed spherical vector in $|\cdot|^{s}\boxplus |\cdot|^{-s}$. Take 
\begin{align*}
h_1(g,s)=h_3(g,s)=h_1^{\circ}(gA_{\mathfrak{p}},s).
\end{align*} 
When $\Re(s)\ggg 1$, we define
\begin{equation}\label{e6.2}
h_2(g,s)=h_4(g,s)
:=|\det g|^{s+\frac{1}{2}}
\int_{\mathbb{A}_F^{\times}}
\Phi((0,t)g)\,\omega^{-1}(t)\,|t|^{2s+1}\,d^{\times}t.
\end{equation}
By Tate's thesis, this admits a meromorphic continuation to all $s\in\mathbb{C}$, which we continue to denote by $h_2(g,s)$ and $h_4(g,s)$. Let $W_j(\cdot,s)$ be the Whittaker function associated with $h_j(\cdot,s)$. 

For $1\leq j\leq 4$, we write $h_j(\cdot):=h_j(\cdot,0)$ and $W_j(\cdot)=W_j(\cdot,s)$ for simplicity. Hence, $h_{2,v}$, defined by \eqref{f5.2}, is the $v$-th component of $h_2$. Moreover, when $\pi_1=|\cdot|^{s}\boxplus |\cdot|^{-s}$, we have
\begin{align*}
W_{1}(x,s)=\int_{N(\mathbb{A}_F)}h_{1}(wux,s)\overline{\theta(u)}du=W_{1}^{\circ}(xA_{\mathfrak{p}},s),
\end{align*}
where $W_{1}^{\circ}$ is the spherical Whittaker vector corresponding to $h_1^{\circ}(\cdot,s)$.

\subsection{Construction of $f_{\mathfrak{q}}$}
Let $v\mid\mathfrak{q}$ and let $\sigma_t^{\zeta}$ be the simple supercuspidal representation of $\mathrm{GL}_2(F_v)$ parametrized by $t\in \mathbb{F}_v^{\times}$, with central character $\omega_v$ and $\varepsilon$-factor $\zeta$.  
Define $f_{\mathfrak{q}}:=f_v^t$ as in \eqref{2.3}.

Let 
$\chi_j=\omega_1=\omega_3=\mathbf{1}$ for  $1\leq j\leq 4$, and $\omega_2=\omega_4=\omega$.  
Set $\pi_1=\pi_3=\mathbf{1}\boxplus\mathbf{1}$ and 
$\pi_2=\pi_4=\mathbf{1}\boxplus\omega$.  
Although $\pi_1=\pi_3$ and $\pi_2=\pi_4$ (and similarly $h_1=h_3$, $h_2=h_4$), we retain the indices for compatibility with the symmetric spectral reciprocity formula: 
\SP*

\section{Lower Bound for \texorpdfstring{$\mathcal{J}_{\mathrm{Spec}}^{\heartsuit}(\mathbf{0},R(f_{\mathfrak{q}})\boldsymbol{h})$}{} and Vanishing of $\widetilde{\Psi}_{\mathrm{RS}}^{(i)}$}\label{sec7}
In this section we establish the following lower bound for $\mathcal{J}_{\mathrm{Spec}}^{\heartsuit}(\mathbf{0},R(f_{\mathfrak{q}})\boldsymbol{h})$. 
\begin{prop}\label{prop7.1}
We have
\begin{equation}\label{e7.1}
\mathcal{J}_{\mathrm{Spec}}^{\heartsuit}(\mathbf{0},R(f_{\mathfrak{q}})\boldsymbol{h})\gg \mathbf{C}_{\infty}^{-1-\varepsilon}N_F(\mathfrak{q})^{-\varepsilon}\sum_{\substack{\pi\in \mathcal{F}_{t}^{\zeta}(\mathfrak{q}^3;\omega)\bigcup \mathcal{F}_{t}^{-\zeta}(\mathfrak{q}^3;\omega)\\
C_v(\pi)\ll C_v,\ v\mid\infty}}|L(1/2,\pi)|^4,
\end{equation}	
and the vanishing of the degenerate terms 
\begin{equation}\label{e7.2}
\widetilde{\Psi}_{\mathrm{RS}}^{(i)}(\mathbf{s},R(f_{\mathfrak{q}})\boldsymbol{h})\equiv 0,\ \ \ 1\leq i\leq 8.
\end{equation}
\end{prop}
\begin{proof}
Let $v\mid\mathfrak{q}$ and let $W_v$ be a Whittaker function of an irreducible admissible generic representation $\pi_v$ of $G(F_v)$.  
By a change of variables,
\begin{multline*}
\int_{\overline{G}(F_v)}
\overline{f_{\mathfrak{q}}(g)}
\,\widetilde{\Psi}_v\!\left(\overline{W_v},
R(g^{-1})W_{1,v}(\cdot,s_1),
R(g^{-1})h_{2,v}(\cdot,s_2)\right)dg \\
=
\widetilde{\Psi}_v\!\left(
\overline{R(f_{\mathfrak{q}})W_v},
W_{1,v}(\cdot,s_1),
h_{2,v}(\cdot,s_2)
\right),
\end{multline*}
where $R(\cdot)$ denotes the right regular representation and
\begin{align*}
R(f_{\mathfrak{q}})W_v
:=
\int_{\overline{G}(F_v)} f_{\mathfrak{q}}(g)\,R(g)W_v\,dg.
\end{align*}

By \textsection\ref{sec3.1}, the operator $R(f_{\mathfrak{q}})$ is the projection onto 
$\Ind_{Z(F_v)I_v(1)}^{\mathrm{GL}_2(F_v)}\widetilde{\theta}_t$.  
Hence
\begin{align*}
R(f_{\mathfrak{q}})W_v=0
\end{align*}
unless $\pi_v\simeq \sigma_t^{\zeta}$ or $\pi_v\simeq \sigma_t^{-\zeta}$, both simple supercuspidal representations.

Now let $\pi=\otimes_v'\pi_v$ be a unitary generic automorphic representation of $[G]$.  
If $\pi_v\simeq \sigma_t^{\zeta}$ or $\pi_v\simeq \sigma_t^{-\zeta}$ for some $v\mid\mathfrak{q}$, then $\pi$ must be cuspidal.  
By the definition in \textsection\ref{sec2.3.4}, we therefore obtain
\begin{equation}\label{e7.3}
\mathcal{J}_{\Spec}^{\heartsuit}(\mathbf{s},R(f_{\mathfrak{q}})\boldsymbol{h})
=
\sum_{\pi\in \mathcal{F}_{t}^{\zeta}(\mathfrak{q}^3;\omega)\,\cup\, \mathcal{F}_{t}^{-\zeta}(\mathfrak{q}^3;\omega)}
\Psi(\pi;\mathbf{s},\boldsymbol{h}).
\end{equation}
In particular, \eqref{e7.2} follows.  
The estimate \eqref{e7.1} follows from \eqref{e7.3}, Lemma~\ref{lem3.2}, and \cite[\textsection 5]{Yan25}.
\end{proof}

\section{Upper Bound for \texorpdfstring{$\mathcal{I}_{\mathrm{Spec}}^{\heartsuit}(\mathbf{0},R(f_{\mathfrak{q}})\boldsymbol{h})$}{}}\label{sec8}
In this section we prove the following upper bound for $\mathcal{I}_{\mathrm{Spec}}^{\heartsuit}(\mathbf{0},R(f_{\mathfrak{q}})\boldsymbol{h})$. 
\begin{prop}\label{prop8.1}
We have
\begin{align*}
\mathcal{I}_{\mathrm{Spec}}^{\heartsuit}(\mathbf{0},R(f_{\mathfrak{q}})\boldsymbol{h})\ll_{F,\varepsilon} \mathbf{C}_{\infty}^{\varepsilon}N_F(\mathfrak{q})^{2+\varepsilon}.
\end{align*}	
\end{prop}

In the following subsections, we classify the spectral support of the integral $\mathcal{I}_{\mathrm{Spec}}^{\heartsuit}(\mathbf{0},R(f_{\mathfrak{q}})\boldsymbol{h})$, bound the ramified weights, and recall the necessary moment estimates for $L$-functions. We then complete the proof of Proposition~\ref{prop8.1} in \textsection~\ref{sec8.4}.

\subsection{Spectral Support of \texorpdfstring{$\mathcal{I}_{\mathrm{Spec}}^{\heartsuit}(\mathbf{0},R(f_{\mathfrak{q}})\boldsymbol{h})$}{}}\label{sec8.1}
Let $\sigma=\otimes_v'\sigma_v$ be a unitary generic automorphic representation of $[G]$ with trivial central character. Let $W=\otimes_v'W_v$ be a Whittaker function of $\sigma$. Then 
\begin{equation}\label{8.1}
\widetilde{\Psi}(\overline{W},W_1,\overline{h_3})\widetilde{\Psi}(W^*,W_2,\overline{h_4})=\mathcal{L}(\sigma,\omega)\prod_{v\mid\infty}\mathcal{P}_v(W_v)\prod_{v<\infty}\mathcal{P}_v^{\sharp}(W_v),
\end{equation}
where $\mathcal{L}(\sigma,\omega):=L(1/2,\sigma)^3L(1/2,\sigma\times\overline{\omega})$, and for each $v\leq \infty$, 
\begin{align*}
\mathcal{P}_v(W_v):=\widetilde{\Psi}_v(\overline{W}_v,W_{1,v},\overline{h_{3,v}})\widetilde{\Psi}_v(W_v^*,W_{2,v}),\overline{h_{4,v}}),\ \ \mathcal{P}_v^{\sharp}(W_v):=\frac{L_v(1/2,\sigma_v)^{-3}\mathcal{P}_v}{L_v(1/2,\sigma_v\times\overline{\omega}_v)}.
\end{align*}

\begin{lemma}
Let $v\mid\mathfrak{q}$. Then $\mathcal{P}_v=0$ unless $W_v$ is right-$I_v(1)$-invariant. Moreover, if $W_v$ is right-$I_v(1)$-invariant, then 
\begin{equation}\label{e8.5}
\mathcal{P}_v(W_v)
=q_v^{-1}\mathcal{I}_{v}(W_v;I_2)\mathcal{J}_{v}(W_v),
\end{equation}
where $\mathcal{I}_{v}(W_v;k')$ and $\mathcal{J}_{v}(W_v)$ are defined in \eqref{e4.6} and \eqref{e5.43}, respectively. 
\end{lemma}
\begin{proof}
Let $v\mid\mathfrak{q}$. By \cite[Corollary~5.2]{Hsi21}, we have
\begin{equation}
\mathcal{I}_{v}(W_v;k')
=\int_{N(F_v)\backslash G(F_v)}
W_v(x)\,W_{2,v}(xk')\,
\overline{\Phi_{2,v}(\mathbf{e}_2x)}\,
|\det x|_v^{\frac{1}{2}}\,dx,
\end{equation}  
It follows that
\begin{equation}\label{e8.2}
\mathcal{P}_v(W_v)
=\int_{I_v(1)}
\overline{f_{\mathfrak{q}}(k'^{-1})}\,
\mathcal{I}_{v}^{\dag}(W_v;k')\,
\mathcal{I}_{v}(W_v;k')\,dk',
\end{equation}
where $\mathcal{I}_{v}^{\dag}(W_v;k')$ is defined in \eqref{eq5.46}. 

By \eqref{f5.9}, we have $\mathcal{P}_v(W_v)=0$ unless  $W_v$ is right-$I_v(1)$-invariant, which implies that $\sigma_v^{I_v(1)}\neq 0$. 

Parametrize $k'\in I_v(1)$ by \eqref{eq5.1}. From the calculations in \textsection\ref{sec5.2} we obtain 
\begin{equation}\label{eq8.2}
\mathcal{I}_{v}(W_v;k')=\psi_v(
\varpi_v^{-1}t\gamma)	\mathcal{I}_{v}(W_v;I_2).
\end{equation}

Substituting \eqref{2.3} and \eqref{eq8.2} into \eqref{e8.2} yields 
\begin{equation}\label{e8.4}
\mathcal{P}_v(W_v)
=\Vol(I_v(1))^{-1}\mathcal{I}_{v}(W_v;I_2)\,\int_{I_v(1)}
\overline{\psi}_v(\varpi_v^{-1}\beta)\,
\mathcal{I}_{v}^{\dag}(W_v;k')\,dk'.
\end{equation}

Substituting \eqref{f5.47} into  \eqref{e8.4} yields \eqref{e8.5}.
\end{proof}

Note that $\mathcal{P}_v$ also depends on $\omega_v$.  
To make this dependence explicit, we set
\begin{align*}
\mathcal{P}_{\mathfrak{q},\omega}(W_v)
:=
\frac{q_vL_v(1/2,\sigma_v)^{-3}\mathcal{P}_v(W_v)}
{L_v(1/2,\sigma_v\times\overline{\omega}_v)}
=
\frac{\mathcal{I}_{v}(W_v;I_2)\mathcal{J}_{v}(W_v)}
{L_v(1/2,\sigma_v)^{3}\,L_v(1/2,\sigma_v\times\overline{\omega}_v)},
\end{align*}
emphasizing its dependence on $\mathfrak{q}$, $\omega$, and $\sigma$. Define
\begin{equation}\label{e8.6}
\mathcal{Q}_{\mathfrak{q},\omega}(\sigma)
:=
\sum_{W_v\in \mathcal{W}(\sigma_v)^{I_v(1)}}
|\mathcal{P}_{\mathfrak{q},\omega}(W_v)|,
\end{equation}
where $v$ is the place corresponding to the prime $\mathfrak{q}$, and 
$\mathcal{W}(\sigma_v)^{I_v(1)}$ denotes an orthonormal basis of the right-$I_v(1)$-invariant subspace of the Whittaker model of $\sigma_v$. We regard $\mathcal{Q}_{\mathfrak{q},\omega}(\sigma)$ as the ramification weight.

Suppose $\sigma_v^{I_v(1)}\neq 0$.
Such representations $\sigma_v$ were classified in \textsection\ref{sec4}.
Consequently, the corresponding automorphic representations $\sigma$
fall into the classes described in Definition \ref{def1.3} of
\textsection\ref{sec1.2}. 
\begin{lemma}
We have 
\begin{equation}\label{f8.7}
\mathcal{I}_{\mathrm{Spec}}^{\heartsuit}(\mathbf{0},R(f_{\mathfrak{q}})\boldsymbol{h})\ll \mathbf{C}_{\infty}^{-\frac{1}{2}+\varepsilon}N_F(\mathfrak{q})^{-1+\varepsilon}\sum_{*\in\{\RNum{1},\RNum{2},\RNum{3},\RNum{4},\RNum{5}\}}S_{*},
\end{equation}
where
\begin{align*}
S_*:= \int_{\substack{\sigma\in \mathcal{F}_*(\mathfrak{q}^2;\mathbf{1})\\
C_v(\pi)\leq \mathbf{C}_v^{\varepsilon}N_F(\mathfrak{q})^{\varepsilon},\ v\mid\infty}}\mathcal{Q}_{\mathfrak{q},\omega}(\sigma)
L(1/2,\sigma)^3|L(1/2,\sigma\times\overline{\omega})|d\mu_{\sigma}.
\end{align*}
\end{lemma}
\begin{proof}
This follows from \eqref{8.1}, \eqref{e8.5}, and \eqref{e8.6}, together with the manipulations in \cite[\textsection 7]{Yan25}.
\end{proof}

\subsection{Estimate of the Ramification Weight}
\begin{prop}\label{prop8.5}
We have the following. 
\begin{itemize}
\item Suppose $\sigma\in \mathcal{F}_{\RNum{1}}(\mathfrak{q}^2;\mathbf{1})\bigsqcup \mathcal{F}_{\RNum{2}}(\mathfrak{q}^2;\mathbf{1})$. Then 
\begin{equation}\label{8.7}
\mathcal{Q}_{\mathfrak{q},\omega}(\sigma)\ll_F N_F(\mathfrak{q})^{2+\frac{7}{32}}\mathbf{1}_{\sigma\in \mathcal{F}_{\RNum{1}}(\mathfrak{q}^2;\mathbf{1})}+N_F(\mathfrak{q})\mathbf{1}_{\sigma\in \mathcal{F}_{\RNum{2}}(\mathfrak{q}^2;\mathbf{1})}.
\end{equation}

\item Suppose $\sigma\in \mathcal{F}_{\RNum{3}}(\mathfrak{q}^2;\mathbf{1})$. Then 
\begin{equation}\label{8.8}
\mathcal{Q}_{\mathfrak{q},\omega}(\sigma)\ll_F N_F(\mathfrak{q})^{1+\frac{7}{32}}\mathbf{1}_{r_{\overline{\omega}_{\mathfrak{q}}\xi_{\mathfrak{q}}}=0}+N_F(\mathfrak{q})^{\frac{21}{64}}\mathbf{1}_{r_{\overline{\omega}_{\mathfrak{q}}\xi_{\mathfrak{q}}}=1}. 
\end{equation}

\item Suppose $\sigma\in \mathcal{F}_{\RNum{4}}(\mathfrak{q}^2;\mathbf{1})\bigsqcup \mathcal{F}_{\RNum{5}}(\mathfrak{q}^2;\mathbf{1})$. Then 
\begin{equation}\label{8.9}
\mathcal{Q}_{\mathfrak{q},\omega}(\sigma)\ll_F N_F(\mathfrak{q})\mathbf{1}_{\sigma\in \mathcal{F}_{\RNum{4}}(\mathfrak{q}^2;\mathbf{1})}+N_F(\mathfrak{q})^{\frac{3}{2}}\mathbf{1}_{\sigma\in \mathcal{F}_{\RNum{5 }}(\mathfrak{q}^2;\mathbf{1})}.
\end{equation}
\end{itemize} 
\end{prop}
\begin{proof}
Suppose $\sigma\in \mathcal{F}_{\RNum{1}}(\mathfrak{q}^2;\mathbf{1})\bigsqcup \mathcal{F}_{\RNum{2}}(\mathfrak{q}^2;\mathbf{1})\bigsqcup \mathcal{F}_{\RNum{3}}(\mathfrak{q}^2;\mathbf{1})$. We may take 
\begin{equation}\label{e8.7}
\mathcal{W}(\sigma_v)^{I_v(1)}=\Big\{W_v^{(f_1)}, W_v^{(f_2)}\Big\},
\end{equation}
which was defined by \eqref{f4.2}; and when $\sigma\in \mathcal{F}_{\RNum{4}}(\mathfrak{q}^2;\mathbf{1})\bigsqcup \mathcal{F}_{\RNum{5}}(\mathfrak{q}^2;\mathbf{1})$, we may take
\begin{equation}\label{e8.8}
\mathcal{W}(\sigma_v)^{I_v(1)}=\{W_v^{\chi_v}\},\ \ \ \sigma_v=\mathrm{St}\otimes\chi_v,	
\end{equation}
defined as in \eqref{4.11}. 

Consider the following scenarios. 
\begin{itemize}
\item Suppose $\sigma\in \mathcal{F}_{\RNum{1}}(\mathfrak{q}^2;\mathbf{1})$. Write $\sigma_v=\xi_v\boxplus\xi_v^{-1}$ with $\xi_v$ unramified and $|\xi_v(\cdot)|^2=\xi_v(\cdot)\overline{\xi}_v(\cdot)=|\cdot|_v^{2\nu}$, $0<\nu\leq 7/64$. 
\begin{itemize}
\item By \eqref{e4.1.},  \eqref{f4.12},  \eqref{5.21},  \eqref{5.25}, \eqref{f5.45},  \eqref{5.44}, and the bounds for the Gauss sums, we have
\begin{equation}\label{e8.9}
|\mathcal{P}_{\mathfrak{q},\omega}(W_v^{(f_1)})|\ll N_F(\mathfrak{q})^{\frac{1}{2}+2\nu}\cdot |\mathcal{P}_{\mathfrak{q},\omega}(W_v^{(1)})|\ll_FN_F(\mathfrak{q})^{2+2\nu}. 
\end{equation}

\item By \eqref{e4.1.},  \eqref{f4.12},  \eqref{c5.28},  \eqref{c5.29}, \eqref{f5.45},  \eqref{5.55}, and the bounds for the Gauss sums, we have
\begin{equation}\label{e8.10}
|\mathcal{P}_{\mathfrak{q},\omega}(W_v^{(f_2)})|\ll  |\mathcal{P}_{\mathfrak{q},\omega}(W_v^{(2)})|\ll_F N_F(\mathfrak{q})^{\frac{3}{2}+2\nu}. 
\end{equation}
\end{itemize}

\item Suppose $\sigma\in \mathcal{F}_{\RNum{2}}(\mathfrak{q}^2;\mathbf{1})$. Write $\sigma_v=\xi_v\boxplus\xi_v^{-1}$ with $\xi_v$ unitary and $r_{\xi_v}=1$.  
\begin{itemize}
\item By \eqref{e4.1.},  \eqref{f4.12}, \eqref{5.13}, \eqref{5.19}, \eqref{f5.45},  \eqref{f5.37}, and the bounds for the Gauss sums, we have
\begin{equation}\label{e8.11}
|\mathcal{P}_{\mathfrak{q},\omega}(W_v^{(f_1)})|\ll N_F(\mathfrak{q})^{\frac{1}{2}}\cdot |\mathcal{P}_{\mathfrak{q},\omega}(W_v^{(1)})|\ll_F N_F(\mathfrak{q}). 	
\end{equation}

\item By \eqref{e4.1.},  \eqref{f4.12}, \eqref{c5.26}, \eqref{c5.27}, \eqref{f5.45},  \eqref{5.54}, and the bounds for the Gauss sums and Jacobi sums, we have
\begin{equation}\label{e8.12}
|\mathcal{P}_{\mathfrak{q},\omega}(W_v^{(f_2)})|\ll |\mathcal{P}_{\mathfrak{q},\omega}(W_v^{(2)})|\ll_F N_F(\mathfrak{q}). 	
\end{equation}
\end{itemize}

\item Suppose $\sigma\in \mathcal{F}_{\RNum{3}}(\mathfrak{q}^2;\mathbf{1})$. Write $\sigma_v=\xi_v\boxplus\xi_v$ with $\xi_v^2=\mathbf{1}$,  $r_{\xi_v}=1$ and $|\xi_v(\cdot)|^2=|\cdot|_v^{2\nu}$. 
\begin{itemize}
\item By \eqref{e4.1.},  \eqref{f4.12}, \eqref{5.13}, \eqref{5.19}, \eqref{f5.45},  \eqref{f5.37}, and the bounds for the Gauss sums, we have
\begin{equation}\label{e8.13}
|\mathcal{P}_{\mathfrak{q},\omega}(W_v^{(f_1)})|\ll_F N_F(\mathfrak{q})^{1+2\nu}\mathbf{1}_{r_{\overline{\omega}_v\xi_v}=0}+N_F(\mathfrak{q})^{3\nu}\mathbf{1}_{r_{\overline{\omega}_v\xi_v}=1}. 	
\end{equation}

\item By \eqref{e4.1.},  \eqref{f4.12}, \eqref{c5.26}, \eqref{c5.27}, \eqref{f5.45},  \eqref{5.54}, and the bounds for the Gauss sums and Jacobi sums, we have
\begin{equation}\label{e8.14}
|\mathcal{P}_{\mathfrak{q},\omega}(W_v^{(f_2)})|\ll |\mathcal{P}_{\mathfrak{q},\omega}(W_v^{(2)})|\ll_F N_F(\mathfrak{q}). 	
\end{equation}
\end{itemize}

\item Suppose $\sigma\in \mathcal{F}_{\RNum{4}}(\mathfrak{q}^2;\mathbf{1})$. Write $\sigma_v\simeq \mathrm{St}\otimes\chi_v$ with $\chi_v$ unramified quadratic. 

By \eqref{e4.1.},  \eqref{f4.12}, \eqref{eq5.41}, \eqref{eq5.42}, \eqref{5.74},  \eqref{e5.77}, \eqref{5.104}, and the bounds for the Gauss sums, we have
\begin{equation}\label{e8.15}
|\mathcal{P}_{\mathfrak{q},\omega}(W_v^{\chi_v})\ll_F N_F(\mathfrak{q}).  
\end{equation}

\item Suppose $\sigma\in \mathcal{F}_{\RNum{5}}(\mathfrak{q}^2;\mathbf{1})$. Write $\sigma_v\simeq \mathrm{St}\otimes\chi_v$ with $\chi_v$ quadratic of conductor exponent $1$. 

By \eqref{e4.1.},  \eqref{f4.12}, \eqref{e5.37}, \eqref{eq5.40}, \eqref{5.74},  \eqref{e5.76}, \eqref{5.103}, and the bounds for the Gauss sums and Jacobi sums, we have
\begin{equation}\label{e8.16}
|\mathcal{P}_{\mathfrak{q},\omega}(W_v^{\chi_v})\ll_F N_F(\mathfrak{q})^{\frac{3}{2}}\mathbf{1}_{r_{\overline{\omega}_v\chi_v}=1}+N_F(\mathfrak{q})\mathbf{1}_{r_{\overline{\omega}_v\chi_v}=0}.  
\end{equation}
\end{itemize}

Therefore, \eqref{8.7} follows from \eqref{e8.7}, \eqref{e8.9}, \eqref{e8.10}, \eqref{e8.11} and \eqref{e8.12}; and \eqref{8.8} follows from \eqref{e8.7}, \eqref{e8.13} and \eqref{e8.14}; and \eqref{8.9} follows from \eqref{e8.15} and \eqref{e8.16}.
\end{proof}

\subsection{Moments Estimates of $L$-functions} 

\subsubsection{The Fourth Moment of $L$-functions}
As a special case of 
\cite[Theorem D \& E]{Yan25}, we have the following. 
\begin{thm}\label{thm8.7}
Let $*\in\{\RNum{1},\RNum{2},\RNum{3},\RNum{4},\RNum{5}\}$. For $v\mid\infty$, let $T_v>1$ be a constant, and $T_{\infty} := \prod_{v \mid \infty} T_v$. Then 
\begin{align*}
\int_{\substack{\sigma\in \mathcal{F}_*(\mathfrak{q}^2;\mathbf{1})\\
C_v(\sigma)\leq T_v,\ v\mid\infty}}|L(1/2,\sigma)|^4d\mu_{\sigma}\ll_{F,\varepsilon} T_{\infty}^{1+\varepsilon}N_F(\mathfrak{q})^{2+\varepsilon}.
\end{align*}
\end{thm}

\subsubsection{The Cubic Moment of $L$-functions}
\begin{thm}\label{thme8.7}
For $v\mid\infty$, let $T_v>1$ be a constant, and $T_{\infty} := \prod_{v \mid \infty} T_v$. There exists an absolute constant $A>0$ such that 
\begin{equation}\label{8.22}
\sum_{\substack{\sigma\in \mathcal{F}_{\RNum{5 }}(\mathfrak{q}^2;\mathbf{1})\\
C_v(\pi)\leq T_v,\ v\mid\infty}}L(1/2,\sigma)^3\ll_{F,\varepsilon} T_{\infty}^AN_F(\mathfrak{q})^{1+\varepsilon}. 
\end{equation}
\end{thm}
\begin{proof}
This is a variant of \cite[Theorem 1.1]{Nel19}. We briefly explain the comparison of the local data.

In \cite{Nel19}, one studies twists of the form $\sigma'\otimes\chi$, where $\chi$ is a quadratic Hecke character. At a place $v\mid\mathfrak{q}$ with $r_{\chi_v}=1$, the spectral side involves local representations of the form $\sigma_v'\otimes\chi_v$, where $\sigma_v'$ has conductor exponent $\le 1$. In particular, if the conductor exponent equals $1$, then $\sigma_v'$ must be $\mathrm{St}$ (up to an unramified quadratic twist).

In our setting, the local representation at $v\mid\mathfrak{q}$ is already of the form $\mathrm{St}\otimes\chi_v$, with $\chi_v$ quadratic of conductor exponent $1$. Thus the local situation coincides with the ramified case considered in loc.\ cit.

By making a minor modification of the test function away from the ramified places, the spectral side produces the family $\mathcal{F}_{\RNum{5}}(\mathfrak{q}^2;\mathbf{1})$ as the main contribution. The remaining contributions (in particular from everywhere unramified representations) are negligible and may be absorbed into the error term. With this adjustment, the local calculations proceed exactly as in \cite{Nel19}, and the same argument yields \eqref{8.22}.

We therefore omit the routine details and refer to the method of \cite[Theorem 1.1]{Nel19}, which itself generalizes \cite[Theorem 1.1]{CI00}.
\end{proof}

\subsection{Proof of Proposition \ref{prop8.1}}\label{sec8.4}
By \eqref{8.7} we have 
\begin{align*}
S_{\RNum{1}}\ll N_F(\mathfrak{q})^{2+\frac{7}{32}}\int_{\substack{\sigma\in \mathcal{F}_{\RNum{1}}(\mathfrak{q}^2;\mathbf{1})\\
C_v(\pi)\leq \mathbf{C}_v^{\varepsilon}N_F(\mathfrak{q})^{\varepsilon},\ v\mid\infty}}
L(1/2,\sigma)^3|L(1/2,\sigma\times\overline{\omega})|d\mu_{\sigma}.
\end{align*}

Recall the convexity bound 
\begin{align*}
L(1/2,\sigma\times\overline{\omega})\ll \mathbf{C}_{\infty}^{\varepsilon} C_{\infty}(\omega)^{\frac{1}{2}+\varepsilon}N_F(\mathfrak{q})^{\frac{1}{2}+\varepsilon}\ll \mathbf{C}_{\infty}^{\frac{1}{2}+2\varepsilon}N_F(\mathfrak{q})^{\frac{1}{2}+\varepsilon}.
\end{align*}

Therefore, it follows from H\"{o}lder's inequality and Theorem \ref{thm8.7} that 
\begin{multline}\label{8.24}
S_{\RNum{1}}\ll N_F(\mathfrak{q})^{2+2\nu}\cdot \mathbf{C}_{\infty}^{\frac{1}{2}+2\varepsilon}N_F(\mathfrak{q})^{\frac{1}{2}+\varepsilon}\cdot 
\bigg[\int_{\substack{\sigma\in \mathcal{F}_{\RNum{1}}(\mathfrak{q}^2;\mathbf{1})\\
C_v(\pi)\leq \mathbf{C}_v^{\varepsilon}N_F(\mathfrak{q})^{\varepsilon},\ v\mid\infty}}
1d\mu_{\sigma}\bigg]^{\frac{1}{4}}\\
\bigg[\int_{\substack{\sigma\in \mathcal{F}_{\RNum{1}}(\mathfrak{q}^2;\mathbf{1})\\
C_v(\pi)\leq \mathbf{C}_v^{\varepsilon}N_F(\mathfrak{q})^{\varepsilon},\ v\mid\infty}}
L(1/2,\sigma)^4d\mu_{\sigma}\bigg]^{\frac{3}{4}}\ll N_F(\mathfrak{q})^{2+\frac{7}{32}}\mathbf{C}_{\infty}^{\frac{1}{2}+20\varepsilon}N_F(\mathfrak{q})^{\frac{1}{2}+10\varepsilon}.
\end{multline}
Here we have used the fact that $\mathcal{F}_{\RNum{1}}(\mathfrak{q}^2;\mathbf{1})\ll \mathbf{C}_{\infty}^{5\varepsilon}N_F(\mathfrak{q})^{5\varepsilon}$. 

Utilizing the estimate \eqref{8.7} we obtain 
\begin{align*}
S_{\RNum{2}}\ll N_F(\mathfrak{q})\int_{\substack{\sigma\in \mathcal{F}_{\RNum{2}}(\mathfrak{q}^2;\mathbf{1})\\
C_v(\pi)\leq \mathbf{C}_v^{\varepsilon}N_F(\mathfrak{q})^{\varepsilon},\ v\mid\infty}}
L(1/2,\sigma)^3|L(1/2,\sigma\times\overline{\omega})|d\mu_{\sigma}.
\end{align*}

Notice that for $\sigma\in \mathcal{F}_{\RNum{2}}(\mathfrak{q}^2;\mathbf{1})$, the representation $\sigma\otimes\overline{\omega}$ is of arithmetic conductor $\mathfrak{q}^2$, unramified at all other finite places, and $C_v(\sigma\otimes\overline{\omega})\ll C_v(\sigma)C_v(\omega)^2$ at all $v\mid\infty$. Therefore, by H\"{o}lder's inequality and Theorem \ref{thm8.7} we obtain 
\begin{multline}\label{8.25}
S_{\RNum{2}}\ll N_F(\mathfrak{q})\cdot \bigg[\int_{\substack{\sigma\in \mathcal{F}_{\RNum{2}}(\mathfrak{q}^2;\mathbf{1})\\
C_v(\pi)\leq \mathbf{C}_v^{\varepsilon}N_F(\mathfrak{q})^{\varepsilon},\ v\mid\infty}}
|L(1/2,\sigma\times\overline{\omega})|^4d\mu_{\sigma}\bigg]^{\frac{1}{4}}\\
\bigg[\int_{\substack{\sigma\in \mathcal{F}_{\RNum{2}}(\mathfrak{q}^2;\mathbf{1})\\
C_v(\pi)\leq \mathbf{C}_v^{\varepsilon}N_F(\mathfrak{q})^{\varepsilon},\ v\mid\infty}}
L(1/2,\sigma)^4d\mu_{\sigma}\bigg]^{\frac{3}{4}}\ll \mathbf{C}_{\infty}^{\frac{1}{2}+20\varepsilon}N_F(\mathfrak{q})^{3+10\varepsilon}.
\end{multline}

Let $\sigma\in \mathcal{F}_{\RNum{3}}(\mathfrak{q}^2;\mathbf{1})$. At $v\mid\mathfrak{q}$, $\sigma_v=\chi_v|\cdot|_v^{\nu}\boxplus \chi_v|\cdot|_v^{-\nu}$, where $\chi_v$ is quadratic with conductor exponent $1$, and $0<\nu\leq 7/64$. 

Utilizing the estimate \eqref{8.8} we obtain
\begin{equation}\label{8.26}
S_{\RNum{3}}\ll N_F(\mathfrak{q})^{1+\frac{7}{32}}\cdot S_{\RNum{3}}^{(1)}+N_F(\mathfrak{q})^{\frac{21}{64}}\cdot S_{\RNum{3}}^{(2)},
\end{equation}
where 
\begin{align*}
&S_{\RNum{3}}^{(1)}:=\int_{\substack{\sigma\in \mathcal{F}_{\RNum{3}}(\mathfrak{q}^2;\mathbf{1})\\
C_v(\pi)\leq \mathbf{C}_v^{\varepsilon}N_F(\mathfrak{q})^{\varepsilon},\ v\mid\infty}}
L(1/2,\sigma)^3|L(1/2,\sigma\times\overline{\omega})|\mathbf{1}_{r_{\overline{\omega}_{\mathfrak{q}}\xi_{\mathfrak{q}}}=0}d\mu_{\sigma}\\
&S_{\RNum{3}}^{(2)}:=\int_{\substack{\sigma\in \mathcal{F}_{\RNum{3}}(\mathfrak{q}^2;\mathbf{1})\\
C_v(\pi)\leq \mathbf{C}_v^{\varepsilon}N_F(\mathfrak{q})^{\varepsilon},\ v\mid\infty}}
L(1/2,\sigma)^3|L(1/2,\sigma\times\overline{\omega})|\mathbf{1}_{r_{\overline{\omega}_{\mathfrak{q}}\xi_{\mathfrak{q}}}=1}d\mu_{\sigma}.
\end{align*}

Under the assumption that $r_{\overline{\omega}_{\mathfrak{q}}\xi_{\mathfrak{q}}}=0$ the representation $\sigma\times\overline{\omega}$ is unramified at all finite places. Hence, the convexity bound yields
\begin{equation}\label{8.27}
|L(1/2,\sigma\times\overline{\omega})|\mathbf{1}_{r_{\overline{\omega}_{\mathfrak{q}}\xi_{\mathfrak{q}}}=0}\ll \mathbf{C}_{\infty}^{\frac{1}{2}+2\varepsilon}N_F(\mathfrak{q})^{\varepsilon}.
\end{equation}
Together with H\"{o}lder's inequality and Theorem \ref{thm8.7} we have 
\begin{equation}\label{8.28}
\int_{\substack{\sigma\in \mathcal{F}_{\RNum{3}}(\mathfrak{q}^2;\mathbf{1})\\
C_v(\pi)\leq \mathbf{C}_v^{\varepsilon}N_F(\mathfrak{q})^{\varepsilon},\ v\mid\infty}}
L(1/2,\sigma)^3d\mu_{\sigma}\ll \mathbf{C}_{\infty}^{10\varepsilon}N_F(\mathfrak{q})^{1+5\varepsilon}.
\end{equation}

By convexity bound, 
\begin{equation}\label{8.29}
|L(1/2,\sigma\times\overline{\omega})|\mathbf{1}_{r_{\overline{\omega}_{\mathfrak{q}}\xi_{\mathfrak{q}}}=1}\ll \mathbf{C}_{\infty}^{\frac{1}{2}+2\varepsilon}N_F(\mathfrak{q})^{\frac{1}{2}+\varepsilon}.
\end{equation} 

Therefore, substituting \eqref{8.27}, \eqref{8.28} and \eqref{8.29} into \eqref{8.26} yields 
\begin{equation}\label{8.30}
S_{\RNum{3}}\ll  \mathbf{C}_{\infty}^{\frac{1}{2}+20\varepsilon}N_F(\mathfrak{q})^{2+\frac{7}{32}+10\varepsilon}.
\end{equation}

By \eqref{8.9} we obtain 
\begin{align*}
S_{\RNum{4}}\ll N_F(\mathfrak{q})\int_{\substack{\sigma\in \mathcal{F}_{\RNum{4}}(\mathfrak{q}^2;\mathbf{1})\\
C_v(\pi)\leq \mathbf{C}_v^{\varepsilon}N_F(\mathfrak{q})^{\varepsilon},\ v\mid\infty}}
L(1/2,\sigma)^3|L(1/2,\sigma\times\overline{\omega})|d\mu_{\sigma}.
\end{align*}

By the convexity for $L(1/2,\sigma\times\overline{\omega})$, H\"{o}lder's inequality, and Theorem \ref{thm8.7}, 
\begin{equation}\label{8.31}
S_{\RNum{4}}\ll \mathbf{C}_{\infty}^{\frac{1}{2}+20\varepsilon}N_F(\mathfrak{q})^{\frac{3}{2}+10\varepsilon}.
\end{equation}

By \eqref{8.9} we obtain 
\begin{align*}
S_{\RNum{5}}\ll N_F(\mathfrak{q})^{\frac{3}{2}}\int_{\substack{\sigma\in \mathcal{F}_{\RNum{5}}(\mathfrak{q}^2;\mathbf{1})\\
C_v(\pi)\leq \mathbf{C}_v^{\varepsilon}N_F(\mathfrak{q})^{\varepsilon},\ v\mid\infty}}
L(1/2,\sigma)^3|L(1/2,\sigma\times\overline{\omega})|d\mu_{\sigma}.
\end{align*}

It thus follows from the subconvexity for $L(1/2,\sigma\times\overline{\omega})$ (see \cite[Corollary 1.4]{Yan25}) and Theorem \ref{thme8.7} that 
\begin{equation}\label{8.32}
S_{\RNum{5}}\ll \mathbf{C}_{\infty}^{\frac{1}{2}-\frac{1}{60}+20\varepsilon}N_F(\mathfrak{q})^{3-\frac{1}{60}+10\varepsilon}.
\end{equation}

Therefore, Proposition \ref{prop8.1} follows from \eqref{f8.7}, \eqref{8.24}, \eqref{8.25}, \eqref{8.30}, \eqref{8.31} and \eqref{8.32}.

\begin{remark}
One could improve the dependence on $\mathbf{C}_{\infty}$ by invoking the Burgess bound for $L(1/2,\sigma\times\overline{\omega})$ established in \cite{Yan23c}. However, this would not yield a strong subconvexity saving in the archimedean aspect. We therefore content ourselves with the convexity bound.
\end{remark}

\section{Estimates of  Degenerate Terms: the Geometric Side}\label{sec9}

In this section we establish upper bounds for the geometric degenerate terms: 
\begin{prop}\label{prop9.1}
Let $1\leq i\leq 8$. We have
\begin{equation}\label{9.1}
\widetilde{\Psi}_{\mathrm{Geo}}^{(i)}(\mathbf{0}\mid R(f_{\mathfrak{q}})\boldsymbol{h})\ll_{F,\varepsilon} \mathbf{C}_{\infty}^{\varepsilon}N_F(\mathfrak{q})^{2+\varepsilon}.
\end{equation}	
\end{prop}

\subsection{Upper Bounds for \texorpdfstring{$\widetilde{\Psi}_{\mathrm{Geo}}^{(i)}(\mathbf{s},R(f_{\mathfrak{q}})\boldsymbol{h}): i=1,2$}{}}

Notice that $h_2(\cdot,s_2)\overline{h_4(\cdot,\overline{s_4})}$ is a section in $|\cdot|^{1/2+s_2+s_4}\boxplus |\cdot|^{-1/2-s_2-s_4}$. Hence 
\begin{align*}
\widetilde{\Psi}_{\mathrm{Geo}}^{(1)}(\mathbf{s},R(f_{\mathfrak{q}})\boldsymbol{h})\propto \mathcal{L}_1(\mathbf{s}):=\frac{\Lambda(1+s_2+s_4,(|\cdot|^{s_1}\boxplus |\cdot|^{-s_1})\times(|\cdot|^{s_3}\boxplus |\cdot|^{-s_3}))}{\Lambda(2+2s_2+2s_4,\mathbf{1})\Lambda(1+2s_2,\overline{\omega})^{-1}\Lambda(1+2s_4,\omega)^{-1}}.
\end{align*}

Consequently, we obtain 
\begin{equation}
\widetilde{\Psi}_{\mathrm{Geo}}^{(1)}(\mathbf{s},R(f_{\mathfrak{q}})\boldsymbol{h})=\mathcal{L}_1(\mathbf{s})\prod_{v\leq\infty}\mathcal{P}_v^{(1)}(\mathbf{s}),
\end{equation}
where 
\begin{align*}
\mathcal{P}_v^{(1)}(\mathbf{s})=\mathcal{L}_{1,v}(\mathbf{s})^{-1}\widetilde{\Psi}_v(W_1(\cdot,s_1),\overline{W_3(\cdot,\overline{s_3})},h_2(\cdot,s_2)\overline{h_4(\cdot,\overline{s_4})})
\end{align*}
if $v\nmid\mathfrak{q}$, and at $v\mid\mathfrak{q}$, the factor $\mathcal{P}_v^{(1)}(\mathbf{s})$ is defined by 
\begin{align*}
\mathcal{L}_{1,v}(\mathbf{s})^{-1}\int_{\overline{G}(F_v)}
\overline{f_{\mathfrak{q}}(g)}
\widetilde{\Psi}_v(R(g^{-1})W_1(\cdot,s_1),\overline{W_3(\cdot,\overline{s_3})},R(g^{-1})h_2(\cdot,s_2)\overline{h_4(\cdot,\overline{s_4})})dg.
\end{align*}

\begin{lemma}\label{lem9.2}
Let $v\mid\mathfrak{q}$. Then 
\begin{equation}\label{eq9.2}
\mathcal{P}_v^{(1)}(\mathbf{s})=\Vol(I_v^*(1))^{-1}q_v^{s_2+s_4}\mathcal{L}_{1,v}(\mathbf{s})^{-1}W_{1,v}^{\circ}\left(I_2
,s_1\right)
\overline{W_{3,v}^{\circ}\left(I_2,\overline{s_3}\right)}.
\end{equation}
\end{lemma}
\begin{proof}
By definition \eqref{2.9} and a change of variable, we obtain 
\begin{multline*}
\mathcal{P}_v^{(1)}(\mathbf{s})=\mathcal{L}_{1,v}(\mathbf{s})^{-1}\int_{I_v(1)}
\overline{f_{\mathfrak{q}}(g)}
\int_{N(F_v)\backslash\overline{G}(F_v)}
W_{1,v}(x,s_1)\overline{W_{3,v}(xg,\overline{s_3})}\\
h_2(x,s_2)\overline{h_4(xg,\overline{s_4})})dxdg.
\end{multline*}

Write $g=k'$, which is parametrized by \eqref{eq5.1}. Utilizing Iwasawa decomposition, together with Lemma \ref{lem4.2}, we rewrite the above integral as 
\begin{multline}\label{e9.2}
\mathcal{P}_v^{(1)}(\mathbf{s})=\Vol(I_v(1))^{-1}\mathcal{L}_{1,v}(\mathbf{s})^{-1}h_2(I_2,s_2)\overline{h_4(I_2,\overline{s_4})})\int_{I_v(1)}
\psi_v(\varpi_v^{-1}\beta)\\
\sum_{i\in \mathbb{Z}}q_v^{-(s_2+s_4)i}\int_{K_v[1]}W_{1,v}\left(\begin{pmatrix}
\varpi_v^i\\
& 1
\end{pmatrix}k
,s_1\right)
\overline{W_{3,v}\left(\begin{pmatrix}
\varpi_v^i\\
& 1
\end{pmatrix}k\begin{pmatrix}
1 & \beta\\
& 1
\end{pmatrix}
,\overline{s_3}\right)}dkdk'.
\end{multline}

Let $W_{1,v}^{\circ}=W_{3,v}^{\circ}$ be the spherical vector. Then \eqref{e9.2} boils down to 
\begin{equation}\label{eq9.3}
\mathcal{P}_v^{(1)}(\mathbf{s})=\frac{\Vol(I_v^*(1))}{\Vol(I_v(1))\mathcal{L}_{1,v}(\mathbf{s})}\cdot h_2(I_2,s_2)\overline{h_4(I_2,\overline{s_4})})
\sum_{i\in \mathbb{Z}}q_v^{-(s_2+s_4)i}I_i',
\end{equation}
where 
\begin{multline*}
I_i':=\sum_{\alpha\in \mathbb{F}_v^{\times}}\int_{I_v^*(1)}
\psi_v(\varpi_v^{-1}\beta)W_{1,v}^{\circ}\left(\begin{pmatrix}
\alpha\varpi_v^{i+1}\\
& 1
\end{pmatrix}
,s_1\right)\\
\overline{W_{3,v}^{\circ}\left(\begin{pmatrix}
\alpha\varpi_v^{i}\\
& 1
\end{pmatrix}\begin{pmatrix}
1 & \beta\\
& 1
\end{pmatrix}\begin{pmatrix}
\varpi_v &  \\
& 1
\end{pmatrix},\overline{s_3}\right)}dk'.
\end{multline*}

Making the change of variable $i\mapsto i-1$ in \eqref{eq9.3} we derive 
\begin{equation}\label{eq9.4}
\mathcal{P}_v^{(1)}(\mathbf{s})=\frac{\Vol(I_v^*(1))}{\Vol(I_v(1))\mathcal{L}_{1,v}(\mathbf{s})}\cdot h_2(I_2,s_2)\overline{h_4(I_2,\overline{s_4})})q_v^{s_2+s_4}
\sum_{i\geq 0}q_v^{-(s_2+s_4)i}I_i,
\end{equation}
where $I_i$ is defined by 
\begin{align*}
\sum_{\alpha\in \mathbb{F}_v^{\times}}W_{1,v}^{\circ}\left(\begin{pmatrix}
\alpha\varpi_v^{i}\\
& 1
\end{pmatrix}
,s_1\right)
\overline{W_{3,v}^{\circ}\left(\begin{pmatrix}
\alpha\varpi_v^{i}\\
& 1
\end{pmatrix},\overline{s_3}\right)}\int_{I_v^*(1)}
\psi_v(\varpi_v^{-1}(1-\alpha\varpi_v^{i})\beta)dk'.
\end{align*}

By orthogonality, we have
\begin{equation}\label{eq9.6}
\int_{I_v^*(1)}
\psi_v(\varpi_v^{-1}(1-\alpha\varpi_v^i)\beta)dk'=\Vol(I_v^*(1))\mathbf{1}_{\alpha\in 1+\mathfrak{p}_v}\mathbf{1}_{i=0}.
\end{equation}
Therefore, 
\begin{equation}\label{eq9.5}
I_i=\Vol(I_v(1))\cdot W_{1,v}^{\circ}\left(I_2
,s_1\right)
\overline{W_{3,v}^{\circ}\left(I_2,\overline{s_3}\right)}\mathbf{1}_{i=0}.
\end{equation}

Combining \eqref{eq9.5} with the Casselman--Shalika formula and the identity
\begin{align*}
h_2(I_2,s_2)=\overline{h_4(I_2,\overline{s_4})})=\Vol(I_v^*(1))^{-1},
\end{align*}
yields \eqref{eq9.2}.
\end{proof}

\begin{cor}\label{cor9.3}
Let $1\leq i\leq 2$, and let $\mathbf{s}\in \mathbf{B}_{\varepsilon}^4$. We have
\begin{equation}\label{e9.7}
\widetilde{\Psi}_{\mathrm{Geo}}^{(i)}(\mathbf{s},R(f_{\mathfrak{q}})\boldsymbol{h})\ll_{F,\varepsilon} \mathbf{C}_{\infty}^{\varepsilon}N_F(\mathfrak{q})^{2+\varepsilon}.
\end{equation}
\end{cor}
\begin{proof}
By Lemma \ref{lem9.2} and the local estimates in \cite[\textsection 9.1]{Yan25}, we obtain \eqref{e9.7} for $i=1$. The remaining case $i=2$ holds similarly, noting that $h_2^{\Diamond}(\cdot,s_2)\overline{h_4(\cdot,\overline{s_4})}$ is a section in $|\cdot|^{1/2-s_2+s_4}\boxplus |\cdot|^{-1/2+s_2-s_4}$ with bounded intertwining multipliers (see \cite[Proposition 2.2.2]{Sch02}).  
\end{proof}

\subsection{Vanishing of \texorpdfstring{$\widetilde{\Psi}_{\mathrm{Geo}}^{(i)}(\mathbf{s},R(f_{\mathfrak{q}})\boldsymbol{h}): i=3,4$}{}}
\begin{lemma}\label{lem9.4}
Let $3\leq i\leq 4$, and let $\mathbf{s}\in \mathbf{B}_{\varepsilon}^4$. We have
\begin{equation}\label{e9.8}
\widetilde{\Psi}_{\mathrm{Geo}}^{(i)}(\mathbf{s},R(f_{\mathfrak{q}})\boldsymbol{h})\equiv 0.
\end{equation}
\end{lemma}
\begin{proof}
Consider the local component of $\widetilde{\Psi}_{\mathrm{Geo}}^{(3)}(\mathbf{s},R(f_{\mathfrak{q}})\boldsymbol{h})$ at $v\mid\mathfrak{q}$:
\begin{align*}
J_v^{(3)}:=\int_{N(F_v)\backslash\overline{G}(F_v)}
W_{1,v}^*(x,s_1)W_{2,v}(x,s_2)\int_{I_v(1)}
\overline{f_{\mathfrak{q}}(g)}\overline{h_{3,v}(xg,\overline{s_3})}
\overline{h_{4,v}(xg,\overline{s_4})})dgdx.
\end{align*}

By \eqref{eq5.1}, Lemmas \ref{lem4.1} and \ref{lem4.2}, as well as the Iwasawa decomposition, we obtain 
\begin{multline*}
J_v^{(3)}=\sum_{i\in \mathbb{Z}}\int_{K_v[1]}
W_{1,v}^{\circ,*}\left(\begin{pmatrix}
\varpi_v^i\\
& 1
\end{pmatrix}k\begin{pmatrix}
\varpi_v\\
& 1
\end{pmatrix}
,s_1\right)W_2\left(\begin{pmatrix}
\varpi_v^i\\
& 1
\end{pmatrix}k,s_2\right)\\
q_v^{-(s_3+s_4)i}\int_{I_v(1)}
\psi_v(\varpi_v^{-1}\beta)\overline{h_{3,v}^{\circ}\left(k\begin{pmatrix}
1 & \beta\\
& 1
\end{pmatrix}\begin{pmatrix}
\varpi_v\\
& 1
\end{pmatrix}
,\overline{s_3}\right)}
\overline{h_4\left(I_2,\overline{s_4}\right)})dk'dk.
\end{multline*}

Notice that for $k \in K_v[1]$, we have
\begin{align*}
k\begin{pmatrix}
1 & \beta\\
& 1
\end{pmatrix}\begin{pmatrix}
\varpi_v\\
& 1
\end{pmatrix}\in N(\mathcal{O}_v)\begin{pmatrix}
\varpi_v\\
& 1
\end{pmatrix}K_v.
\end{align*}
Since $h_{3,v}^{\circ}$ is left $N(\mathcal{O}_v)$-invariant, it follows that
\begin{align*}
h_{3,v}^{\circ}\left(k\begin{pmatrix}
1 & \beta\\
& 1
\end{pmatrix}\begin{pmatrix}
\varpi_v\\
& 1
\end{pmatrix}
,\overline{s_3}\right)=h_{3,v}^{\circ}\left(\begin{pmatrix}
\varpi_v\\
& 1
\end{pmatrix}
,\overline{s_3}\right).
\end{align*}

Consequently, by the orthogonality of the additive character, we obtain
\begin{align*}
\int_{I_v(1)}
\psi_v(\varpi_v^{-1}\beta)\overline{h_{3,v}^{\circ}\left(k\begin{pmatrix}
1 & \beta\\
& 1
\end{pmatrix}\begin{pmatrix}
\varpi_v\\
& 1
\end{pmatrix}
,\overline{s_3}\right)}
\overline{h_4\left(I_2,\overline{s_4}\right)})dk'=0.
\end{align*}

Thus $J_v^{(3)}=0$. Therefore,
\begin{align*}
\widetilde{\Psi}_{\mathrm{Geo}}^{(1)}(\mathbf{s},R(f_{\mathfrak{q}})\boldsymbol{h})\equiv 0. 
\end{align*}

Similarly, $\widetilde{\Psi}_{\mathrm{Geo}}^{(2)}(\mathbf{s},R(f_{\mathfrak{q}})\boldsymbol{h})\equiv 0$. Hence \eqref{e9.8} holds. 
\end{proof}

\subsection{Vanishing of  \texorpdfstring{$\widetilde{\Psi}_{\mathrm{Geo}}^{(i)}(\mathbf{s},R(f_{\mathfrak{q}})\boldsymbol{h}): i=5,6$}{}}
\begin{lemma}\label{lem9.5}
Let $5\leq i\leq 6$, and let $\mathbf{s}\in \mathbf{B}_{\varepsilon}^4$. We have
\begin{equation}\label{e9.9}
\widetilde{\Psi}_{\mathrm{Geo}}^{(i)}(\mathbf{s},R(f_{\mathfrak{q}})\boldsymbol{h})\equiv 0.
\end{equation}
\end{lemma}
\begin{proof}
Consider the local component of $\widetilde{\Psi}_{\mathrm{Geo}}^{(5)}(\mathbf{s},R(f_{\mathfrak{q}})\boldsymbol{h})$ at $v\mid\mathfrak{q}$:
\begin{align*}
J_v^{(5)}:=\int_{I_v(1)}
\overline{f_{\mathfrak{q}}(g)}\int_{N(F_v)\backslash\overline{G}(F_v)}
W_{h_{1,v},\overline{R(g)h_{3,v}}}(x,s_1,\overline{s_3})W_{2,v}(x,s_2)
\overline{h_{4,v}(xg,\overline{s_4})}dxdg.
\end{align*}

By Iwasawa decomposition and Lemma \ref{lem4.2}, 
we obtain 
\begin{multline*}
J_v^{(5)}=\frac{1}{\Vol(I_v(1))}\int_{I_v(1)}
\psi_v(\varpi_v^{-1}\beta)\sum_{i\in \mathbb{Z}}q_v^{(1/2-s_4)i}\sum_{\alpha\in \mathbb{F}_v^{\times}}\int_{I_v^*(1)}
W_v\left(\begin{pmatrix}
\alpha\varpi_v^i\\
& 1
\end{pmatrix}k\right)\\
W_{2,v}\left(\begin{pmatrix}
\alpha\varpi_v^i\\
& 1
\end{pmatrix}k
,s_2\right)
\overline{h_{4,v}(k,\overline{s_4})}dxdg,
\end{multline*}
where 
\begin{multline*}
W_v\left(\begin{pmatrix}
\alpha\varpi_v^i\\
& 1
\end{pmatrix}k\right)=\int_{N(F_v)}
h_{1,v}^{\circ}\left(wu\begin{pmatrix}
\alpha\varpi_v^i\\
& 1
\end{pmatrix}k\begin{pmatrix}
\varpi_v\\
& 1
\end{pmatrix}
,s_1\right)\\
h_{3,v}^{\circ}\left(wu\begin{pmatrix}
\alpha\varpi_v^i\\
& 1
\end{pmatrix}k\begin{pmatrix}
1 & \beta\\
& 1
\end{pmatrix}\begin{pmatrix}
\varpi_v\\
& 1
\end{pmatrix},\overline{s_3}\right)\theta_v(u)du.
\end{multline*}

As a consequence of the identity 
\begin{align*}
\begin{pmatrix}
1\\
\gamma'\varpi_v & 1
\end{pmatrix}\begin{pmatrix}
1 & \beta\\
& 1
\end{pmatrix}=\begin{pmatrix}
1 & \beta(1+\beta\gamma' \varpi_v)^{-1}\\
& 1
\end{pmatrix}\begin{pmatrix}
(1+\beta\gamma' \varpi_v)^{-1}
& \\
\gamma' \varpi_v
&
1 + \gamma' \varpi_v \beta
\end{pmatrix},
\end{align*} 
we obtain the expansion 
\begin{multline}\label{9.10}
J_v^{(5)}=\Vol(I_v^*(1))\Vol(I_v(1))^{-1}\sum_{i\in \mathbb{Z}}q_v^{(1/2-s_4)i}W_{2,v}\left(\begin{pmatrix}
\alpha\varpi_v^i\\
& 1
\end{pmatrix}
,s_2\right)\\
\sum_{\alpha\in \mathbb{F}_v^{\times}}\int_{I_v(1)}
\psi_v(\varpi_v^{-1}\beta)\overline{\psi}_v(\alpha\varpi_v^i\beta)dk'\overline{h_{4,v}(I_2,\overline{s_4})}\\
\int_{N(F_v)}
h_{1,v}^{\circ}\left(wu\begin{pmatrix}
\varpi_v^{i+1}\\
& 1
\end{pmatrix},s_1\right)
h_{3,v}^{\circ}\left(wu\begin{pmatrix}
\varpi_v^{i+1}\\
& 1
\end{pmatrix},\overline{s_3}\right)\theta_v(u)du.
\end{multline}

By Lemma \ref{lem4.1}, 
\begin{equation}\label{}
W_{2,v}\left(\begin{pmatrix}
\alpha\varpi_v^i\\
& 1
\end{pmatrix}
,s_2\right)\equiv 0\ \ \ \text{unless\, $i\geq 0$}.
\end{equation}

For $i\geq 0$, by orthogonality of additive characters, 
\begin{align*}
\int_{I_v(1)}
\psi_v(\varpi_v^{-1}\beta)\overline{\psi}_v(\alpha\varpi_v^i\beta)dk'=\int_{I_v(1)}
\psi_v(\varpi_v^{-1}\beta)dk'=0. 
\end{align*}

Therefore, $\widetilde{\Psi}_{\mathrm{Geo}}^{(5)}(\mathbf{s},R(f_{\mathfrak{q}})\boldsymbol{h})\equiv 0$. Similarly, we obtain $\widetilde{\Psi}_{\mathrm{Geo}}^{(6)}(\mathbf{s},R(f_{\mathfrak{q}})\boldsymbol{h})\equiv 0$, obtaining \eqref{e9.9}. 
\end{proof}

\subsection{Vanishing of  \texorpdfstring{$\widetilde{\Psi}_{\mathrm{Geo}}^{(i)}(\mathbf{s},R(f_{\mathfrak{q}})\boldsymbol{h}): i=7,8$}{}}
\begin{lemma}\label{lem9.6}
Let $7\leq i\leq 8$, and let $\mathbf{s}\in \mathbf{B}_{\varepsilon}^4$. We have
\begin{equation}\label{e9.12}
\widetilde{\Psi}_{\mathrm{Geo}}^{(i)}(\mathbf{s},R(f_{\mathfrak{q}})\boldsymbol{h})\equiv 0.
\end{equation}
\end{lemma}
\begin{proof}
Consider the local component of $\widetilde{\Psi}_{\mathrm{Geo}}^{(7)}(\mathbf{s},R(f_{\mathfrak{q}})\boldsymbol{h})$ at $v\mid\mathfrak{q}$:
\begin{align*}
J_v^{(7)}:=\int_{I_v(1)}
\overline{f_{\mathfrak{q}}(g)}\int_{N(F_v)\backslash\overline{G}(F_v)}
W_{h_{1,v},h_{2,v}}(x,s_1,s_2)\overline{W_{3,v}(xg,\overline{s_3})}
\overline{h_{4,v}(xg,\overline{s_4})}dxdg.
\end{align*}

By Iwasawa decomposition, 
we obtain
\begin{multline}\label{9.13}
J_v^{(7)}=\Vol(I_v(1))^{-1}\int_{I_v(1)}
\psi_v(\varpi_v^{-1}\beta)\sum_{i\in \mathbb{Z}}q_v^{(1/2-s_4)i}\sum_{\alpha\in \mathbb{F}_v^{\times}}\int_{I_v^*(1)}\overline{h_{4,v}\left(k,\overline{s_4}\right)}
\\
W_{h_{1,v},h_{2,v}}\left(\begin{pmatrix}
\alpha\varpi_v^i\\
& 1
\end{pmatrix}k
,s_1,s_2\right)\overline{W_{3,v}^{\circ}\left(\begin{pmatrix}
\alpha\varpi_v^i\\
& 1
\end{pmatrix}k\begin{pmatrix}
\varpi_v & \beta\\
& 1
\end{pmatrix}
,\overline{s_3}\right)}
dkdk'.
\end{multline}

Substituting the parametrization 
\begin{align*}
k=\kappa\begin{pmatrix}
1 & \beta'\\
& 1
\end{pmatrix}\begin{pmatrix}
1+\alpha'\varpi_v & \\
\gamma'\varpi_v & 1+\delta'\varpi_v 
\end{pmatrix}\in I_v^*(1)
\end{align*}
into \eqref{9.13}, along with Lemma \ref{lem4.2} and the formula \eqref{eq9.6}, we derive 
\begin{equation}\label{9.14}
J_v^{(7)}=\overline{h_{4,v}\left(I_2,\overline{s_4}\right)}q_v^{-\frac{1}{2}+s_4}\overline{W_{3,v}^{\circ}\left(I_2
,\overline{s_3}\right)}
\int_{I_v^*(1)}F(k)dk.
\end{equation}
where 
\begin{align*}
F(k):=\overline{\psi}_v(
\varpi_v^{-1}t\gamma')
W_{h_{1,v},h_{2,v}}\left(\begin{pmatrix}
\varpi_v^{-1}\\
& 1
\end{pmatrix}\begin{pmatrix}
1 & \\
\gamma'\varpi_v & 1
\end{pmatrix}
,s_1,s_2\right).
\end{align*}

By definition \eqref{e2.3} we have 
\begin{equation}\label{e9.16}
F(k)=\overline{\psi}_v(
\varpi_v^{-1}t\gamma')\int
h_{1,v}^{\circ}\left(wu, s_1\right)
h_{2,v}\left(wu\begin{pmatrix}
\varpi_v^{-1} & \\
\gamma'\varpi_v & 1
\end{pmatrix},s_2\right)\overline{\theta}_v(u)du,
\end{equation}
where the integral is over $N(F_v)$. By \eqref{f2.4} and \eqref{e6.2}, 
\begin{align*}
h_{2,v}\left(wu\begin{pmatrix}
\varpi_v^{-1} & \\
\gamma'\varpi_v & 1
\end{pmatrix},s_2\right)
=q_v^{\frac{1}{2}+s_2}
\int_{F_v^{\times}}
\Phi_v(a\varpi_v^{-1}+ab\gamma'\varpi_v,ab)\overline{\omega}_v(a)|a|_v^{2s_2+1}d^{\times}a,
\end{align*}
where the term $\Phi_v(a\varpi_v^{-1}+ab\gamma'\varpi_v,ab)$ is equal to 
\begin{align*}
\psi_v(\varpi_v^{-1} t \gamma')\Vol(I_v^*(1))^{-1}
\mathbf{1}_{\mathfrak{p}_v^2}(a)\mathbf{1}_{\mathcal{O}_v^{\times}}(ab)
\psi_v(\varpi_v^{-3} t b^{-1})
\omega_v(ab).
\end{align*}

Substituting these into \eqref{e9.16} leads to 
\begin{multline}\label{e9.17}
F(k)=q_v^{\frac{1}{2}+s_2}\Vol(I_v^*(1))^{-1}\int_{F_v}
h_{1,v}^{\circ}\left(\begin{pmatrix}
& 1\\
1& b
\end{pmatrix}
, s_1\right)
\omega_v(b)\mathbf{1}_{e_v(b)\leq -2}\\
\psi_v(\varpi_v^{-3} t b^{-1})\overline{\psi}_v(b)|b|_v^{-1-2s_2}db.
\end{multline} 

Changing the variable $b\mapsto b^{-1}$ into \eqref{e9.17}, together with  
\begin{align*}
\begin{pmatrix}
& 1\\
1& b^{-1}
\end{pmatrix}=\begin{pmatrix}
b & 1\\
 & b^{-1}
\end{pmatrix}\begin{pmatrix}
-1 & \\
b & 1
\end{pmatrix},
\end{align*}
we obtain 
\begin{equation}\label{eq9.18}
F(k)=q_v^{\frac{1}{2}+s_2}\Vol(I_v^*(1))^{-1}h_{1,v}^{\circ}\left(I_2
, s_1\right)
\sum_{l=2}^{\infty}\overline{\omega}_v(\varpi_v^{l})q_v^{-(1+2s_1+2s_2)l}B_l,
\end{equation}
where 
\begin{align*}
B_l:=\int_{\mathcal{O}_v^{\times}}
\psi_v(\varpi_v^{l-3}t\beta'-\varpi_v^{-l}\beta'^{-1})\overline{\omega}_v(\beta')d\beta'.
\end{align*}

Notice that $r_{\omega_v}\leq 1$. For $l\geq 2$, we have 
\begin{equation}\label{e9.19}
\int_{\mathcal{O}_v^{\times}}
\psi_v(\varpi_v^{l-3}t\beta'-\varpi_v^{-l}\beta'^{-1})\overline{\omega}_v(\beta')d\beta'\equiv 0.
\end{equation}

By \eqref{e9.19} and \eqref{eq9.18}, we have $F(k)\equiv 0$. 
Hence \eqref{9.14} implies
\[
\widetilde{\Psi}_{\mathrm{Geo}}^{(7)}(\mathbf{s},R(f_{\mathfrak{q}})\boldsymbol{h}) \equiv 0,
\]
and similarly $\widetilde{\Psi}_{\mathrm{Geo}}^{(8)}(\mathbf{s},R(f_{\mathfrak{q}})\boldsymbol{h}) \equiv 0$, proving \eqref{e9.12}. 
\end{proof}

\subsection{Proof of Proposition \ref{prop9.1}}
For each $1\leq i\leq 8$, the estimate \eqref{9.1} follows from Corollary \ref{cor9.3}, Lemmas \ref{lem9.4}, \ref{lem9.5}, \ref{lem9.6}, and the definition \eqref{2.12}.

\section{Estimates of  Degenerate Terms: the Dual Side}\label{sec10}

In this section we establish upper bounds for the dual degenerate terms: 
\begin{prop}\label{prop10.1}
Let $1\leq i\leq 8$. We have
\begin{align*}
\widetilde{\Psi}_{\mathrm{Dual}}^{(i)}(\mathbf{0}\mid R(f_{\mathfrak{q}})\boldsymbol{h})\ll_{F,\varepsilon} \mathbf{C}_{\infty}^{\varepsilon}N_F(\mathfrak{q})^{2+\varepsilon}.
\end{align*}	
\end{prop}

\subsection{Structure of  \texorpdfstring{$\widetilde{\Psi}_{\mathrm{Dual}}^{(i)}(\mathbf{s},R(f_{\mathfrak{q}})\boldsymbol{h})$}{}} 
Recall the definition \eqref{c3.5}: 
\begin{align*}
\widetilde{\Psi}^*(\lambda,\mu;\mathbf{s},R(f_{\mathfrak{q}})\boldsymbol{h}):=\frac{1}{2}\sum_{h\in \mathfrak{B}(\mu,\overline{\mu})}\widetilde{\Psi}^*(\lambda;W_h;\mathbf{s},R(f_{\mathfrak{q}})\boldsymbol{h}),
\end{align*}
where $\widetilde{\Psi}^*(\lambda;W_h;\mathbf{s},R(f_{\mathfrak{q}})\boldsymbol{h})$ is defined by 
\begin{align*}
\widetilde{\Psi}(\overline{W_{E(\cdot,h,-\overline{\lambda})}},R(g)W_1(\cdot,s_1),\overline{h_3(\cdot,\overline{s_3})})\widetilde{\Psi}(W_{E(\cdot,h,\lambda)}^*,R(g)W_2(\cdot,s_2),\overline{h_4(\cdot,\overline{s_4})}).
\end{align*}

Define the product of the complete $L$-functions 
\begin{multline}\label{10.2}
\mathbf{L}^*(\lambda,\mu;\mathbf{s}):=\Lambda(1+2\lambda,\mu^2)^{-1}\Lambda(1-2\lambda,\overline{\mu}^{2})^{-1}\Lambda(1/2+s_3+s_1-\lambda,\overline{\mu})\\
\Lambda(1/2+s_3+s_1+\lambda,\mu)\Lambda(1/2+s_3-s_1-\lambda,\overline{\mu})\Lambda(1/2+s_3-s_1+\lambda,\mu)\\
\Lambda(1/2+s_4+s_2+\lambda,\mu)\Lambda(1/2+s_4+s_2-\lambda,\overline{\mu})\\
\Lambda(1/2+s_4-s_2+\lambda,\mu\overline{\omega})\Lambda(1/2+s_4-s_2-\lambda,\overline{\mu}\overline{\omega}).
\end{multline}

By \cite[\textsection 8]{Yan25}, the ratio  
\begin{equation}\label{f10.2}
\boldsymbol{e}^*(\lambda,\mu;\mathbf{s}):=\widetilde{\Psi}^*(\lambda,\mu;\mathbf{s},R(f_{\mathfrak{q}})\boldsymbol{h})/\mathbf{L}^*(\lambda,\mu;\mathbf{s})
\end{equation}
is an entire function of $(\lambda,\mathbf{s})\in \mathbb{C}^5$. Therefore, 
\begin{align*} 
\widetilde{\Psi}_{\mathrm{Dual}}^{(1)}(\mathbf{s},R(f_{\mathfrak{q}})\boldsymbol{h})=-\mathbf{1}_{\mu=\mathbf{1}}\cdot \boldsymbol{e}^*(s_3+s_1-1/2,\mu;\mathbf{s})\underset{\lambda=s_3+s_1-\frac{1}{2}}{\Res}\ \mathbf{L}^*(\lambda,\mu;\mathbf{s}).
\end{align*} 
Similarly, the other terms
$\widetilde{\Psi}_{\mathrm{Dual}}^{(i)}(\mathbf{s},R(f_{\mathfrak{q}})\boldsymbol{h})$
can be expressed in the same form.

By definition and the Rankin--Selberg theory, $\boldsymbol{e}^*(\lambda,\mu;\mathbf{s})=\prod_v\boldsymbol{e}_v^*(\lambda,\mu;\mathbf{s})$, where 
\begin{align*}
\boldsymbol{e}_v^*(\lambda,\mu;\mathbf{s})=\sum_{h_v\in \mathfrak{B}(\mu_v,\overline{\mu}_v)}\frac{\widetilde{\Psi}_{\mathrm{Dual},v}^{(1)}(\mathbf{s},R(f_{\mathfrak{q}})\boldsymbol{h})}{\mathbf{L}_v^*(\lambda,\mu_v;\mathbf{s})}.
\end{align*}

The local factors $\boldsymbol{e}_v^*(\lambda,\mu;\mathbf{s})$ were analyzed in loc.\ cit.\ for all $v\nmid\mathfrak{q}$. In particular, for all but finitely many places $v$, one has
\begin{align*}
\boldsymbol{e}_v^*(\lambda,\mu;\mathbf{s})\equiv W_{1,v}(I_2,s_1)\overline{W_{3,v}(I_2,\overline{s}_3)}. 
\end{align*}

Hence, in order to bound each
$\widetilde{\Psi}_{\mathrm{Dual}}^{(i)}(\mathbf{s},R(f_{\mathfrak{q}})\boldsymbol{h})$,
it remains to analyze the local factor
$\boldsymbol{e}_v^*(\lambda,\mu;\mathbf{s})$
at the places $v\mid\mathfrak{q}$.  

Let $\mathbf{S}_{\varepsilon}^2:=\big\{(0,0,s_3,s_4)\in \mathbb{C}^4:\ |s_3|=5\varepsilon^2,\ |s_4|=10\varepsilon^2\big\}$. It follows from \eqref{2.12} and the entireness of $\boldsymbol{e}^*(\lambda,\mu;\mathbf{s})$ that 
\begin{equation}\label{e10.3}
\widetilde{\Psi}_{*}^{(i)}(\mathbf{0}\mid R(f_{\mathfrak{q}})\boldsymbol{h})=\frac{1}{4\pi^2}\iint_{\mathbf{S}_{\varepsilon}^2}\frac{\widetilde{\Psi}_{*}^{(i)}(\mathbf{s},R(f_{\mathfrak{q}})\boldsymbol{h})}{s_3s_4}ds_3ds_4. 
\end{equation}

Suppose $\lambda\in i\mathbb{R}$ and $\mu_v\in \{\mathbf{1}, \omega_v, \overline{\omega}_v\}$. Let $\sigma_v=\xi_v\boxplus\xi_v^{-1}$ with $\xi_v=\mu_v|\cdot|_v^{\lambda}$. We obtain from Lemmas \ref{lem4.1} and \ref{lem4.2} that 
\begin{align*}
\widetilde{\Psi}^*(\lambda;W_h;\mathbf{s},R(f_{\mathfrak{q}})\boldsymbol{h})\equiv 0\ \ \ \text{unless\ \ $\sigma_v^{I_v(1)}\neq 0$}.
\end{align*}

Let $f_1$ and $f_2$ be sections in $\sigma_v$, defined as in \textsection\ref{sec5}. Since $\lambda\in i\mathbb{R}$, $\sigma_v$ is unitary. So 
\begin{align*}
&\langle f_1, f_1\rangle_v:=\int_{K_v}|f_1(k)|^2dk=\Vol(K_v[1])=(1+q_v)^{-1},\\
&\langle f_2, f_2\rangle_v:=\int_{K_v}|f_2(k)|^2dk=1-\Vol(K_v[1])=q_v(1+q_v)^{-1}. 
\end{align*} 

For $1 \leq j \leq 2$, let $W_v^{(j)}$ denote the Whittaker function associated with $f_j$; see \eqref{f4.2}. Define the normalized Whittaker function
\begin{align*}
W_v^{(f_j,\sharp)}
:=
\langle f_j, f_j \rangle_v^{-1/2} \, W_v^{(j)}.
\end{align*}
Then $\{ W_v^{(f_1,\sharp)},\, W_v^{(f_2,\sharp)} \}$ forms an orthonormal basis of $\sigma_v^{I_v(1)}$.

Each $W_v \in \{ W_v^{(f_1,\sharp)}, W_v^{(f_2,\sharp)} \}$ depends on the spectral parameter $\lambda \in i\mathbb{R}$; we therefore write it as $W_v(\cdot,\lambda)$. Moreover, $W_v(\cdot,\lambda)$ admits a meromorphic continuation to $\lambda \in \mathbb{C}$.

In parallel to \eqref{e8.5} we obtain 
\begin{align*}
\boldsymbol{e}_v^*(\lambda,\mu;\mathbf{s})=N_F(\mathfrak{q})^{-1}
\sum_{W_v(\cdot,\lambda)\in \{W_v^{(f_1,\sharp)}, W_v^{(f_2,\sharp)}\}}
\frac{\mathcal{I}_{v}(W_v(\cdot,\lambda);I_2)\mathcal{J}_{v}(W_v(\cdot,-\overline{\lambda}))}{\mathbf{L}_v^*(\lambda,\mu_v;\mathbf{s})},
\end{align*}
where 
\begin{align*}
\mathcal{I}_{v}(W_v(\cdot,\lambda);I_2):=\int_{N(F_v)\backslash G(F_v)}W_v(x,\lambda)\overline{W_{2,v}(x)}\Phi_{2,v}(\mathbf{e}_2x)|\det x|_v^{\frac{1}{2}+s_4}dx,
\end{align*}
and $\mathcal{J}_{v}(W_v(\cdot,-\overline{\lambda}))$ is defined by 
\begin{align*}
c\cdot \sum_{\beta\in \mathbb{F}_v}\overline{\psi}_v(\varpi_v^{-1}\beta)\,\int_{N(F_v)\backslash \overline{G}(F_v)}W_{1,v}^{\circ}\left(xA_{\mathfrak{q}}\right)\overline{W_{3,v}^{\circ}\left(xA_{\mathfrak{q}},\overline{s_3}\right)f(x,-\overline{\lambda})}dx.
\end{align*}
Here $f(\cdot,-\overline{\lambda})$ is the section corresponding to $W_v(\cdot,-\overline{\lambda})$, and 
\begin{align*}
|c|=\sqrt{\zeta_v(1)\cdot |\gamma(1/2,\pi_{1,v}\times\widetilde{\pi}_{1,v}\times\xi_v)|}
\end{align*}

Notice that when $s_4=0$, the integral $\mathcal{I}_{v}(\cdot;I_2)$ coincides with the one defined in \eqref{e4.6} of \textsection\ref{sec5.2}, and when $s_3=0$, $\mathcal{J}_{v}(\cdot)$ coincides with that defined in \eqref{e5.43} of \textsection\ref{sec5.3}. 

To emphasize the close connection with the previous calculations in \textsection\ref{sec5}, we retain the notation $\mathcal{I}_{v}(\cdot;I_2)$ and $\mathcal{J}_{v}(\cdot)$ throughout. We hope this slight abuse of notation will not cause any confusion.

\subsection{Upper Bound for \texorpdfstring{$\boldsymbol{e}_v^*(\lambda,\mu;\mathbf{s})$}{}}
Let $\mathbf{s}\in \mathbf{S}_{\varepsilon}^2$. Let 
\begin{multline*}
\mathfrak{S}:=\Big\{s_3+s_1-1/2, 1/2-s_3-s_1, s_3-s_1-1/2, 1/2-s_3+s_1,\\
s_4-s_2-1/2, 1/2-s_4+s_2, s_4+s_2-1/2, 1/2-s_4-s_2\Big\}.
\end{multline*}

Let $\lambda\in \mathfrak{S}$. Then $\mathbf{L}_v^*(\lambda,\mu_v;\mathbf{s})\gg N_F(\mathfrak{q})^{-100\varepsilon}$. Consequently, we obtain 
\begin{align*}
\boldsymbol{e}_v^*(\lambda,\mu;\mathbf{s})\ll N_F(\mathfrak{q})^{-1+100\varepsilon}
\sum_{W_v(\cdot,\lambda)\in \{W_v^{(f_1,\sharp)}, W_v^{(f_2,\sharp)}\}}
\big|\mathcal{I}_{v}(W_v(\cdot,\lambda);I_2)\mathcal{J}_{v}(W_v(\cdot,-\overline{\lambda}))\big|.	
\end{align*}

\begin{lemma}\label{lem10.2}
Let $\mathbf{s}\in \mathbf{S}_{\varepsilon}^2$ and $\lambda\in \mathfrak{S}$. Suppose $v\mid\mathfrak{q}$, the exponent $r_{\omega_v}=0$ and $\mu_v\in \{\mathbf{1}, \omega_v, \overline{\omega}_v\}$. Then  
\begin{equation}\label{f10.4}
\boldsymbol{e}_v^*(\lambda,\mu;\mathbf{s})\ll N_F(\mathfrak{q})^{\frac{3}{2}+200\varepsilon}.
\end{equation}
\end{lemma}
\begin{proof}
Suppose $W_v(\cdot,\lambda)=W_v^{(f_1,\sharp)}$. For $|\Re(\lambda)|<1/2$, we have by  \eqref{5.25} that 
\begin{equation}\label{e10.4}
\mathcal{I}_{v}(W_v(\cdot,\lambda);I_2)=\frac{\xi_v(-1)\omega_v(\varpi_v)}{\Vol(I_v^*(1))\zeta_v(1)\langle f_1, f_1 \rangle_v^{1/2}q_v^{d_v}}\cdot S(\lambda,s_4),
\end{equation}
where 
\begin{multline*}
S(\lambda,s_4):=\sum_{i\geq 0}q_v^{-(1/2+s_4)i}\xi_v^{-i}(\varpi_v)\bigg[\zeta_v(1)^{-1}\mathbf{1}_{i\geq 1}\sum_{m=1}^{i}\xi_v^{2m}(\varpi_v)-\xi_v^{2i+2}(\varpi_v)q_v^{-1}\bigg]\\
\Big[q_v\zeta_v(1)^{-1}\mathbf{1}_{i\geq 2}\sum_{l=0}^{i-2}\overline{\omega}_v^{l+1}(\varpi_v)-(\overline{\xi}_v(t)\overline{\omega}_v^{i}(\varpi_v) 
-1)\mathbf{1}_{i\geq 1}+q_v^{-1}\zeta_v(1)\overline{\xi}_v(t)\Big].
\end{multline*}

Expand the brackets we obtain 
\begin{equation}\label{e5.31}
S(\lambda,s_4)=S_1(\lambda,s_4)+S_2(\lambda,s_4)+S_3(\lambda,s_4)+S_4(\lambda,s_4)+S_5(\lambda,s_4)+S_6(\lambda,s_4),	
\end{equation}
where 
\begin{align*}
&S_1(\lambda,s_4):=q_v\zeta_v(1)^{-2}\sum_{i\geq 2}q_v^{-(1/2+s_4)i}\xi_v^{-i}(\varpi_v)\sum_{m=1}^{i}\xi_v^{2m}(\varpi_v)
\sum_{l=0}^{i-2}\overline{\omega}_v^{l+1}(\varpi_v),\\
&S_2(\lambda,s_4):=-\zeta_v(1)^{-1}\sum_{i\geq 1}q_v^{-(1/2+s_4)i}\xi_v^{-i}(\varpi_v)\sum_{m=1}^{i}\xi_v^{2m}(\varpi_v)
(\overline{\xi}_v(t)\overline{\omega}_v^{i}(\varpi_v) 
-1),\\
&S_3(\lambda,s_4):=q_v^{-1}\overline{\xi}_v(t)\sum_{i\geq 1}q_v^{-(1/2+s_4)i}\xi_v^{-i}(\varpi_v)\sum_{m=1}^{i}\xi_v^{2m}(\varpi_v),\\
&S_4(\lambda,s_4):=-\zeta_v(1)^{-1}\sum_{i\geq 2}q_v^{-(1/2+s_4)i}\xi_v^{i+2}(\varpi_v)\sum_{l=0}^{i-2}\overline{\omega}_v^{l+1}(\varpi_v),\\
&S_5(\lambda,s_4):=q_v^{-1}\xi_v^{2}(\varpi_v)\sum_{i\geq 1}q_v^{-(1/2+s_4)i}\xi_v^{i}(\varpi_v)(\overline{\xi}_v(t)\overline{\omega}_v^{i}(\varpi_v) 
-1),\\
&S_6(\lambda,s_4):=-q_v^{-2}\xi_v^{2}(\varpi_v)\zeta_v(1)\overline{\xi}_v(t)\sum_{i\geq 0}q_v^{-(1/2+s_4)i}\xi_v^{i}(\varpi_v).
\end{align*}

Upon making the change of variables $i \mapsto i+2$, a direct computation shows that
\begin{multline}\label{e10.6}
S_1(\lambda,s_4)=\overline{\omega}_v(\varpi_v)\big[1+\xi_v^{2}(\varpi_v)-q_v^{-1/2-s_4}\xi_v(\varpi_v)(1+\overline{\omega}_v(\varpi_v))\big]\\
q_v^{-2s_4}\zeta_v(1)^{-2}L(1/2+s_4,\sigma_v)L(1/2+s_4,\sigma_v\times\overline{\omega}_v). 
\end{multline}

Making the change of variable $i\mapsto i+1$ yields 
\begin{equation}\label{e10.8}
S_3(\lambda,s_4)=q_v^{-\frac{3}{2}-s_4}\overline{\xi}_v(t)\xi_v(\varpi_v)L(1/2+s_4,\sigma_v),
\end{equation}
and 
\begin{multline}\label{e10.9}
S_4(\lambda,s_4)=-q_v^{-1-2s_4}\xi_v^{4}(\varpi_v)\overline{\omega}_v(\varpi_v)\zeta_v(1)^{-1}(1-q_v^{-1/2-s_4}\xi_v^{-1}(\varpi_v))\\
(1-q_v^{-1/2-s_4}\xi_v^{-1}(\varpi_v)\overline{\omega}_v(\varpi_v))L(1/2+s_4,\sigma_v)L(1/2+s_4,\sigma_v\times\overline{\omega}_v).
\end{multline}

Likewise, we obtain the explicit formula 
\begin{multline}\label{e10.7}
S_2(\lambda,s_4)=-\zeta_v(1)^{-1}\xi_v(\varpi_v)
\Big[
q_v^{-1/2-s_4}\overline{\omega}_v\xi_v^{-1}(\varpi_v)(1-\overline{\xi}_v(t))(1+\xi_v^2(\varpi_v))\\
+\overline{\xi}_v(t)\overline{\omega}_v(\varpi_v)-1+
q_v^{-(1+2s_4)}\overline{\omega}_v(\varpi_v)(\overline{\xi}_v(t)-\overline{\omega}_v(\varpi_v))\Big]q_v^{-1/2-s_4}L(\cdots),
\end{multline}
where $L(\cdots):=L(1/2+s_4,\sigma_v)L(1/2+s_4,\sigma_v\times\overline{\omega}_v)$, 
and 
\begin{multline}\label{e10.10}
S_5(\lambda,s_4)=\xi_v^{3}(\varpi_v)\big[\overline{\xi}_v(t)\overline{\omega}_v(\varpi_v)-1
+q_v^{-1/2-s_4}\overline{\omega}_v\xi_v(\varpi_v)(1-\overline{\xi}_v(t))\big]\\
(1-q_v^{-1/2-s_4}\xi_v^{-1}(\varpi_v))
(1-q_v^{-1/2-s_4}\xi_v^{-1}(\varpi_v)\overline{\omega}_v(\varpi_v))q_v^{-\frac{3}{2}-s_4}L(\cdots).
\end{multline}

Finally, we have 
\begin{equation}\label{e10.11}
S_6(\lambda,s_4)=-q_v^{-2}\xi_v^{2}(\varpi_v)\zeta_v(1)\overline{\xi}_v(t)(1-q_v^{-1/2-s_4}\xi_v^{-1}(\varpi_v))L(1/2+s_4,\sigma_v).
\end{equation}

As a consequence, substituting \eqref{e5.31}, \eqref{e10.6}, \eqref{e10.8}, \eqref{e10.9}, \eqref{e10.7}, \eqref{e10.10} and \eqref{e10.11} into \eqref{e10.4} yields a meromorphic continuation of $\mathcal{I}_{v}(W_v(\cdot,\lambda);I_2)$ into $(\lambda,s_4)\in \mathbb{C}^2$:
\begin{equation}\label{10.12}
\mathcal{I}_{v}(W_v(\cdot,\lambda);I_2)\ll q_v^{\frac{1}{2}}\Vol(I_v^*(1))^{-1}\sum_{i=1}^6\big|S_i(\lambda,s_4)\big|.	
\end{equation} 

Moreover, by \eqref{e5.39} we obtain 
\begin{multline*}
\mathcal{J}_{v}(W_v(\cdot,-\overline{\lambda}))=c\cdot \Vol(I_v^*(1))\sum_{i\in \mathbb{Z}}q_v^{(1/2-\lambda)i}\overline{\xi}_v^i(\varpi_v)\\
W_{1,v}^{\circ}\left(\begin{pmatrix}
\varpi_v^{i+1}\\
& 1
\end{pmatrix}
\right)\overline{W_{3,v}^{\circ}\left(\begin{pmatrix}
\varpi_v^{i+1}\\
& 1
\end{pmatrix}
\right)}\cdot S(i),
\end{multline*}
where $S(i)$ is given by \eqref{5.40}. Therefore, 
\begin{equation}\label{10.13}
\mathcal{J}_{v}(W_v(\cdot,-\overline{\lambda}))\lll \Vol(I_v^*(1))q_v^{\frac{1}{2}+100\varepsilon}|\overline{\xi}_v^{-1}(\varpi_v)||W_{1,v}^{\circ}(I_2)|^2. 
\end{equation}
Here, $\pi_{1,v}=\mathbf{1}\boxplus \mathbf{1}$ and $\overline{\pi}_{3,v}=|\cdot|_v^{s_3}\boxplus|\cdot|_v^{-s_3}$. 

Notice that for $\mathbf{s}\in \mathbf{S}_{\varepsilon}^2$ and $\lambda\in \mathfrak{S}$, 
\begin{equation}\label{e10.15}
|\xi_v(\varpi_v)|\leq N_F(\mathfrak{q})^{1/2+20\varepsilon^2}.
\end{equation}
By \eqref{10.12},  \eqref{10.13} and \eqref{e10.15} we obtain 
\begin{equation}\label{f10.5}
\big|\mathcal{I}_{v}(W_v(\cdot,\lambda);I_2)\mathcal{J}_{v}(W_v(\cdot,-\overline{\lambda}))\big|\ll   N_F(\mathfrak{q})^{\frac{5}{2}+100\varepsilon}.
\end{equation}

Suppose $W_v(\cdot,\lambda)=W_v^{(f_2,\sharp)}$. A modification of 
\eqref{c5.29} and \eqref{5.55} as above yields 
\begin{equation}\label{10.15}
\big|\mathcal{I}_{v}(W_v(\cdot,\lambda);I_2)\mathcal{J}_{v}(W_v(\cdot,-\overline{\lambda}))\big|\ll N_F(\mathfrak{q})^{\frac{5}{2}+100\varepsilon}.
\end{equation}

Therefore, \eqref{f10.4} follows from \eqref{f10.5} and \eqref{10.15}. 
\end{proof}

\begin{lemma}\label{lem10.3}
Let $\mathbf{s}\in \mathbf{S}_{\varepsilon}^2$ and $\lambda\in \mathfrak{S}$. Suppose $v\mid\mathfrak{q}$,  $r_{\omega_v}=1$, and $\mu_v=\mathbf{1}$.  Then 
\begin{equation}\label{f10.18}
\boldsymbol{e}_v^*(\lambda,\mu;\mathbf{s})\ll N_F(\mathfrak{q})^{\frac{3}{2}+200\varepsilon}.
\end{equation}
\end{lemma}
\begin{proof}
Suppose $W_v(\cdot,\lambda)=W_v^{(f_1,\sharp)}$. A modification of \eqref{5.21} and \eqref{5.44} yields 
\begin{equation}\label{10.19}
\big|\mathcal{I}_{v}(W_v(\cdot,\lambda);I_2)\mathcal{J}_{v}(W_v(\cdot,-\overline{\lambda}))\big|\ll N_F(\mathfrak{q})^{1+100\varepsilon}.
\end{equation}

Suppose $W_v(\cdot,\lambda)=W_v^{(f_2,\sharp)}$. A modification of \eqref{c5.28} and \eqref{5.55} leads to  
\begin{equation}\label{10.20}
\big|\mathcal{I}_{v}(W_v(\cdot,\lambda);I_2)\mathcal{J}_{v}(W_v(\cdot,-\overline{\lambda}))\big|\ll N_F(\mathfrak{q})^{\frac{5}{2}+100\varepsilon}.
\end{equation}

Therefore, \eqref{f10.18} follows from \eqref{10.19} and \eqref{10.20}. 
\end{proof}

\begin{lemma}\label{lem10.4}
Let $\mathbf{s}\in \mathbf{S}_{\varepsilon}^2$ and $\lambda\in \mathfrak{S}$. Suppose $v\mid\mathfrak{q}$,  $r_{\omega_v}=1$, and $\mu_v=\omega_v$. Then 
\begin{equation}\label{f10.21}
\boldsymbol{e}_v^*(\lambda,\mu;\mathbf{s})\ll N_F(\mathfrak{q})^{\frac{1} {2}+200\varepsilon}.
\end{equation}
\end{lemma}
\begin{proof}
Suppose $W_v(\cdot,\lambda)=W_v^{(f_1,\sharp)}$. A modification of \eqref{5.19} and \eqref{f5.37} yields 
\begin{equation}\label{f10.22}
\big|\mathcal{I}_{v}(W_v(\cdot,\lambda);I_2)\mathcal{J}_{v}(W_v(\cdot,-\overline{\lambda}))\big|\ll N_F(\mathfrak{q})^{1+100\varepsilon}.
\end{equation}

Suppose $W_v(\cdot,\lambda)=W_v^{(f_2,\sharp)}$. 
\begin{itemize}
\item Suppose $r_{\omega_v^2}=0$. A modification of \eqref{c5.27} and \eqref{5.54} leads to  
\begin{equation}\label{f10.23}
\big|\mathcal{I}_{v}(W_v(\cdot,\lambda);I_2)\mathcal{J}_{v}(W_v(\cdot,-\overline{\lambda}))\big|\ll N_F(\mathfrak{q})^{\frac{3}{2}+100\varepsilon}. 
\end{equation}
\item Suppose $r_{\omega_v^2}=1$. A modification of \eqref{c5.26} and \eqref{5.54} leads to  
\begin{equation}\label{f10.24}
\big|\mathcal{I}_{v}(W_v(\cdot,\lambda);I_2)\mathcal{J}_{v}(W_v(\cdot,-\overline{\lambda}))\big|\ll N_F(\mathfrak{q})^{1+100\varepsilon}. 
\end{equation}
\end{itemize}

Therefore, \eqref{f10.21} follows from \eqref{f10.22}, \eqref{f10.23} and \eqref{f10.24}. 
\end{proof}

\begin{lemma}\label{lem10.5}
Let $\mathbf{s}\in \mathbf{S}_{\varepsilon}^2$ and $\lambda\in \mathfrak{S}$. Suppose $v\mid\mathfrak{q}$,  $r_{\omega_v}=1$, and $\mu_v=\overline{\omega}_v$. Then 
\begin{equation}\label{f10.25}
\boldsymbol{e}_v^*(\lambda,\mu;\mathbf{s})\ll N_F(\mathfrak{q})^{\frac{1}{2}+200\varepsilon}.
\end{equation}
\end{lemma}
\begin{proof}
Suppose $W_v(\cdot,\lambda)=W_v^{(f_1,\sharp)}$. 
\begin{itemize}
\item Suppose $r_{\omega_v^2}=0$. A modification of \eqref{5.19} and \eqref{f5.37} yields 
\begin{equation}\label{10.26}
\big|\mathcal{I}_{v}(W_v(\cdot,\lambda);I_2)\mathcal{J}_{v}(W_v(\cdot,-\overline{\lambda}))\big|\ll N_F(\mathfrak{q})^{1+100\varepsilon}.
\end{equation}

\item Suppose $r_{\omega_v^2}=1$. A modification of \eqref{5.13} and \eqref{f5.37} yields 
\begin{equation}\label{10.27}
\big|\mathcal{I}_{v}(W_v(\cdot,\lambda);I_2)\mathcal{J}_{v}(W_v(\cdot,-\overline{\lambda}))\big|\ll N_F(\mathfrak{q})^{1+100\varepsilon}.
\end{equation}
\end{itemize}

Suppose $W_v(\cdot,\lambda)=W_v^{(f_2,\sharp)}$. A modification of \eqref{c5.27} and \eqref{5.54} leads to  
\begin{equation}\label{10.23}
\big|\mathcal{I}_{v}(W_v(\cdot,\lambda);I_2)\mathcal{J}_{v}(W_v(\cdot,-\overline{\lambda}))\big|\ll N_F(\mathfrak{q})^{\frac{3}{2}+100\varepsilon}.
\end{equation}

Therefore, \eqref{f10.25} follows from \eqref{10.26}, \eqref{10.27} and \eqref{10.23}.  
\end{proof}

\subsection{Proof of Proposition \ref{prop10.1}}
By \eqref{f10.2}, the Lemmas \ref{lem10.2} and \ref{lem10.3}, and \cite[\textsection 8]{Yan25}, we obtain, for $i\in \{1, 2, 3, 4, 7, 8\}$, that  
\begin{equation}\label{10.28}
\widetilde{\Psi}_{\mathrm{Dual}}^{(i)}(\mathbf{s},R(f_{\mathfrak{q}})\boldsymbol{h})\ll \mathbf{C}_{\infty}^{\varepsilon}N_F(\mathfrak{q})^{\frac{3}{2}+\varepsilon}N_F(\mathfrak{q})^{\frac{1}{2}+\varepsilon} \ll \mathbf{C}_{\infty}^{\varepsilon}N_F(\mathfrak{q})^{2+2\varepsilon}.
\end{equation}

Let $\mathbf{L}(\lambda,\mu;\mathbf{s})$ be the finite part of $\mathbf{L}^*(\lambda,\mu;\mathbf{s})$. Suppose $\mu=\omega$. Then 
\begin{equation}\label{e10.29}
\underset{\lambda=1/2+s_2-s_4}{\Res}\ \mathbf{L}^*(\lambda,\mu;\mathbf{s})\ll N_F(\mathfrak{q})^{\frac{3}{2}+\varepsilon}.
\end{equation}

By Lemma \ref{lem10.4}, along with \eqref{e10.29} and  the local analysis in loc. cit., we obtain 
\begin{equation}\label{10.29}
\widetilde{\Psi}_{\mathrm{Dual}}^{(6)}(\mathbf{s},R(f_{\mathfrak{q}})\boldsymbol{h})\ll \mathbf{C}_{\infty}^{\varepsilon}N_F(\mathfrak{q})^{\frac{1}{2}+\varepsilon}N_F(\mathfrak{q})^{\frac{3}{2}+\varepsilon} \ll \mathbf{C}_{\infty}^{\varepsilon}N_F(\mathfrak{q})^{2+2\varepsilon}.
\end{equation}

By Lemma \ref{lem10.5}, along with the local analysis in loc. cit., we obtain 
\begin{equation}\label{10.30}
\widetilde{\Psi}_{\mathrm{Dual}}^{(5)}(\mathbf{s},R(f_{\mathfrak{q}})\boldsymbol{h})\ll \mathbf{C}_{\infty}^{\varepsilon}N_F(\mathfrak{q})^{\frac{1}{2}+\varepsilon}N_F(\mathfrak{q})^{\frac{3}{2}+\varepsilon} \ll \mathbf{C}_{\infty}^{\varepsilon}N_F(\mathfrak{q})^{2+2\varepsilon}.
\end{equation}

Therefore, Proposition \ref{prop10.1} follows from \eqref{e10.3}, \eqref{10.28}, \eqref{10.29} and \eqref{10.30}.

\section{Proof of Main Theorem \ref{thmA}}\label{sec11}
Combining Propositions \ref{prop7.1}, \ref{prop8.1}, \ref{prop9.1} and \ref{prop10.1}, we obtain  
\begin{align*}
\sum_{\substack{\pi\in \mathcal{F}_{t}^{\zeta}(\mathfrak{q}^3;\omega)\bigcup \mathcal{F}_{t}^{-\zeta}(\mathfrak{q}^3;\omega)\\
C_v(\pi)\ll C_v,\ v\mid\infty}}|L(1/2,\pi)|^4\ll_{F,\varepsilon} \mathbf{C}_{\infty}^{1+\varepsilon}N_F(\mathfrak{q})^{2+\varepsilon}.
\end{align*}
This establishes Theorem \ref{thmA}.

\bibliographystyle{alpha}

\bibliography{LY}

\end{document}